\documentclass[final,leqno]{siamltex}

\usepackage[leqno]{amsmath}
\usepackage{graphicx}
\usepackage{graphics}
\usepackage{epstopdf}
\usepackage{subfigure}
\usepackage{threeparttable}

\usepackage{bm}
\usepackage{amssymb}
\usepackage{leftidx}
\usepackage{empheq}
\usepackage{color}
\usepackage{makecell}
\usepackage{enumerate}
\allowdisplaybreaks
\renewcommand{\theequation}{\arabic{section}.\arabic{equation}}
\numberwithin{equation}{section}
\newtheorem{remark}{Remark}[section]
\title{A novel Lagrange Multiplier approach with relaxation for gradient flows.
        \thanks{
We would like to acknowledge the assistance of volunteers in putting together this example manuscript and supplement. This work is supported by National Natural Science Foundation of China (Grant Nos: 12001336, 11901489, 12131014).}}

      \author{Zhengguang Liu
             \thanks{School of Mathematics and Statistics, Shandong Normal University, Jinan, China. Email: liuzhg@sdnu.edu.cn.}
                                       \and
             Xiaoli Li\textsuperscript{*}
             \thanks{Shandong University, Jinan, Shandong, 250100, China. Email: xiaomath@sdu.edu.cn}. }
\begin{document}
\UseRawInputEncoding
\maketitle

\begin{abstract}
In this paper, we propose a novel Lagrange Multiplier approach, named zero-factor (ZF) approach to solve a series of gradient flow problems. The numerical schemes based on the new algorithm are unconditionally energy stable with the original energy and do not require any extra assumption conditions. We also prove that the ZF schemes with specific zero factors lead to the popular SAV-type method. To reduce the computation cost and improve the accuracy and consistency, we propose a zero-factor approach with relaxation, which we named the relaxed zero-factor (RZF) method, to design unconditional energy stable schemes for gradient flows. The RZF schemes can be proved to be unconditionally energy stable with respect to a modified energy that is closer to the original energy, and provide a very simple calculation process. The variation of the introduced zero factor is highly consistent with the nonlinear free energy which implies that the introduced ZF method is a very efficient way to capture the sharp dissipation of nonlinear free energy. Several numerical examples are provided to demonstrate the improved efficiency and accuracy of the proposed method.
\end{abstract}

\begin{keywords}
Lagrange Multiplier approach, Zero-factor approach, Gradient flows, Relaxation, Energy stable, Numerical examples.
\end{keywords}

    \begin{AMS}
         65M12; 35K20; 35K35; 35K55; 65Z05
    \end{AMS}

\pagestyle{myheadings}
\thispagestyle{plain}
\markboth{ZHENGGUANG LIU AND XIAOLI LI} {RELAXED ZERO-FACTOR APPROACH FOR GRADIENT FLOWS}
  \section{Introduction}
Gradient flows are a kind of important models to simulate many physical problems such as the interface behavior of multi-phase materials, the interface problems of fluid mechanics, environmental science and material mechanics. In general, as the highly complex high-order nonlinear dissipative systems, it is a great challenge to construct effective and accurate numerical schemes with physical constraints such as energy dissipation and mass conservation. Many experts and scholars considered some unconditionally energy stable schemes. These numerical schemes preserve the energy dissipation law which does not depend on the time step. Some popular and widely used methods include convex splitting approach \cite{eyre1998unconditionally,shen2012second,shin2016first}, linear stabilized approach \cite{shen2010numerical,yang2017numerical}, exponential time differencing (ETD) approach \cite{du2019maximum,du2021maximum,WangEfficient}, invariant energy quadratization (IEQ) approach \cite{gong2020arbitrarily,yang2016linear,yang2018efficient,zhao2021revisit}, scalar auxiliary variable (SAV) approach \cite{xiaoli2019energy,shen2018scalar,ShenA},  Lagrange multiplier approach \cite{cheng2020new} and so on.

Gradient flow models are generally derived from the functional variation of free energy. In general, the free energy $E(\phi)$ contains the sum of an integral phase of a nonlinear functional and a quadratic term:
\begin{equation}\label{intro-e1}
E(\phi)=\frac12(\phi,\mathcal{L}\phi)+E_1(\phi)=\frac12(\phi,\mathcal{L}\phi)+\int_\Omega F(\phi)d\textbf{x},
\end{equation}
where $\mathcal{L}$ is a symmetric non-negative linear operator, and $E_1(\phi)=\int_\Omega F(\phi)d\textbf{x}$ is nonlinear free energy. $F(\textbf{x})$ is the energy density function. The gradient flow from the energetic variation of the above energy functional $E(\phi)$ in \eqref{intro-e1} can be obtained as follows:
\begin{equation}\label{intro-e2}
\displaystyle\frac{\partial \phi}{\partial t}=-\mathcal{G}\mu,\quad\mu=\displaystyle\mathcal{L}\phi+F'(\phi),
\end{equation}
where $\mu=\frac{\delta E}{\delta \phi}$ is the chemical potential. $\mathcal{G}$ is a positive operator. For example, $\mathcal{G}=I$ for the $L^2$ gradient flow and $\mathcal{G}=-\Delta$ for the $H^{-1}$ gradient flow.

It is not difficult to find that the above phase field system satisfies the following energy dissipation law:
\begin{equation*}
\frac{d}{dt}E=(\frac{\delta E}{\delta \phi},\frac{\partial\phi}{\partial t})=-(\mathcal{G}\mu,\mu)\leq0,
\end{equation*}
which is a very important property for gradient flows in physics and mathematics.

Recently, many SAV-type methods are developed to optimize the traditional SAV method. For example, in \cite{yang2020roadmap}, the authors introduced the generalized auxiliary variable method for devising energy stable schemes for general dissipative systems. An exponential SAV approach in \cite{liu2020exponential} is developed to modify the traditional method to construct energy stable schemes by introducing an exponential SAV. In \cite{huang2020highly}, the authors consider a new SAV approach to construct high-order energy stable schemes. In \cite{cheng2020new}, the authors introduce a new Lagrange multiplier approach which is unconditionally energy stable with the original energy. However, the new approach requires solving a nonlinear algebraic equation for the Lagrange multiplier which brings some additional costs and theoretical difficulties for its analysis. Recently, Jiang et al. \cite{jiang2022improving} present a relaxation technique to construct a relaxed SAV (RSAV) approach to improve the accuracy and consistency noticeably.

In this paper, inspired by the new Lagrange multiplier approach and RSAV approach, we propose a novel technique to construct the unconditional energy stable schemes for gradient flows by introducing a zero factor. Compared with the recently proposed SAV-type approach, the numerical schemes based on the new zero-factor (ZF) method dissipate the original energy and do not require the explicitly treated part of the free energy to be bounded from below. The core idea of the zero-factor approach is to introduce a zero factor to modify the solution $\overline{\phi}^{n+1}$ of the baseline semi-implicit method at each time step. The value of the introduced zero factor $\mathcal{P}(\eta)$ is controlled by energy stability. To reduce the computation cost and improve the accuracy and consistency, we propose a zero-factor approach with relaxation, which we named the relaxed zero-factor (RZF) method, to design unconditional energy stable schemes for gradient flows. The RZF approach almost preserves all the advantages of the new zero-factor approach. It is unconditionally energy stable with respect to a modified energy that is closer to the original energy, and provides a very simple calculation process. Our main contributions of this paper are:

(i). The new introduced RZF method can keep the original energy in most cases and provides a very simple calculation process;

(ii). We prove that the zero factor schemes with specific $\mathcal{P}(\eta)$ lead to the popular SAV-type and Lagrange multiplier methods;

(iii). The variation of the introduced zero factor is highly consistent with the nonlinear free energy which implies that the introduced zero factor is very efficient to capture the sharp dissipation of the nonlinear free energy.

The paper is organized as follows. In Sect.2, we introduce a zero factor to construct a new zero-factor approach to simulate a series of gradient flows. In Sect.3, by using a relaxation technique, we propose a relaxed ZF approach. Then the second-order Crank-Nicloson and BDF2 schemes based on RZF method are constructed. In Sect.4, we briefly illustrate that the RZF approach can be easily applied to simulate the gradient flow with several disparate nonlinear terms. Finally, in Sect.5, various 2D and 3D numerical simulations are demonstrated to verify the accuracy and efficiency of our proposed schemes.

\section{The Zero-Factor Approach}
Introduce a scalar auxiliary function $\eta(t)$ to construct a linear function $\mathcal{P}(\eta)$, and rewrite the gradient flow \eqref{intro-e2} with a zero factor $\mathcal{P}(\eta)$ as follows:
\begin{equation}\label{ZF-e1}
   \begin{array}{l}
\displaystyle\frac{\partial \phi}{\partial t}=-\mathcal{G}\mu,\\
\mu=\mathcal{L}\phi+F'(\phi)+\mathcal{P}(\eta)F'(\phi),\\
\displaystyle\frac{d}{dt}\int_\Omega F(\phi)d\textbf{x}=\displaystyle\int_\Omega F'(\phi)\phi_td\textbf{x}+\mathcal{P}(\eta)\int_\Omega F'(\phi)\phi_td\textbf{x}.
   \end{array}
  \end{equation}
Here the zero factor $\mathcal{P}(\eta)$ is a linear zero function which can be chosen flexibly, such as the following $\mathcal{P}_1$ and $\mathcal{P}_2$
\begin{equation}\label{ZF-e2}
\mathcal{P}_1(\eta)=k_1\eta,\quad\mathcal{P}_2(\eta)=k_2\eta_t,
\end{equation}
where $k_1$ and $k_2$ are any non-zero constants.

Set the initial condition for $\eta(t)$ to be $\eta(0)=0$ for $\mathcal{P}_1(\eta)$ or $\eta(0)=c_0$ for $\mathcal{P}_2(\eta)$ where $c_0$ is an arbitrary constant, then it is easy to see that the new system \eqref{ZF-e1} is equivalent to the original system \eqref{intro-e2}, i.e., $\mathcal{P}(\eta)=0$ in \eqref{ZF-e1}.

Taking the inner products of the first two equations in the above equivalent system \eqref{ZF-e1} with $\mu$ and $-\phi_t$, respectively, then summing up the results together with the third equation, we obtain the original energy dissipative law:
\begin{equation*}
\frac{d}{dt}E=-(\mathcal{G}\mu,\mu)\leq0,
\end{equation*}
It means that the linear functional $\mathcal{P}(\eta)$ here is to serve as a zero factor to enforce dissipation of the original energy.

\subsection{A second-order Crank-Nicloson ZF scheme}
Before giving a detailed introduction, we let $N>0$ be a positive integer and set
\begin{equation*}
\Delta t=T/N,\quad t^n=n\Delta t,\quad \text{for}\quad n\leq N.
\end{equation*}

In the following, we will consider a second-order Crank-Nicolson scheme for the system \eqref{ZF-e1}. Discretize the nonlinear functional $F'(\phi)$ explicitly and the other items implicitly in \eqref{ZF-e1}, and give the initial values $\phi^0=\phi_0(x)$, $\eta(0)=c_0$, then couple with Crank-Nicolson formula, a second-order energy stable schemes can be constructed as follows:
\begin{equation}\label{ZF-e3}
   \begin{array}{l}
\displaystyle\frac{\phi^{n+1}-\phi^n}{\Delta t}=-\mathcal{G}\mu^{n+\frac12},\\
\displaystyle\mu^{n+\frac12}=\frac12\mathcal{L}\phi^{n+1}+\frac12\mathcal{L}\phi^{n}+F'(\widehat{\phi}^{n+\frac12})+\mathcal{P}(\eta^{n+\frac12})F'(\widehat{\phi}^{n+\frac12}),\\
\displaystyle(F(\phi^{n+1}),1)-(F(\phi^{n}),1)=\displaystyle\left(F'(\widehat{\phi}^{n+\frac12}),\phi^{n+1}-\phi^n\right)+\mathcal{P}(\eta^{n+\frac12})\left(F'(\widehat{\phi}^{n+\frac12}),\phi^{n+1}-\phi^n\right),
   \end{array}
  \end{equation}
where $\widehat{\phi}^{n+\frac12}=\frac32\phi^n-\frac12\phi^{n-1}$.

Taking the inner products of first two equation in \eqref{ZF-e3} with $\mu^{n+\frac12}$ and $-\frac{\phi^{n+1}-\phi^n}{\Delta t}$ respectively,
and multiplying the third equation with $\Delta t$, then combining these equations, we obtain the above Crank-Nicolson scheme satisfies the following original energy dissipative law:
\begin{equation}\label{ZF-e4}
\mathcal{E}(\phi^{n+1})-\mathcal{E}(\phi^{n})=-\Delta t(\mathcal{G}\mu^{n+\frac12},\mu^{n+\frac12})\leq0,
\end{equation}
where $\mathcal{E}(\phi^{n})=\frac12(\mathcal{L}\phi^{n},\phi^n)+(F(\phi^{n}),1)$.

The Crank-Nicolson scheme \eqref{ZF-e3} is nonlinear for the variables $\phi^{n+1}$ and $\eta^{n+1}$. We now show how to solve it efficiently. Combining the first two equations in \eqref{ZF-e3}, we can obtain the following linear matrix equation
\begin{equation*}
\aligned
(I+\frac12\Delta t\mathcal{G}\mathcal{L})\phi^{n+1}=(I-\frac12\Delta t\mathcal{G}\mathcal{L})\phi^{n}-\Delta t\mathcal{G}F'(\widehat{\phi}^{n+\frac12})-\mathcal{P}(\eta^{n+\frac12})\Delta t\mathcal{G}F'(\widehat{\phi}^{n+\frac12}).
\endaligned
\end{equation*}

Noting that the coefficient matrix $A=(I+\frac12\Delta t\mathcal{G}\mathcal{L})$ is a symmetric positive matrix, then we obtain
\begin{equation}\label{ZF-e5}
\aligned
\phi^{n+1}
&=A^{-1}\left[(I-\frac12\Delta t\mathcal{G}\mathcal{L})\phi^{n}-\Delta t\mathcal{G}F'(\widehat{\phi}^{n+\frac12})\right]-\mathcal{P}(\eta^{n+\frac12})\Delta tA^{-1}\mathcal{G}F'(\widehat{\phi}^{n+\frac12})\\
&=\overline{\phi}^{n+1}+\mathcal{P}(\eta^{n+\frac12})q^{n+1},
\endaligned
\end{equation}
Here $\overline{\phi}^{n+1}$ and $q^{n+1}$ can be solved directly by $\phi^n$ and $\widehat{\phi}^{n+\frac12}$ as follows:
\begin{equation}\label{ZF-e6}
\aligned
\overline{\phi}^{n+1}=A^{-1}\left[(I-\frac12\Delta t\mathcal{G}\mathcal{L})\phi^{n}-\Delta t\mathcal{G}F'(\widehat{\phi}^{n+\frac12})\right],\quad q^{n+1}=-\Delta tA^{-1}\mathcal{G}F'(\widehat{\phi}^{n+\frac12}).
\endaligned
\end{equation}
Combining the equation \eqref{ZF-e5} with the third equation in \eqref{ZF-e3}, we have
\begin{equation}\label{ZF-e7}
\aligned
&\left(F(\overline{\phi}^{n+1}+\mathcal{P}(\eta^{n+\frac12})q^{n+1}),1\right)-\left(F(\phi^{n}),1\right)\\
&=\displaystyle
\left[1+\mathcal{P}(\eta^{n+\frac12})\right]\left({F'}(\widehat{\phi}^{n+\frac12}),p^{n+1}+\mathcal{P}(\eta^{n+\frac12})q^{n+1}-\phi^n\right).
\endaligned
\end{equation}
One can see that to solve above nonlinear numerical scheme \eqref{ZF-e7}, we need to solve $\eta^{n+1}$ by the Newton iteration as the initial condition. The computational complexity depends on $F(\phi)$. The computational cost is equal to the Lagrange Multiplier approach which was proposed by Shen et al. \cite{cheng2020new}.\\
\begin{remark}\label{ZF-le1}
In principle one can choose any linear function to be zero factor $\mathcal{P}(\eta)$ in equation \eqref{ZF-e1}. A special case is $\mathcal{P}(\eta)=\eta(t)-1$, then the zero factor method leads to the new Lagrange multiplier approach in \cite{cheng2020new}.
\end{remark}\\
\begin{remark}\label{ZF-le2}
From the equation \eqref{ZF-e5}, we notice that $\overline{\phi}^{n+1}$ is the solution of the baseline semi-implicit Crank-Nicolson scheme. Hence the core idea of the zero factor approach is to introduce a zero factor to modify the solution $\overline{\phi}^{n+1}$ at each time step. The value of the zero factor is controlled by energy stability.
\end{remark}
\subsection{A revisit of the SAV-type approach}
In this subsection, we will review the SAV-type approach and prove that the introduced scalar auxiliary variables can be seen as the specific zero factors. Furthermore, we can modify the SAV-type methods to construct new schemes which dissipate the original energy.

The key for the SAV approach is to introduce a scalar variable $r(t)=\sqrt{E_(\phi)+C}$ where $E_1(\phi)=(F(\phi),1)$ is the nonlinear free energy and rewrite the gradient flows \eqref{intro-e2} as the following equivalent system:
\begin{equation}\label{SAV-e1}
   \begin{array}{l}
\displaystyle\frac{\partial \phi}{\partial t}=-\mathcal{G}\mu,\\
\displaystyle\mu=\mathcal{L}\phi+\frac{r(t)}{\sqrt{E_1(\phi)+C}}F'(\phi),\\
\displaystyle\frac{dr}{dt}=\frac{1}{2\sqrt{E_1(\phi)+C}}({F'}(\phi),\phi_t).
   \end{array}
\end{equation}
A second-order Crank-Nicloson SAV scheme for above equivalent system is as follows:
\begin{equation}\label{SAV-e2}
   \begin{array}{l}
\displaystyle\frac{\phi^{n+1}-\phi^n}{\Delta t}=-\mathcal{G}\mu^{n+\frac12},\\
\displaystyle\mu^{n+\frac12}=\frac12\mathcal{L}\phi^{n+1}+\frac12\mathcal{L}\phi^{n}+\frac{r^{n+\frac12}}{\sqrt{E_1(\widehat{\phi}^{n+\frac12})+C}}F'(\widehat{\phi}^{n+\frac12}),\\
\displaystyle\frac{r^{n+1}-r^n}{\Delta t}=\frac{1}{2\sqrt{E_1(\widehat{\phi}^{n+\frac12})+C}}({F'}(\widehat{\phi}^{n+\frac12}),\frac{\phi^{n+1}-\phi^n}{\Delta t}).
   \end{array}
\end{equation}
Here $r^{n+\frac12}=(r^{n+1}+r^n)/2$.

Combining the first two equations in above second-order scheme, we can obtain:
\begin{equation}\label{SAV-e3}
\aligned
(I+\frac12\Delta t\mathcal{G}\mathcal{L})\phi^{n+1}
&=(I-\frac12\Delta t\mathcal{G}\mathcal{L})\phi^n-\Delta t\frac{r^{n+\frac12}}{\sqrt{E_1(\widehat{\phi}^{n+\frac12})+C}}\mathcal{G}F'(\widehat{\phi}^{n+\frac12})\\
&=\left[(I-\frac12\Delta t\mathcal{G}\mathcal{L})-\Delta t\mathcal{G}F'(\widehat{\phi}^{n+\frac12})\right]-\left(\frac{r^{n+\frac12}}{\sqrt{E_1(\widehat{\phi}^{n+\frac12})+C}}-1\right)\Delta t\mathcal{G}F'(\widehat{\phi}^{n+\frac12}).
\endaligned
\end{equation}
Using the same definitions of $\overline{\phi}^{n+1}$ and $q^{n+1}$ in \eqref{ZF-e6}, we can obtain $\phi^{n+1}$ as follows:
\begin{equation}\label{SAV-e4}
\aligned
\phi^{n+1}
&=A^{-1}\left[(I-\frac12\Delta t\mathcal{G}\mathcal{L})\phi^{n}-\Delta t\mathcal{G}F'(\widehat{\phi}^{n+\frac12})\right]-\left(\frac{r^{n+\frac12}}{\sqrt{E_1(\widehat{\phi}^{n+\frac12})+C}}-1\right)\Delta tA^{-1}\mathcal{G}F'(\widehat{\phi}^{n+\frac12})\\
&=\overline{\phi}^{n+1}+\left(\frac{r^{n+\frac12}}{\sqrt{E_1(\widehat{\phi}^{n+\frac12})+C}}-1\right)q^{n+1}.
\endaligned
\end{equation}
Compared above equation \eqref{SAV-e4} with \eqref{ZF-e6}, we can obviously obtain that the key for the SAV approach is to introduce a zero factor
\begin{equation}\label{SAV-e5}
\aligned
P(r)=\frac{r(t)}{\sqrt{E_1(\phi)+C}}-1.
\endaligned
\end{equation}
It means that the core idea of the SAV scheme \eqref{SAV-e2} is also to introduce a special zero factor to modify the solution $\overline{\phi}^{n+1}$ which is the solution of the baseline semi-implicit Crank-Nicolson scheme at each time step. The value of the zero factor $P(r)$ is controlled by energy stability.

Inspired by the introduced ZF method, we can obtain a new SAV approach which is unconditionally energy stable with the original energy by changing the third equation in the equivalent system \eqref{SAV-e1}:
\begin{equation}\label{SAV-e6}
   \begin{array}{l}
\displaystyle\frac{\partial \phi}{\partial t}=-\mathcal{G}\mu,\\
\displaystyle\mu=\mathcal{L}\phi+\frac{r(t)}{\sqrt{E_1(\phi)+C}}F'(\phi),\\
\displaystyle\frac{d}{dt}\int_\Omega F(\phi)d\textbf{x}=\displaystyle\frac{r(t)}{\sqrt{E_1(\phi)+C}}\int_\Omega F'(\phi)\phi_td\textbf{x}.
   \end{array}
\end{equation}
A second-order Crank-Nicloson SAV scheme for above equivalent system \eqref{SAV-e6} is as follows:
\begin{equation}\label{SAV-e7}
   \begin{array}{l}
\displaystyle\frac{\phi^{n+1}-\phi^n}{\Delta t}=-\mathcal{G}\mu^{n+\frac12},\\
\displaystyle\mu^{n+\frac12}=\frac12\mathcal{L}\phi^{n+1}+\frac12\mathcal{L}\phi^{n}+\frac{r^{n+\frac12}}{\sqrt{E_1(\widehat{\phi}^{n+\frac12})+C}}F'(\widehat{\phi}^{n+\frac12}),\\
\displaystyle(F(\phi^{n+1}),1)-(F(\phi^{n}),1)=\frac{r^{n+\frac12}}{\sqrt{E_1(\widehat{\phi}^{n+\frac12})+C}}({F'}(\widehat{\phi}^{n+\frac12}),\phi^{n+1}-\phi^n).
   \end{array}
\end{equation}
Taking the inner products of first two equation in \eqref{SAV-e7} with $\mu^{n+\frac12}$ and $-\frac{\phi^{n+1}-\phi^n}{\Delta t}$ respectively,
and multiplying the third equation with $\Delta t$, then combining these equations, we obtain the above Crank-Nicolson scheme satisfies the following original energy dissipative law:
\begin{equation}\label{SAV-e8}
\mathcal{E}(\phi^{n+1})-\mathcal{E}(\phi^{n})=-\Delta t(\mathcal{G}\mu^{n+\frac12},\mu^{n+\frac12})\leq0,
\end{equation}
where $\mathcal{E}(\phi^{n})=\frac12(\mathcal{L}\phi^{n},\phi^n)+(F(\phi^{n}),1)$.\\
\begin{remark}\label{SAV-le1}
For other SAV-type approaches, the core idea is also to introduce a special zero factor to modify the solution $\overline{\phi}^{n+1}$. For example, the zero factor $\mathcal{P}(r)=\frac{r(t)}{\exp(E_1(\phi))}-1$ for ESAV approach in \cite{liu2020exponential}.
\end{remark}
\section{The Relaxed Zero-Factor Approach}
From above analysis, we notice that the scheme based on the zero-factor approach dissipates the original energy but needs to solve a nonlinear algebraic equation which brings some additional costs and theoretical difficulties for its analysis. In general, compared with the baseline SAV scheme, the new algorithm brings some additional costs because it requires solving a nonlinear algebraic equation for $\eta^{n+1}$. To reduce the computation cost and improve the efficiency, inspired by the R-SAV approach described in \cite{jiang2022improving}, we consider a zero-factor approach with relaxation, which we named the relaxed zero-factor (RZF) method, to design unconditional energy stable schemes for gradient flows. It can be proved that the RZF approach not only determines $\eta^{n+1}$ explicitly, but also dissipates an almost original energy.

Firstly, we introduce a new scalar auxiliary function $R(t)=(F(\phi),1)$, and rewrite the equivalent gradient flow \eqref{ZF-e1} as follows:
\begin{equation}\label{RZF-e1}
   \begin{array}{l}
\displaystyle\frac{\partial \phi}{\partial t}=-\mathcal{G}\mu,\\
\mu=\mathcal{L}\phi+F'(\phi)+\mathcal{P}(\eta)F'(\phi),\\
\displaystyle\frac{dR}{dt}=\displaystyle\int_\Omega F'(\phi)\phi_td\textbf{x}+\mathcal{P}(\eta)\int_\Omega F'(\phi)\phi_td\textbf{x},\\
R(t)=(F(\phi),1).
   \end{array}
\end{equation}
Taking the inner products of the first two equations in the above equivalent system \eqref{RZF-e1} with $\mu$ and $-\phi_t$, respectively, then summing up the results together with the third and fourth equations, we obtain the original energy dissipative law:
\begin{equation}\label{RZF-e2}
\frac{d}{dt}E=\frac12(\phi,\mathcal{L}\phi)+\int_\Omega F(\phi)d\textbf{x}=\frac12(\phi,\mathcal{L}\phi)+R(t)=-(\mathcal{G}\mu,\mu)\leq0,
\end{equation}

For the equivalent system \eqref{RZF-e1}, a Crank-Nicolson scheme based on above zero-factor (ZF-CN) approach can be given as follows:
\begin{equation}\label{ZF-CN-e1}
   \begin{array}{l}
\displaystyle\frac{\phi^{n+1}-\phi^n}{\Delta t}=-\mathcal{G}\mu^{n+\frac12},\\
\displaystyle\mu^{n+\frac12}=\frac12\mathcal{L}\phi^{n+1}+\frac12\mathcal{L}\phi^{n}+F'(\widehat{\phi}^{n+\frac12})+\mathcal{P}(\eta^{n+\frac12})F'(\widehat{\phi}^{n+\frac12}),\\
\displaystyle R^{n+1}-R^n=\displaystyle\left[1+\mathcal{P}(\eta^{n+\frac12})\right]\left(F'(\widehat{\phi}^{n+\frac12}),\phi^{n+1}-\phi^n\right),\\
\displaystyle R^{n+1}=\left(F(\phi^{n+1}),1\right).
   \end{array}
\end{equation}
Similar as \eqref{ZF-e6}, introduce $\overline{\phi}^{n+1}$ and $q^{n+1}$ as follows:
\begin{equation}\label{ZF-CN-e2}
\aligned
\overline{\phi}^{n+1}=A^{-1}\left[(I-\frac12\Delta t\mathcal{G}\mathcal{L})\phi^{n}-\Delta t\mathcal{G}F'(\widehat{\phi}^{n+\frac12})\right],\quad q^{n+1}=-\Delta tA^{-1}\mathcal{G}F'(\widehat{\phi}^{n+\frac12}),
\endaligned
\end{equation}
where $A$ is the coefficient matrix to satisfy $A=(I+\frac12\Delta t\mathcal{G}\mathcal{L})$. The baseline ZF-CN scheme \eqref{ZF-CN-e1} can be rewrite as follows:
\begin{equation}\label{ZF-CN-e3}
   \begin{array}{l}
\displaystyle\frac{\overline{\phi}^{n+1}-\phi^n}{\Delta t}=-\mathcal{G}\mu^{n+\frac12},\\
\displaystyle\mu^{n+\frac12}=\frac12\mathcal{L}\overline{\phi}^{n+1}+\frac12\mathcal{L}\phi^{n}+F'(\widehat{\phi}^{n+\frac12}),\\
\displaystyle\phi^{n+1}=\overline{\phi}^{n+1}+\mathcal{P}(\eta^{n+\frac12})q^{n+1},\\
\displaystyle R^{n+1}-R^n=\displaystyle\left[1+\mathcal{P}(\eta^{n+\frac12})\right]\left(F'(\widehat{\phi}^{n+\frac12}),\phi^{n+1}-\phi^n\right),\\
\displaystyle R^{n+1}=\left(F(\phi^{n+1}),1\right).
   \end{array}
\end{equation}
It is not difficult to obtain that the above ZF-CN scheme \eqref{ZF-CN-e3} is nonlinear for the variables $\phi^{n+1}$ and $\eta^{n+1}$.
\subsection{A Second-order RZF-CN Scheme}
To reduce the computation cost and improve the efficiency, we consider the following second order Crank-Nicolson scheme based on the relaxed zero-factor approach (RZF-CN): set $R^0=(F(\phi(0,\textbf{x}),1)$, and compute $\phi^{n+1}$, $R^{n+1}$ via the following two steps:

\textbf{Step I}: Compute $\phi^{n+1}$ and $\widetilde{R}^{n+1}$ by the following semi-implicit Crank-Nicolson scheme:
\begin{equation}\label{RZF-CN-e1}
   \begin{array}{l}
\displaystyle\frac{\overline{\phi}^{n+1}-\phi^n}{\Delta t}=-\mathcal{G}\mu^{n+\frac12},\\
\displaystyle\mu^{n+\frac12}=\frac12\mathcal{L}\overline{\phi}^{n+1}+\frac12\mathcal{L}\phi^{n}+F'(\widehat{\phi}^{n+\frac12}),\\
\displaystyle\phi^{n+1}=\overline{\phi}^{n+1}+\mathcal{P}(\eta^{n+\frac12})q^{n+1},\\
\displaystyle \widetilde{R}^{n+1}-R^n=\displaystyle\left[1+\mathcal{P}(\eta^{n+\frac12})\right]\left(F'(\widehat{\phi}^{n+\frac12}),\phi^{n+1}-\phi^n\right),\\
\displaystyle \widetilde{R}^{n+1}=\left(F(\overline{\phi}^{n+1}),1\right).
   \end{array}
\end{equation}
where $\widehat{\phi}^{n+\frac12}=\frac32\phi^n-\frac12\phi^{n-1}$.

\textbf{Step II}: Update the scalar auxiliary variable $R^{n+1}$ via a relaxation step as
\begin{equation}\label{RZF-CN-e2}
R^{n+1}=\lambda_0\widetilde{R}^{n+1}+(1-\lambda_0)\left(F(\phi^{n+1}),1\right),\quad \lambda_0\in\mathcal{V}.
\end{equation}
Here is $\mathcal{V}$ a set defined by
\begin{equation}\label{RZF-CN-e3}
\mathcal{V}=\left\{\lambda|\lambda\in[0,1]~s.t.~R^{n+1}-\widetilde{R}^{n+1}\leq\Delta t\kappa^{n+1}\left(\mathcal{G}\mu^{n+\frac12},\mu^{n+\frac12}\right),\ R^{n+1}=\lambda\widetilde{R}^{n+1}+(1-\lambda)\left(F(\phi^{n+1}),1\right)\right\}.
\end{equation}
Here, $\kappa^{n+1}\in[0,1]$ will be given below.

We first show how to solve the scheme \eqref{RZF-CN-e1}. Substituting the third equation in \eqref{RZF-CN-e1} into the fourth equation to obtain
\begin{equation}\label{RZF-CN-e4}
\displaystyle\widetilde{R}^{n+1}-R^n=\displaystyle\left[1+\mathcal{P}(\eta^{n+\frac12})\right]\left(F'(\widehat{\phi}^{n+\frac12}),\overline{\phi}^{n+1}+\mathcal{P}(\eta^{n+\frac12})q^{n+1}-\phi^n\right),
\end{equation}
Note that $\mathcal{P}(\eta)$ is a linear functional of $\eta$, hence, the above equation is a quadratic equation with one unknown for $\eta^{n+\frac12}$. It means we can obtain determine $\eta^{n+\frac12}$ explicitly from \eqref{RZF-CN-e4}, namely,
\begin{equation}\label{RZF-CN-e5}
\displaystyle\eta^{n+\frac12}=\frac{-b\pm\sqrt{b^2-4ac}}{2a}.
\end{equation}
Here, if we set $\mathcal{P}(\eta)=\mathcal{P}_1(\eta)=k_1\eta$, then we have $\mathcal{P}(\eta^{n+\frac12})=k_1\eta^{n+\frac12}$. The coefficients $a$, $b$ and $c$ of above quadratic equation \eqref{RZF-CN-e5} will satisfy:
\begin{equation*}
\aligned
&a=k_1^2\left({F'}(\widehat{\phi}^{n+\frac12}),q^{n+1}\right),\\
&b=k_1\left({F'}(\widehat{\phi}^{n+\frac12}),\overline{\phi}^{n+1}-\phi^n\right)+\left({F'}(\widehat{\phi}^{n+\frac12}),q^{n+1}\right),\\
&c=-\widetilde{R}^{n+1}+R^n+\left({F'}(\widehat{\phi}^{n+\frac12}),\overline{\phi}^{n+1}-\phi^n\right).
\endaligned
\end{equation*}
If we set $\mathcal{P}(\eta)=\mathcal{P}_2(\eta)=k_2\eta_t$, then we have $\mathcal{P}(\eta^{n+\frac12})=k_2\frac{\eta^{n+1}-\eta^n}{\Delta t}$. The coefficients $a$, $b$ and $c$ of above quadratic equation \eqref{RZF-CN-e5} will satisfy:
\begin{equation*}
\aligned
&a=\frac{k_2^2}{\Delta t^2}\left({F'}(\widehat{\phi}^{n+\frac12}),q^{n+1}\right),\\
&b=\frac{k_2}{\Delta t}\left({F'}(\widehat{\phi}^{n+\frac12}),\overline{\phi}^{n+1}-\phi^n-\frac{k_2}{\Delta t}\eta^nq^{n+1}\right)+\frac{k_2}{\Delta t}(1-\frac{k_2}{\Delta t}\eta^n)\left({F'}(\widehat{\phi}^{n+\frac12}),q^{n+1}\right),\\
&c=-\widetilde{R}^{n+1}+R^n+(1-\frac{k_2}{\Delta t}\eta^n)\left({F'}(\widehat{\phi}^{n+\frac12}),\overline{\phi}^{n+1}-\phi^n-\frac{k_2}{\Delta t}\eta^nq^{n+1}\right).
\endaligned
\end{equation*}\\
\begin{remark}
The left side of the equation \eqref{RZF-CN-e4} is an approximation of the free energy difference between two adjacent time steps. If this value is less than the round-off error of numerical integration, it might cause the zero factor $\mathcal{P}(\eta^{n+\frac12})$ tends to be -1. To avoid this mistake, if $|\widetilde{R}^{n+1}-R^n|<1e(-15)$ or $\mathcal{P}(\eta^{n+\frac12})\rightarrow-1$, we can solve the following equation by the Newton iteration to update $\mathcal{P}(\eta^{n+\frac12})$:
\begin{equation*}
\displaystyle\mathcal{P}(\eta^{n+\frac12})\left(\widetilde{R}^{n+1}-R^n\right)=\displaystyle\mathcal{P}(\eta^{n+\frac12})\left[1+\mathcal{P}(\eta^{n+\frac12})\right]\left(F'(\widehat{\phi}^{n+\frac12}),\overline{\phi}^{n+1}+\mathcal{P}(\eta^{n+\frac12})q^{n+1}-\phi^n\right),
\end{equation*}
The similar technique can be seen in \cite{lin2019numerical}.
\end{remark}

Next, we will show that how to obtain the optimal choice for the relaxation parameter $\lambda_0$. The set $\mathcal{V}$ in \eqref{RZF-CN-e3} can be simplified as
\begin{equation}\label{RZF-CN-e6}
\mathcal{V}=\left\{\lambda|\lambda\in[0,1] ~s.t.~\left[\widetilde{R}^{n+1}-\left(F(\phi^{n+1}),1\right)\right]\lambda\leq\left[\widetilde{R}^{n+1}-\left(F(\phi^{n+1}),1\right)\right]+
\Delta t\kappa^{n+1}\left(\mathcal{G}\mu^{n+\frac12},\mu^{n+\frac12}\right)\right\}.
\end{equation}

Noting the fact $\Delta t\kappa^{n+1}(\mathcal{G}\overline{\mu}^{n+\frac12},\overline{\mu}^{n+\frac12})\geq0$, then we obtain $1\in\mathcal{V}$ which means the set $\mathcal{V}$ is non-empty. We can choose the optimal relaxation parameter $\lambda_0$ as follows: the optimal $\lambda_0$ can be chosen as a solution of the following optimization problem:
\begin{equation}\label{RZF-CN-e7}
\aligned
\lambda_0=\min\limits_{\lambda\in[0,1]}\lambda\quad s.t.~\left[\widetilde{R}^{n+1}-\left(F(\phi^{n+1}),1\right)\right]\lambda\leq\left[\widetilde{R}^{n+1}-\left(F(\phi^{n+1}),1\right)\right]+
\Delta t\kappa^{n+1}\left(\mathcal{G}\mu^{n+\frac12},\mu^{n+\frac12}\right).
\endaligned
\end{equation}
The next theorem summarizes the choice of $\lambda_0$ and $\kappa^{n+1}$:
\begin{theorem}\label{RZF-CN-th1}
If $\widetilde{R}^{n+1}-\left(F(\phi^{n+1}),1\right)\neq0$, setting $\alpha=\frac{
\Delta t\left(\mathcal{G}\mu^{n+\frac12},\mu^{n+\frac12}\right)}{\left|\widetilde{R}^{n+1}-\left(F(\phi^{n+1}),1\right)\right|}$, then we can choose the optimal relaxation parameter $\lambda_0$ and $\kappa^{n+1}$ as follows:
\begin{enumerate}
\item[1.] If $\widetilde{R}^{n+1}\geq\left(F(\phi^{n+1}),1\right)$, we set $\lambda_0=0$ and $\kappa^{n+1}=0$;
\item[2.] If $\widetilde{R}^{n+1}<\left(F(\phi^{n+1}),1\right)$ and $\alpha\geq1$, we set $\lambda_0=0$ and $\kappa^{n+1}=\frac1\alpha$;
\item[3.] If $\widetilde{R}^{n+1}<\left(F(\phi^{n+1}),1\right)$ and $\alpha\in[0,1)$, we set $\lambda_0=1-\alpha$ and $\kappa^{n+1}=1$.
\end{enumerate}
\end{theorem}
\begin{proof}
(1) if $\widetilde{R}^{n+1}=\left(F(\phi^{n+1}),1\right)$, any arbitrary parameter $\lambda$ between $0$ and $1$ will satisfy the inequality in \eqref{RZF-CN-e7}. Thus, we have $\lambda_0=\min\limits_{\lambda\in[0,1]}\lambda=0$.

(2) if $\widetilde{R}^{n+1}>\left(F(\phi^{n+1}),1\right)$, the inequality in \eqref{RZF-CN-e7} will be simplified as
\begin{equation*}
\aligned
\lambda\leq1+\kappa^{n+1}\alpha.
\endaligned
\end{equation*}
Set $\kappa^{n+1}=0$, then $\lambda\leq1$ is always true. Thus, we also have $\lambda_0=\min\limits_{\lambda\in[0,1]}\lambda=0$.

(3) if $\widetilde{R}^{n+1}<\left(F(\phi^{n+1}),1\right)$, the inequality in \eqref{RZF-CN-e7} will be simplified as
\begin{equation*}
\aligned
\lambda\geq1-\kappa^{n+1}\alpha.
\endaligned
\end{equation*}
Firstly, if $\alpha\geq1$, we have $\frac1\alpha\in(0,1]$. Then we set $\kappa^{n+1}=\frac1\alpha$ to obtain that $\lambda\geq1-\kappa^{n+1}\alpha=0$ is always true. It means $\lambda_0=\min\limits_{\lambda\in[0,1]}\lambda=0$. Secondly, if $\alpha\in[0,1)$, we have $(1-\kappa^{n+1}\alpha)\in(0,1]$ for any $\kappa^{n+1}$. Then we obtain $\lambda_0=\min\limits_{\lambda\in[0,1]}\lambda=1-\kappa^{n+1}\alpha$. By setting $\kappa^{n+1}=1$, we obtain the optimal solution $\lambda_0=1-\alpha$.
\end{proof}

Next, the following theorem will shown that the above RZF-CN scheme \eqref{RZF-CN-e1}-\eqref{RZF-CN-e2} is unconditional energy stable with a modified energy that is directly linked to the original free energy.\\
\begin{theorem}\label{RZF-CN-th2}
The second-order Crank-Nicolson scheme \eqref{RZF-CN-e1}-\eqref{RZF-CN-e2} based on the RZF approach with the above choice of $\lambda_0$ and $\kappa^{n+1}$ is unconditionally energy stable in the sense that
\begin{equation}\label{RZF-CN-e8}
\mathcal{\widetilde{E}}(\phi^{n+1})-\mathcal{\widetilde{E}}(\phi^{n})\leq-\Delta t(1-\kappa^{n+1})(\mathcal{G}\mu^{n+\frac12},\mu^{n+\frac12})\leq0,
\end{equation}
where $\mathcal{\widetilde{E}}(\phi^{n+1})=\frac12(\mathcal{L}\phi^{n+1},\phi^{n+1})+R^{n+1}$ and more importantly we have
\begin{equation*}
\aligned
\mathcal{E}(\phi^{n+1})\leq\mathcal{E}(\phi^{n}),
\endaligned
\end{equation*}
under the condition of $\widetilde{R}^{n+1}\geq\left(F(\phi^{n+1}),1\right)$ or $\widetilde{R}^{n+1}<\left(F(\phi^{n+1}),1\right)$ with $\alpha\geq1$. Here $\mathcal{E}(\phi^{n})$ is the original energy where $\mathcal{E}(\phi^{n+1})=\frac12(\mathcal{L}\phi^{n+1},\phi^{n+1})+(F(\phi^{n+1}),1)$. If $\widetilde{R}^{n+1}<\left(F(\phi^{n+1}),1\right)$ and $\alpha\in[0,1)$, we could have
\begin{equation*}
\aligned
\mathcal{\widetilde{E}}(\phi^{n+1})\leq\mathcal{E}(\phi^{n+1}).
\endaligned
\end{equation*}
\end{theorem}
\begin{proof}
The first three equations in the \textbf{Step I} of the RZF-CN scheme \eqref{RZF-CN-e1} can be rewrite as follows:
\begin{equation}\label{RZF-CN-e9}
   \begin{array}{l}
\displaystyle\frac{\phi^{n+1}-\phi^n}{\Delta t}=-\mathcal{G}\mu^{n+\frac12},\\
\displaystyle\mu^{n+\frac12}=\frac12\mathcal{L}\phi^{n+1}+\frac12\mathcal{L}\phi^{n}+\left[1+\mathcal{P}(\eta^{n+\frac12})\right]F'(\widehat{\phi}^{n+\frac12}),\\
   \end{array}
\end{equation}
Taking the inner products of \eqref{RZF-CN-e9} with $\mu^{n+\frac12}$ and $-\frac{\phi^{n+1}-\phi^n}{\Delta t}$ respectively,
and multiplying the fourth equation with $\Delta t$ in \eqref{RZF-CN-e1}, then combining these equations, we obtain immediately
\begin{equation}\label{RZF-CN-e10}
\aligned
\displaystyle \left[\frac12(\mathcal{L}\phi^{n+1},\phi^{n+1})+\widetilde{R}^{n+1}\right]-\left[\frac12(\mathcal{L}\phi^{n},\phi^{n})+R^{n}\right]=-\Delta t(\mathcal{G}\mu^{n+\frac12},\mu^{n+\frac12})\leq0.
\endaligned
\end{equation}
From the constraint condition in \eqref{RZF-CN-e3}, we could obtain
\begin{equation}\label{RZF-CN-e11}
R^{n+1}-\widetilde{R}^{n+1}\leq\Delta t\kappa^{n+1}\left(\mathcal{G}\mu^{n+\frac12},\mu^{n+\frac12}\right).
\end{equation}
Substituting the inequality \eqref{RZF-CN-e11} into \eqref{RZF-CN-e10} and noting $\kappa^{n+1}\in[0,1]$, we could have
\begin{equation}\label{RZF-CN-e12}
\aligned
\mathcal{\widetilde{E}}(\phi^{n+1})-\mathcal{\widetilde{E}}(\phi^{n})
&=\left[\frac12(\mathcal{L}\phi^{n+1},\phi^{n+1})+R^{n+1}\right]-\left[\frac12(\mathcal{L}\phi^{n},\phi^{n})+R^{n}\right]\\
&\leq\left[\frac12(\mathcal{L}\phi^{n+1},\phi^{n+1})+\widetilde{R}^{n+1}\right]-\left[\frac12(\mathcal{L}\phi^{n},\phi^{n})+R^{n}\right]+\Delta t\kappa^{n+1}(\mathcal{G}\mu^{n+\frac12},\mu^{n+\frac12})\\
&=-\Delta t(1-\kappa^{n+1})(\mathcal{G}\mu^{n+\frac12},\mu^{n+\frac12})\leq0.
\endaligned
\end{equation}
Noting that the original energy $\mathcal{E}(\phi^{n+1})=\frac12(\mathcal{L}\phi^{n+1},\phi^{n+1})+(F(\phi^{n+1}),1)$ and using the \textbf{Step II} of the RZF-CN scheme \eqref{RZF-CN-e2}, we could get
\begin{equation}\label{RZF-CN-e13}
\aligned
\mathcal{\widetilde{E}}(\phi^{n+1})-\mathcal{E}(\phi^{n+1})
&=R^{n+1}-(F(\phi^{n+1}),1)\\
&=\lambda_0\widetilde{R}^{n+1}+(1-\lambda_0)\left(F(\phi^{n+1}),1\right)-(F(\phi^{n+1}),1)\\
&=\lambda_0\left[\widetilde{R}^{n+1}-(F(\phi^{n+1}),1)\right].
\endaligned
\end{equation}
From the choice of $\lambda_0$ in Theorem \ref{RZF-CN-th1}, if $\widetilde{R}^{n+1}\geq\left(F(\phi^{n+1}),1\right)$ or $\widetilde{R}^{n+1}<\left(F(\phi^{n+1}),1\right)$ and $\alpha\geq1$, we have $\lambda_0=0$. It means $\mathcal{\widetilde{E}}(\phi^{n+1})-\mathcal{E}(\phi^{n+1})=\lambda_0\left[\widetilde{R}^{n+1}-(F(\phi^{n+1}),1)\right]=0$. If $\widetilde{R}^{n+1}<\left(F(\phi^{n+1}),1\right)$ and $\alpha\in[0,1)$, we have $\lambda_0\left[\widetilde{R}^{n+1}-(F(\phi^{n+1}),1)\right]\leq0$, then the following inequality will hold:
\begin{equation*}
\aligned
\mathcal{\widetilde{E}}(\phi^{n+1})\leq\mathcal{E}(\phi^{n+1}).
\endaligned
\end{equation*}
From above analysis, we can obtain that
\begin{equation*}
\aligned
\mathcal{\widetilde{E}}(\phi^{n})\leq\mathcal{E}(\phi^{n}),\quad \forall n\geq0.
\endaligned
\end{equation*}
Combining above inequality with \eqref{RZF-CN-e9}, when $\widetilde{R}^{n+1}\geq\left(F(\phi^{n+1}),1\right)$ or $\widetilde{R}^{n+1}<\left(F(\phi^{n+1}),1\right)$ and $\alpha\geq1$, we could have
\begin{equation}\label{RZF-CN-e14}
\aligned
\mathcal{E}(\phi^{n+1})=\mathcal{\widetilde{E}}(\phi^{n+1})\leq\mathcal{\widetilde{E}}(\phi^{n})\leq\mathcal{E}(\phi^{n}).
\endaligned
\end{equation}
\end{proof}\\
\begin{remark}\label{RZF-CN-re1}
From Theorem \ref{RZF-CN-th2}, we observe that in most cases, we have $\mathcal{E}(\phi^{n+1})\leq\mathcal{E}(\phi^{n})$ which means the RZF-CN scheme \eqref{RZF-CN-e1}-\eqref{RZF-CN-e2} dissipates the original energy. Only if $\widetilde{R}^{n+1}<\left(F(\phi^{n+1}),1\right)$ and $\alpha\in[0,1)$, we can not obtain the original energy dissipative law. Hence, by a relaxation technique, the original energy is proved to be dissipative in most situations, which is a significant improvement over the modified ZF-CN scheme \eqref{RZF-CN-e1}. More importantly, the new proposed RZF-CN scheme \eqref{RZF-CN-e1}-\eqref{RZF-CN-e2} keeps the advantage of the scheme \eqref{RZF-CN-e1} in calculation.
\end{remark}
%
%
\subsection{A Second-order RZF-BDF2 Scheme}
In this subsection, we consider a second-order RZF scheme based on 2-step backward difference formula (BDF2). For gradient flow models, it is usually better to use BDF schemes. The second-order RZF-BDF2 scheme for the equivalent system \eqref{RZF-e1} is as follows: given $R^0=(F(\phi(0,\textbf{x}),1)$, $R^{n-1}$, $R^n$ $\phi^{n-1}$, $\phi^n$, we can update $\phi^{n+1}$ via the following two steps:

\textbf{Step I}: Compute $\phi^{n+1}$ and $\widetilde{R}^{n+1}$ by the following second-order semi-implicit BDF2 scheme:
\begin{equation}\label{RZF-BDF-e1}
   \begin{array}{l}
\displaystyle\frac{3\phi^{n+1}-4\phi^n+\phi^{n-1}}{2\Delta t}=-\mathcal{G}\mu^{n+1},\\
\displaystyle\mu^{n+1}=\mathcal{L}\phi^{n+1}+F'(\widehat{\phi}^{n+1})+\mathcal{P}(\eta^{n+1})F'(\widehat{\phi}^{n+1}),\\
\displaystyle3\widetilde{R}^{n+1}-4R^n+R^{n-1}=\displaystyle\left[1+\mathcal{P}(\eta^{n+1})\right]\left(F'(\widehat{\phi}^{n+1}),3\phi^{n+1}-4\phi^n+\phi^{n-1}\right),\\
\displaystyle\widetilde{R}^{n+1}=\left(F(\overline{\phi}^{n+1}),1\right),
   \end{array}
\end{equation}
where $\widehat{\phi}^{n+1}=2\phi^n-\phi^{n-1}$.

\textbf{Step II}: Update the scalar auxiliary variable $R^{n+1}$ via a relaxation step as
\begin{equation}\label{RZF-BDF-e2}
R^{n+1}=\lambda_0\widetilde{R}^{n+1}+(1-\lambda_0)\left(F(\phi^{n+1}),1\right),\quad \lambda_0\in\mathcal{V}.
\end{equation}
Here is $\mathcal{V}$ a set defined by
\begin{equation}\label{RZF-BDF-e3}
\aligned
\mathcal{V}=\left\{\lambda|\lambda\in[0,1]~s.t.~\right. &R^{n+1}-\widetilde{R}^{n+1}\leq\Delta t\kappa^{n+1}\left(\mathcal{G}\mu^{n+1},\mu^{n+1}\right)\\
&+\frac14\kappa^{n+1}\left(\mathcal{L}(\phi^{n+1}-2\phi^n+\phi^{n-1}),\phi^{n+1}-2\phi^n+\phi^{n-1}\right),\\
&\left.R^{n+1}=\lambda\widetilde{R}^{n+1}+(1-\lambda)\left(F(\phi^{n+1}),1\right)\right\}.
\endaligned
\end{equation}
Here, $\kappa^{n+1}\in[0,\frac23]$ will be given below.

Firstly, we show how to solve the scheme \eqref{RZF-BDF-e1} efficiently. Combining the first two equations in equation \eqref{RZF-BDF-e1} for the RZF-BDF2 scheme, we can obtain the following linear matrix equation
\begin{equation*}
\aligned
\displaystyle(3I+2\Delta t\mathcal{G}\mathcal{L})\phi^{n+1}=4I\phi^n-I\phi^{n-1}-2\Delta t\mathcal{G}F'(\widehat{\phi}^{n+1})-2\Delta t\mathcal{P}(\eta^{n+1})\mathcal{G}F'(\widehat{\phi}^{n+1}).
\endaligned
\end{equation*}

The coefficient matrix $A=(3I+2\Delta t\mathcal{G}\mathcal{L})$ is a symmetric positive matrix, then we have
\begin{equation}\label{RZF-BDF-e4}
\aligned
\phi^{n+1}
&=A^{-1}(4I\phi^n-I\phi^{n-1})-2\Delta tA^{-1}\mathcal{G}F'(\widehat{\phi}^{n+1})-2\Delta t\mathcal{P}(\eta^{n+1})A^{-1}\mathcal{G}F'(\widehat{\phi}^{n+1})\\
&=\overline{\phi}^{n+1}+\mathcal{P}(\eta^{n+1})q^{n+1}.
\endaligned
\end{equation}
Here $\overline{\phi}^{n+1}$ and $q^{n+1}$ can be determined as follows:
\begin{equation}\label{RZF-BDF-e5}
\aligned
\overline{\phi}^{n+1}=A^{-1}(4I\phi^n-I\phi^{n-1}-2\Delta t\mathcal{G}F'(\widehat{\phi}^{n+1})),\quad q^{n+1}=-2\Delta tA^{-1}\mathcal{G}F'(\widehat{\phi}^{n+1}).
\endaligned
\end{equation}
Noting that $\phi^{n+1}=\overline{\phi}^{n+1}+\mathcal{P}(\eta^{n+1})q^{n+1}$, then we can compute $\eta^{n+1}$ by the third and the fourth equations in \eqref{RZF-BDF-e1}:
\begin{equation}\label{RZF-BDF-e6}
\displaystyle3\widetilde{R}^{n+1}-4R^n+R^{n-1}=\displaystyle\left[1+\mathcal{P}(\eta^{n+1})\right]\left(F'(\widehat{\phi}^{n+1}),3\overline{\phi}^{n+1}+3\mathcal{P}(\eta^{n+1})q^{n+1}-4\phi^n+\phi^{n-1}\right),
\end{equation}
Similar as the RZF-CN scheme, the above equation is a also quadratic equation with one unknown for $\eta^{n+1}$. It means we can obtain determine $\eta^{n+1}$ explicitly from \eqref{RZF-BDF-e6}, namely,
\begin{equation}\label{RZF-BDF-e7}
\displaystyle\eta^{n+1}=\frac{-b\pm\sqrt{b^2-4ac}}{2a}.
\end{equation}
Here, if we set $\mathcal{P}(\eta)=\mathcal{P}_1(\eta)=k_1\eta$, then we have $\mathcal{P}(\eta^{n+1})=k_1\eta^{n+1}$. The coefficients $a$, $b$ and $c$ of above quadratic equation \eqref{RZF-BDF-e7} will satisfy:
\begin{equation*}
\aligned
&a=3k_1^2\left({F'}(\widehat{\phi}^{n+1}),q^{n+1}\right),\\
&b=k_1\left({F'}(\widehat{\phi}^{n+1}),3\widehat{\phi}^{n+1}-4\phi^n+\phi^{n-1}\right)+3k_1\left({F'}(\widehat{\phi}^{n+1}),q^{n+1}\right),\\
&c=-\left(3\widetilde{R}^{n+1}-4R^n+R^{n-1}\right)+\left({F'}(\widehat{\phi}^{n+1}),3\widehat{\phi}^{n+1}-4\phi^n+\phi^{n-1}\right).
\endaligned
\end{equation*}
If we set $\mathcal{P}(\eta)=\mathcal{P}_2(\eta)=k_2\eta_t$, then we have $\mathcal{P}(\eta^{n+1})=k_2\frac{3\eta^{n+1}-4\eta^n+\eta^{n-1}}{2\Delta t}$. The coefficients $a$, $b$ and $c$ of above quadratic equation \eqref{RZF-BDF-e7} will satisfy:
\begin{equation*}
\aligned
&a=\frac{27k_2^2}{4\Delta t^2}\left({F'}(\widehat{\phi}^{n+1}),q^{n+1}\right),\\
&b=\frac{3k_2}{2\Delta t}\left({F'}(\widehat{\phi}^{n+1}),3\overline{\phi}^{n+1}-4\phi^n+\phi^{n-1}-\frac{12\eta^n-3\eta^{n-1}}{2\Delta t}k_2q^{n+1}\right)\\
&\quad+\frac{9k_2}{2\Delta t}\left(1-\frac{4\eta^n-\eta^{n-1}}{2\Delta t}k_2\right)\left({F'}(\widehat{\phi}^{n+1}),q^{n+1}\right),\\
&c=\left(1-\frac{4\eta^n-\eta^{n-1}}{2\Delta t}k_2\right)\left({F'}(\widehat{\phi}^{n+1}),3\overline{\phi}^{n+1}-4\phi^n+\phi^{n-1}-\frac{12\eta^n-3\eta^{n-1}}{2\Delta t}k_2q^{n+1}\right)\\
&\quad-\left(3\widetilde{R}^{n+1}-4R^n+R^{n-1}\right).
\endaligned
\end{equation*}

Next, we will show that how to obtain the optimal choice for the relaxation parameter $\lambda_0$. The set $\mathcal{V}$ in \eqref{RZF-BDF-e3} can be simplified as
\begin{equation}\label{RZF-BDF-e8}
\aligned
\mathcal{V}=\left\{\lambda|\lambda\in[0,1]~s.t.~\right. &\left(\widetilde{R}^{n+1}-(F(\phi^{n+1}),1)\right)\lambda\leq\left(\widetilde{R}^{n+1}-(F(\phi^{n+1}),1)\right)+\Delta t\kappa^{n+1}\left(\mathcal{G}\mu^{n+1},\mu^{n+1}\right)\\
&\left.+\frac14\kappa^{n+1}\left(\mathcal{L}(\phi^{n+1}-2\phi^n+\phi^{n-1}),\phi^{n+1}-2\phi^n+\phi^{n-1}\right)\right\}.
\endaligned
\end{equation}

It is to obtain that $1\in\mathcal{V}$ which means the set $\mathcal{V}$ is non-empty because of the fact $\Delta t\kappa^{n+1}(\mathcal{G}\mu^{n+1},\mu^{n+1})\geq0$ and $\frac14\kappa^{n+1}\left(\mathcal{L}(\phi^{n+1}-2\phi^n+\phi^{n-1}),\phi^{n+1}-2\phi^n+\phi^{n-1}\right)\geq0$. The optimal parameter $\lambda_0$ can be chosen as a solution of the following optimization problem:
\begin{equation}\label{RZF-BDF-e9}
\aligned
\lambda_0=\min\limits_{\lambda\in[0,1]}\lambda~s.t.~ &\left(\widetilde{R}^{n+1}-(F(\phi^{n+1}),1)\right)\lambda\leq\left(\widetilde{R}^{n+1}-(F(\phi^{n+1}),1)\right)+\Delta t\kappa^{n+1}\left(\mathcal{G}\mu^{n+1},\mu^{n+1}\right)\\
&+\frac14\kappa^{n+1}\left(\mathcal{L}(\phi^{n+1}-2\phi^n+\phi^{n-1}),\phi^{n+1}-2\phi^n+\phi^{n-1}\right).
\endaligned
\end{equation}
The next theorem summarizes the choice of $\lambda_0$ and $\kappa^{n+1}$:\\
\begin{theorem}\label{RZF-BDF-th1}
Define $\beta=\frac{
\Delta t\left(\mathcal{G}\mu^{n+\frac12},\mu^{n+\frac12}\right)+\frac14\left(\mathcal{L}(\phi^{n+1}-2\phi^n+\phi^{n-1}),\phi^{n+1}-2\phi^n+\phi^{n-1}\right)}{\left|\widetilde{R}^{n+1}-\left(F(\phi^{n+1}),1\right)\right|}$ under the condition of $\widetilde{R}^{n+1}-\left(F(\phi^{n+1}),1\right)\neq0$, then we can choose the optimal relaxation parameter $\lambda_0$ and $\kappa^{n+1}$ as follows:
\begin{enumerate}
\item[1.] If $\widetilde{R}^{n+1}\geq\left(F(\phi^{n+1}),1\right)$, we set $\lambda_0=0$ and $\kappa^{n+1}=0$;
\item[2.] If $\widetilde{R}^{n+1}<\left(F(\phi^{n+1}),1\right)$ and $\beta\geq\frac32$, we set $\lambda_0=0$ and $\kappa^{n+1}=\frac1\beta$;
\item[3.] If $\widetilde{R}^{n+1}<\left(F(\phi^{n+1}),1\right)$ and $\beta\in[0,\frac32)$, we set $\lambda_0=1-\frac23\beta$ and $\kappa^{n+1}=\frac23$.
\end{enumerate}
\end{theorem}
\begin{proof}
(1) if $\widetilde{R}^{n+1}=\left(F(\phi^{n+1}),1\right)$, the inequality \eqref{RZF-BDF-e9} is always true for any $\lambda\in[0,1]$. Thus, we have the optimal relaxation parameter $\lambda_0=\min\limits_{\lambda\in[0,1]}\lambda=0$.

(2) if $\widetilde{R}^{n+1}>\left(F(\phi^{n+1}),1\right)$, dividing by $\widetilde{R}^{n+1}-\left(F(\phi^{n+1}),1\right)$ for both sides of the inequality \eqref{RZF-BDF-e9}, we have
\begin{equation*}
\aligned
\lambda\leq1+\kappa^{n+1}\beta.
\endaligned
\end{equation*}
Setting $\kappa^{n+1}=0$, then we have $\lambda\leq1$ which means all parameters in $[0,1]$ are satisfy above inequality. Hence, we have $\lambda_0=\min\limits_{\lambda\in[0,1]}\lambda=0$.

(3) if $\widetilde{R}^{n+1}<\left(F(\phi^{n+1}),1\right)$, the optimization problem \eqref{RZF-BDF-e9} will be simplified as:
\begin{equation*}
\aligned
\lambda_0=\min\limits_{\lambda\in[0,1]}\lambda~s.t.~\lambda\geq1-\kappa^{n+1}\beta.
\endaligned
\end{equation*}
Noting that $\kappa^{n+1}\in[0,\frac23]$, then if $\alpha\geq\frac32$, we have $\kappa^{n+1}\beta\in[0,1]$. Then we set $\kappa^{n+1}=\frac1\beta$ to obtain that $\lambda\geq1-\kappa^{n+1}\beta=0$ is always true. It means $\lambda_0=\min\limits_{\lambda\in[0,1]}\lambda=0$. Secondly, if $\beta\in[0,\frac32)$, we have $\kappa^{n+1}\beta<1$ to let $(1-\kappa^{n+1}\beta)\in(0,1]$ for any $\kappa^{n+1}\in[0,\frac23]$. Then we obtain $\lambda_0=\min\limits_{\lambda\in[0,1]}\lambda=1-\beta\max{\kappa^{n+1}}$. Noting that $\max{\kappa^{n+1}}=\frac23$, we obtain the optimal solution $\lambda_0=1-\frac23\beta$.
\end{proof}\\
\begin{theorem}\label{RZF-BDF-th2}
The second-order RZF-BDF2 scheme \eqref{RZF-BDF-e1}-\eqref{RZF-BDF-e2} with the above optimal choice of $\lambda_0$ and $\kappa^{n+1}$ is unconditionally energy stable in the sense that
\begin{equation}\label{RZF-BDF-e10}
\mathcal{\widetilde{E}}(\phi^{n+1})-\mathcal{\widetilde{E}}(\phi^{n})\leq-(1-\frac32\kappa^{n+1})\left[\Delta t(\mathcal{G}\mu^{n+1},\mu^{n+1})+\frac14\left(\mathcal{L}(\phi^{n+1}-2\phi^n+\phi^{n-1}),\phi^{n+1}-2\phi^n+\phi^{n-1}\right)\right]\leq0,
\end{equation}
where $\mathcal{\widetilde{E}}(\phi^{n+1})=\frac14(\mathcal{L}\phi^{n+1},\phi^{n+1})+\frac14\left(\mathcal{L}(2\phi^{n+1}-\phi^n),2\phi^{n+1}-\phi^n\right)+\frac32R^{n+1}-\frac12R^n$.
\end{theorem}
\begin{proof}
Firstly, taking the inner product of the first equation in the RZF-BDF2 scheme \eqref{RZF-BDF-e1} with $\Delta t\mu^{n+1}$, we could get
\begin{equation}\label{RZF-BDF-e11}
\aligned
\displaystyle \frac12\left(3\phi^{n+1}-4\phi^n+\phi^{n-1},\mu^{n+1}\right)=-\Delta t(\mathcal{G}\mu^{n+1},\mu^{n+1}).
\endaligned
\end{equation}
Secondly, by taking the inner products of the second equation in \eqref{RZF-BDF-e1} with $\frac12\left(3\phi^{n+1}-4\phi^n+\phi^{n-1}\right)$ and using the identity:
\begin{equation*}
\aligned
\displaystyle 2(a^{k+1},3a^{k+1}-4a^k+a^{k-1})=|a^{k+1}|^2+|2a^{k+1}-a^{k}|^2+|a^{k+1}-2a^k+a^{k-1}|^2-|a^{k}|^2-|2a^{k}-a^{k-1}|^2,
\endaligned
\end{equation*}
one obtains
\begin{equation}\label{RZF-BDF-e12}
\aligned
&\displaystyle \frac12\left(3\phi^{n+1}-4\phi^n+\phi^{n-1},\mu^{n+1}\right)\\
&\displaystyle=\frac12\left(\mathcal{L}\phi^{n+1},3\phi^{n+1}-4\phi^n+\phi^{n-1}\right)+\frac12\left[1+\mathcal{P}(\eta^{n+1})\right]\left({F'}(\widehat{\phi}^{n+1}),3\phi^{n+1}-4\phi^n+\phi^{n-1}\right)\\
&\displaystyle=\frac14(\mathcal{L}\phi^{n+1},\phi^{n+1})+\frac14\left(\mathcal{L}(2\phi^{n+1}-\phi^n),2\phi^{n+1}-\phi^n\right)\\
&\quad\displaystyle-\frac14(\mathcal{L}\phi^{n},\phi^{n})-\frac14\left(\mathcal{L}(2\phi^{n}-\phi^{n-1}),2\phi^{n}-\phi^{n-1}\right)\\
&\quad\displaystyle+\frac14\left(\mathcal{L}(\phi^{n+1}-2\phi^n+\phi^{n-1}),\phi^{n+1}-2\phi^n+\phi^{n-1}\right)\\
&\quad\displaystyle+\frac12\left[1+\mathcal{P}(\eta^{n+1})\right]\left({F'}(\widehat{\phi}^{n+1}),3\phi^{n+1}-4\phi^n+\phi^{n-1}\right).
\endaligned
\end{equation}
We rewrite the third equation in \eqref{RZF-BDF-e1} as follows:
\begin{equation}\label{RZF-BDF-e13}
\aligned
\displaystyle\left(\frac32\widetilde{R}^{n+1}-\frac12R^{n}\right)-\left(\frac32R^{n}-\frac12R^{n-1}\right)=\frac12\left[1+\mathcal{P}(\eta^{n+1})\right]\left({F'}(\widehat{\phi}^{n+1}),3\phi^{n+1}-4\phi^n+\phi^{n-1}\right).
\endaligned
\end{equation}
Substituting above equation \eqref{RZF-BDF-e13} into \eqref{RZF-BDF-e12} and combining it with \eqref{RZF-BDF-e11}, we could have
\begin{equation}\label{RZF-BDF-e14}
\aligned
&\frac14(\mathcal{L}\phi^{n+1},\phi^{n+1})+\frac14\left(\mathcal{L}(2\phi^{n+1}-\phi^n),2\phi^{n+1}-\phi^n\right)+\frac32\widetilde{R}^{n+1}-\frac12R^n\\
&-\frac14(\mathcal{L}\phi^{n},\phi^{n})-\frac14\left(\mathcal{L}(2\phi^{n}-\phi^{n-1}),2\phi^{n}-\phi^{n-1}\right)-\frac32R^{n}+\frac12R^{n-1}\\
&=-\left[\Delta t(\mathcal{G}\mu^{n+1},\mu^{n+1})+\frac14\left(\mathcal{L}(\phi^{n+1}-2\phi^n+\phi^{n-1}),\phi^{n+1}-2\phi^n+\phi^{n-1}\right)\right]\leq0
\endaligned
\end{equation}
From the constraint condition in \eqref{RZF-BDF-e3}, we could obtain
\begin{equation}\label{RZF-BDF-e15}
\aligned
\displaystyle \frac32R^{n+1}-\frac32\widetilde{R}^{n+1}\leq\frac32\kappa^{n+1}\left[\Delta t\left(\mathcal{G}\mu^{n+1},\mu^{n+1}\right)
+\frac14\left(\mathcal{L}(\phi^{n+1}-2\phi^n+\phi^{n-1}),\phi^{n+1}-2\phi^n+\phi^{n-1}\right)\right],
\endaligned
\end{equation}
Adding the above two equations together and noting that $\kappa^{n+1}\in[0,\frac23]$, we could have
\begin{equation}\label{RZF-BDF-e16}
\mathcal{\widetilde{E}}(\phi^{n+1})-\mathcal{\widetilde{E}}(\phi^{n})\leq-(1-\frac32\kappa^{n+1})\left[\Delta t(\mathcal{G}\mu^{n+1},\mu^{n+1})+\frac14\left(\mathcal{L}(\phi^{n+1}-2\phi^n+\phi^{n-1}),\phi^{n+1}-2\phi^n+\phi^{n-1}\right)\right]\leq0,
\end{equation}
which completes the proof.
\end{proof}
\section{The Relaxed Multiple Zero-Factor Approach}
The nonlinear free energy of many complex gradient flows contains disparate terms such that schemes with a single zero-factor may require excessively small time steps to obtain correct simulations. In this section, Inspired by the multiple SAV (MSAV) approach in \cite{cheng2018multiple}, we will give a relaxed multiple zero-factor (RMZF) approach to simulate the gradient flow with two disparate nonlinear terms:
\begin{equation}\label{RMZF-e1}
   \begin{array}{l}
\displaystyle\frac{\partial \phi}{\partial t}=-\mathcal{G}\mu,\\
\mu=\mathcal{L}\phi+F_1'(\phi)+F_2'(\phi),\\
   \end{array}
  \end{equation}
where $\mathcal{L}$ is a linear operator, $F_1'(\phi)$ and $F_2'(\phi)$ are nonlinear potential functions, $\mathcal{G}$ is
a positive or semi-positive linear operator. The above system satisfies the following energy dissipation law:
\begin{equation}\label{RMZF-e2}
\aligned
\displaystyle\frac{dE}{dt}=-(\mathcal{G}\mu,\mu)\leq0,
\endaligned
\end{equation}
where $\mu=\frac{\delta E}{\delta \phi}$ and the free energy $E(\phi)$ is
\begin{equation}\label{RMZF-e3}
E(\phi)=\frac12(\phi,\mathcal{L}\phi)+\int_\Omega F_1(\phi)d\textbf{x}+\int_\Omega F_2(\phi)d\textbf{x}.
\end{equation}
Introducing two linear zero factors $\mathcal{P}(\eta)$, $\mathcal{S}(\eta)$ of a scalar auxiliary function $\eta(t)$ and two SAVs $R_1=\int_\Omega F_1(\phi)d\textbf{x}$, $R_2=\int_\Omega F_2(\phi)d\textbf{x}$, we can rewrite the system \eqref{RMZF-e1} as the following
\begin{equation}\label{RMZF-e4}
   \begin{array}{l}
\displaystyle\frac{\partial \phi}{\partial t}=-\mathcal{G}\mu,\\
\mu=\mathcal{L}\phi+F_1'(\phi)+F_2'(\phi)+\mathcal{P}(\eta)F_1'(\phi)+\mathcal{S}(\eta)F_2'(\phi),\\
\displaystyle\frac{\partial R_1}{\partial t}=\left[1+\mathcal{P}(\eta)\right]F_1'(\phi),\\
\displaystyle\frac{\partial R_2}{\partial t}=\left[1+\mathcal{S}(\eta)\right]F_2'(\phi),\\
R_1=(F_1(\phi),1),\\
R_2=(F_2(\phi),1).
   \end{array}
\end{equation}
Here $\mathcal{P}(\eta)$ and $\mathcal{S}(\eta)$ are two linear zero factors. We can choose the following linear functionals with initial conditions:
\begin{equation}\label{RMZF-e5}
\aligned
&\displaystyle \mathcal{P}_1(\eta)=k_1\eta \text{~with~} \eta(0)=0,\quad \mathcal{P}_2(\eta)=k_2\eta_t \text{~with~} \eta(0)=c_1,\\
&\displaystyle \mathcal{S}_1(\eta)=k_3\eta \text{~with~} \eta(0)=0,\quad \mathcal{S}_2(\eta)=k_4\eta_t \text{~with~} \eta(0)=c_2,
\endaligned
\end{equation}
where $k_1$, $k_2$, $k_3$, $k_4$ are arbitrary non-zero constants and $c_1$, $c_2$ are any real numbers.

A second order Crank-Nicolson scheme based on the relaxed multiple zero-factor approach (RMZF-CN): set $R_1^0=(F_1(\phi(0,\textbf{x}),1)$, $R_2^0=(F_2(\phi(0,\textbf{x}),1)$ and compute $\phi^{n+1}$, $R^{n+1}$ via the following two steps:

\textbf{Step I}: Compute $\phi^{n+1}$, $\widetilde{R}_1^{n+1}$ and $\widetilde{R}_2^{n+1}$ by the following semi-implicit Crank-Nicolson scheme:
\begin{equation}\label{RMZF-CN-e1}
   \begin{array}{lrl}
&\displaystyle\frac{\overline{\phi}^{n+1}-\phi^n}{\Delta t}&=-\mathcal{G}\mu^{n+\frac12},\\
&\displaystyle\mu^{n+\frac12}&=\left[\frac12\mathcal{L}\overline{\phi}^{n+1}+\frac12\mathcal{L}\phi^{n}+F_1'(\widehat{\phi}^{n+\frac12})++F_2'(\widehat{\phi}^{n+\frac12})\right],\\
&\displaystyle\phi^{n+1}&=\overline{\phi}^{n+1}+\mathcal{P}(\eta^{n+\frac12})q_1^{n+1}+\mathcal{S}(\eta^{n+\frac12})q_2^{n+1},\\
&\displaystyle (\widetilde{R}_1^{n+1}-R_1^n)+(\widetilde{R}_2^{n+1}-R_2^n)&=\displaystyle\left[1+\mathcal{P}(\eta^{n+\frac12})\right]\left(F_1'(\widehat{\phi}^{n+\frac12}),\phi^{n+1}-\phi^n\right)\\
&&\quad+\left[1+\mathcal{S}(\eta^{n+\frac12})\right]\left(F_2'(\widehat{\phi}^{n+\frac12}),\phi^{n+1}-\phi^n\right),\\
\displaystyle &\widetilde{R}_1^{n+1}&=\left(F_1(\overline{\phi}^{n+1}),1\right),\quad\widetilde{R}_2^{n+1}=\left(F_2(\overline{\phi}^{n+1}),1\right).
   \end{array}
\end{equation}
where $\widehat{\phi}^{n+\frac12}=\frac32\phi^n-\frac12\phi^{n-1}$, and
\begin{equation*}
\aligned
q_1^{n+1}=-\Delta tA^{-1}\mathcal{G}F_1'(\widehat{\phi}^{n+\frac12}),\quad q_2^{n+1}=-\Delta tA^{-1}\mathcal{G}F_2'(\widehat{\phi}^{n+\frac12}).
\endaligned
\end{equation*}
Here the coefficient matrix $A=(I+\frac12\Delta t\mathcal{G}\mathcal{L})$.

\textbf{Step II}: Update the scalar auxiliary variable $R_1^{n+1}$ and $R_2^{n+1}$ via a relaxation step as
\begin{equation}\label{RMZF-CN-e2}
R_1^{n+1}=\lambda_0\widetilde{R}_1^{n+1}+(1-\lambda_0)\left(F_1(\phi^{n+1}),1\right),\quad R_2^{n+1}=\lambda_0\widetilde{R}_2^{n+1}+(1-\lambda_0)\left(F_2(\phi^{n+1}),1\right),\quad \lambda_0\in\mathcal{V}.
\end{equation}
Here is $\mathcal{V}$ a set defined by
\begin{equation}\label{RMZF-CN-e3}
\aligned
\mathcal{V}=\left\{\lambda|\lambda\in[0,1]~s.t.~\right.
&(R_1^{n+1}+R_2^{n+1})-(\widetilde{R}_1^{n+1}+\widetilde{R}_2^{n+1})\leq\Delta t\kappa^{n+1}\left(\mathcal{G}\mu^{n+\frac12},\mu^{n+\frac12}\right),\\
&\left.R_1^{n+1}=\lambda\widetilde{R}_1^{n+1}+(1-\lambda)\left(F_1(\phi^{n+1}),1\right),~ R_2^{n+1}=\lambda\widetilde{R}_2^{n+1}+(1-\lambda)\left(F_2(\phi^{n+1}),1\right)\right\}.
\endaligned
\end{equation}
Here, $\kappa^{n+1}\in[0,1]$ will be given below.

Substituting the third equation in \eqref{RMZF-CN-e1} into the fourth equation to obtain
\begin{equation}\label{RMZF-CN-e4}
\aligned
\displaystyle(\widetilde{R}_1^{n+1}-R_1^n)+(\widetilde{R}_2^{n+1}-R_2^n)=
&\displaystyle\left[1+\mathcal{P}(\eta^{n+\frac12})\right]\left(F_1'(\widehat{\phi}^{n+\frac12}),\overline{\phi}^{n+1}+\mathcal{P}(\eta^{n+\frac12})q_1^{n+1}+\mathcal{S}(\eta^{n+\frac12})q_2^{n+1}-\phi^n\right)\\
&+\displaystyle\left[1+\mathcal{S}(\eta^{n+\frac12})\right]\left(F_2'(\widehat{\phi}^{n+\frac12}),\overline{\phi}^{n+1}+\mathcal{P}(\eta^{n+\frac12})q_1^{n+1}+\mathcal{S}(\eta^{n+\frac12})q_2^{n+1}-\phi^n\right).
\endaligned
\end{equation}
Noting that both $\mathcal{P}(\eta)$ and $\mathcal{S}(\eta)$ are all linear functionals of $\eta$, we immediately obtain that the above equation is a quadratic equation with one unknown for $\eta^{n+\frac12}$. It means we can obtain determine $\eta^{n+\frac12}$ explicitly from \eqref{RMZF-CN-e4}, namely,
\begin{equation}\label{RMZF-CN-e5}
\displaystyle\eta^{n+\frac12}=\frac{-b\pm\sqrt{b^2-4ac}}{2a}.
\end{equation}
Here, The values of the coefficients $a$, $b$ and $c$ depend on the choice of the zero factors $\mathcal{P}(\eta)$ and $\mathcal{S}(\eta)$. For example, if we set $\mathcal{P}(\eta)=\mathcal{P}_1(\eta)=k_1\eta$, and $\mathcal{S}(\eta)=\mathcal{S}_2(\eta)=k_4\eta_t$, then we have $\mathcal{P}(\eta^{n+\frac12})=k_1\eta^{n+\frac12}$ and $\mathcal{S}(\eta^{n+\frac12})=k_4\frac{\eta^{n+1}-\eta^n}{\Delta t}$. The coefficients $a$, $b$ and $c$ can be given as follows:
\begin{equation*}
\aligned
&a=\left(k_1{F'}_1(\widehat{\phi}^{n+\frac12})+\frac{k_4}{\Delta t}{F'}_2(\widehat{\phi}^{n+\frac12}),k_1q_1^{n+1}+\frac{k_4}{\Delta t}q_2^{n+1}\right),\\
&b=\left(k_1{F'}_1(\widehat{\phi}^{n+\frac12})+\frac{k_4}{\Delta t}{F'}_2(\widehat{\phi}^{n+\frac12}),\overline{\phi}^{n+1}-\phi^n-\frac{k_4}{\Delta t}\eta^nq_2^{n+1}\right)\\
&\qquad+\left({F'}_1(\widehat{\phi}^{n+\frac12})+(1-\frac{k_4}{\Delta t}\eta^n){F'}_2(\widehat{\phi}^{n+\frac12}),k_1q_1^{n+1}+\frac{k_4}{\Delta t}q_2^{n+1}\right),\\
&c=-(\widetilde{R}_1^{n+1}-R_1^n)-(\widetilde{R}_2^{n+1}-R_2^n)+\left({F'}_1(\widehat{\phi}^{n+\frac12})+(1-\frac{k_4}{\Delta t}\eta^n){F'}_2(\widehat{\phi}^{n+\frac12}),\overline{\phi}^{n+1}-\phi^n-\frac{k_4}{\Delta t}\eta^nq_2^{n+1}\right).
\endaligned
\end{equation*}

Next, we will show that how to obtain the optimal choice for the relaxation parameter $\lambda_0$. Setting $\widetilde{R}^{n+1}=\widetilde{R}_1^{n+1}+\widetilde{R}_2^{n+1}$, $R^{n+1}=R_1^{n+1}+R_2^{n+1}$ and $E_1(\phi^{n+1})=(F_(\phi^{n+1})+F_2(\phi^{n+1}),1)$, we observe that the optimal relaxation parameter $\lambda_0$ is the solution of the following optimization problem:
\begin{equation}\label{RMZF-CN-e7}
\aligned
\lambda_0=\min\limits_{\lambda\in[0,1]}\lambda\quad s.t.~\left[\widetilde{R}^{n+1}-E_1(\phi^{n+1})\right]\lambda\leq\left[\widetilde{R}^{n+1}-E_1(\phi^{n+1})\right]+
\Delta t\kappa^{n+1}\left(\mathcal{G}\mu^{n+\frac12},\mu^{n+\frac12}\right).
\endaligned
\end{equation}
The next theorem summarizes the choice of $\lambda_0$ and $\kappa^{n+1}$:\\
\begin{theorem}\label{RMZF-CN-th1}
If $\widetilde{R}^{n+1}-E_1(\phi^{n+1})\neq0$, setting $\alpha=\frac{
\Delta t\left(\mathcal{G}\mu^{n+\frac12},\mu^{n+\frac12}\right)}{\left|\widetilde{R}^{n+1}-E_1(\phi^{n+1})\right|}$, then we can choose the optimal relaxation parameter $\lambda_0$ and $\kappa^{n+1}$ as follows:
\begin{enumerate}
\item[1.] If $\widetilde{R}^{n+1}\geq E_1(\phi^{n+1})$, we set $\lambda_0=0$ and $\kappa^{n+1}=0$;
\item[2.] If $\widetilde{R}^{n+1}<E_1(\phi^{n+1})$ and $\alpha\geq1$, we set $\lambda_0=0$ and $\kappa^{n+1}=\frac1\alpha$;
\item[3.] If $\widetilde{R}^{n+1}<E_1(\phi^{n+1})$ and $\alpha\in[0,1)$, we set $\lambda_0=1-\alpha$ and $\kappa^{n+1}=1$.
\end{enumerate}
\end{theorem}

\begin{theorem}\label{RMZF-CN-th2}
The second-order Crank-Nicolson scheme \eqref{RMZF-CN-e1}-\eqref{RMZF-CN-e2} based on the RMZF approach with the above choice of $\lambda_0$ and $\kappa^{n+1}$ is unconditionally energy stable in the sense that
\begin{equation}\label{RMZF-CN-e8}
\mathcal{\widetilde{E}}(\phi^{n+1})-\mathcal{\widetilde{E}}(\phi^{n})\leq-\Delta t(1-\kappa^{n+1})(\mathcal{G}\mu^{n+\frac12},\mu^{n+\frac12})\leq0,
\end{equation}
where $\mathcal{\widetilde{E}}(\phi^{n+1})=\frac12(\mathcal{L}\phi^{n+1},\phi^{n+1})+R_1^{n+1}+R_2^{n+1}$ and we further have
\begin{equation*}
\aligned
\mathcal{E}(\phi^{n+1})\leq\mathcal{E}(\phi^{n}),
\endaligned
\end{equation*}
under the condition of $\widetilde{R}^{n+1}\geq E_1(\phi^{n+1})$ or $\widetilde{R}^{n+1}<E_1(\phi^{n+1})$ with $\alpha\geq1$. Here $\mathcal{E}(\phi^{n})$ is the original energy where $\mathcal{E}(\phi^{n+1})=\frac12(\mathcal{L}\phi^{n+1},\phi^{n+1})+(F_1(\phi^{n+1}),1)+(F_2(\phi^{n+1}),1)$. If $\widetilde{R}^{n+1}<E_1(\phi^{n+1})$ and $\alpha\in[0,1)$, we could have
\begin{equation*}
\aligned
\mathcal{\widetilde{E}}(\phi^{n+1})\leq\mathcal{E}(\phi^{n+1}).
\endaligned
\end{equation*}
\end{theorem}
\section{Examples and discussion}
In this section, we implement the proposed Crank-Nicolson scheme based on the relaxed zero-factor approach (RZF-CN) and BDF2 scheme based on the zero-factor approach with relaxation (RZF-BDF2) and apply them to several classical gradient flow models include the Allen-Cahn model, the Cahn-Hilliard model, and the phase-field crystal model. In all considered examples, we consider the periodic boundary conditions and use a Fourier spectral method in space.

\subsection{Allen-Cahn model}
Consider the following Lyapunov energy functional:
\begin{equation}\label{section5_energy1}
E(\phi)=\int_{\Omega}(\frac{1}{2}|\nabla \phi|^2+\frac{1}{4\epsilon^2}(\phi^2-1)^2d\textbf{x},
\end{equation}
Given $\mathcal{G}=I$, then the gradient flow model in \eqref{intro-e2} reduces to the corresponding Allen-Cahn equation
\begin{equation}\label{section5_AC_model}
\displaystyle\frac{\partial \phi}{\partial t}=M\left(\Delta \phi-\frac{1}{\epsilon^2}(\phi^3-\phi)\right),\quad(\textbf{x},t)\in\Omega\times [0,T],
\end{equation}

For above Allen-Cahn model, given the following initial condition:
\begin{equation*}
\aligned
\phi(x,y,0)=0.001\cos(x)\cos(y),
\endaligned
\end{equation*}
in domain $\Omega=[0,2\pi]^2$. We use $128^2$ Fourier-spectral modes in space and set model parameters $T=1$, $\epsilon=0.4$, $M=1$. We consider Crank-Nicolson SAV scheme, the proposed second-order RZF-CN and RZF-BDF2 schemes to obtain the numerical error and convergence rates for above example. We set the zero factor $\mathcal{P}(\eta)=\eta_t$ in the RZF method. The results are shown in Table \ref{tab:tab1} which indicate that all convergence rates are consistent with the theoretical results. We also observe that the RZF-CN method is more accurate than the baseline SAV-CN method. The energy curves are plotted for both SAV-CN and RZF-CN schemes with $\Delta t=0.01$ in Figure \ref{fig:fig1}. It can be observed from this figure that the computed energy for both schemes decays with time. Figure \ref{fig:fig1} also indicates that our proposed RZF method is a very efficient way for preserving the consistency between the modified energy and the original energy. In the following, we also give the detailed results of the energy error and the introduced zero factor $\mathcal{P}(\eta)$. From Figure \ref{fig:fig2}, we observe that the RZF-CN scheme provides less error between $\mathcal{\widetilde{E}}$ and $E(\phi)$ than the RZF-BDF2 schemes. The values of the zero factor $\mathcal{P}(\eta)$ are very close to zero for both RZF-CN and RZF-BDF2 schemes.
\begin{table}[h!b!p!]
\small
\centering
\caption{\small The $L^\infty$ errors, convergence rates at $T=1$ for the second-order RZF-CN and RZF-BDF2 schemes.}\label{tab:tab1}
\begin{tabular}{|c|c|c|c|c|c|c|c|c|}
\hline
&\multicolumn{2}{c|}{SAV-CN}&\multicolumn{2}{c|}{RZF-CN}&\multicolumn{2}{c|}{RZF-BDF2}\\
\cline{1-7}
$\Delta t$&Error&Rate&Error&Rate&Error&Rate\\
\cline{1-7}
$5\times10^{-2}$        &2.1972e-2   &---   &1.2748e-2   &---   &3.0129e-2   &---   \\
$2.5\times10^{-2}$      &6.3229e-3   &1.7970&3.5123e-3   &1.8597&9.7308e-3   &1.6305\\
$1.25\times10^{-2}$     &1.6715e-3   &1.9194&9.1399e-4   &1.9421&2.7363e-3   &1.8303\\
$6.25\times10^{-3}$     &4.2834e-4   &1.9643&2.3249e-4   &1.9750&7.2166e-4   &1.9228\\
$3.125\times10^{-3}$    &1.0859e-4   &1.9798&5.8549e-5   &1.9894&1.8486e-4   &1.9649\\
\hline
\end{tabular}
\end{table}
\begin{figure}[htp]
\centering
\subfigure[The energy evolution for SAV-CN and RZF-CN schemes]{
\includegraphics[width=8cm,height=8cm]{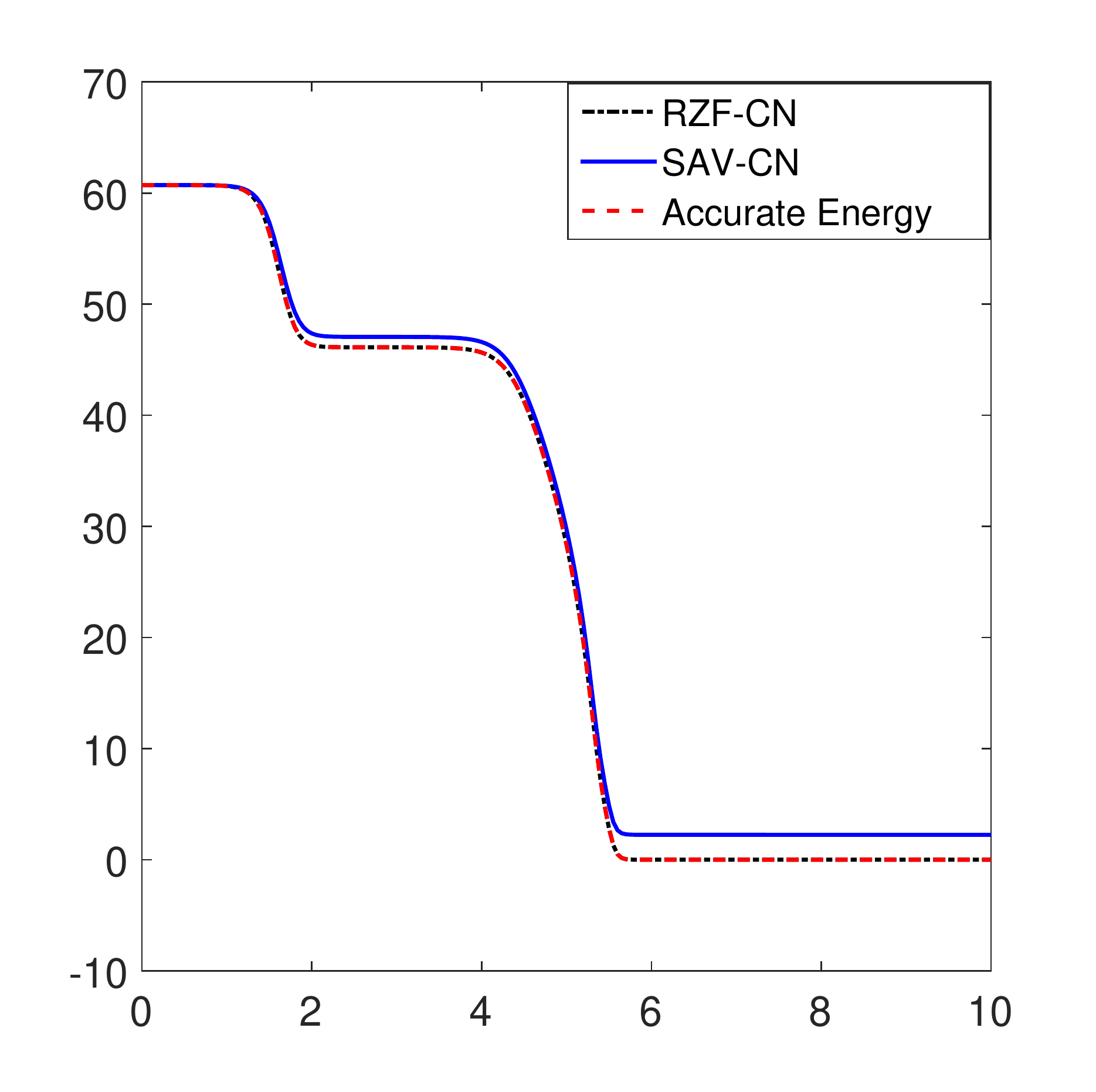}}
\subfigure[The energy error for SAV-CN and RZF-CN schemes]{
\includegraphics[width=8cm,height=8cm]{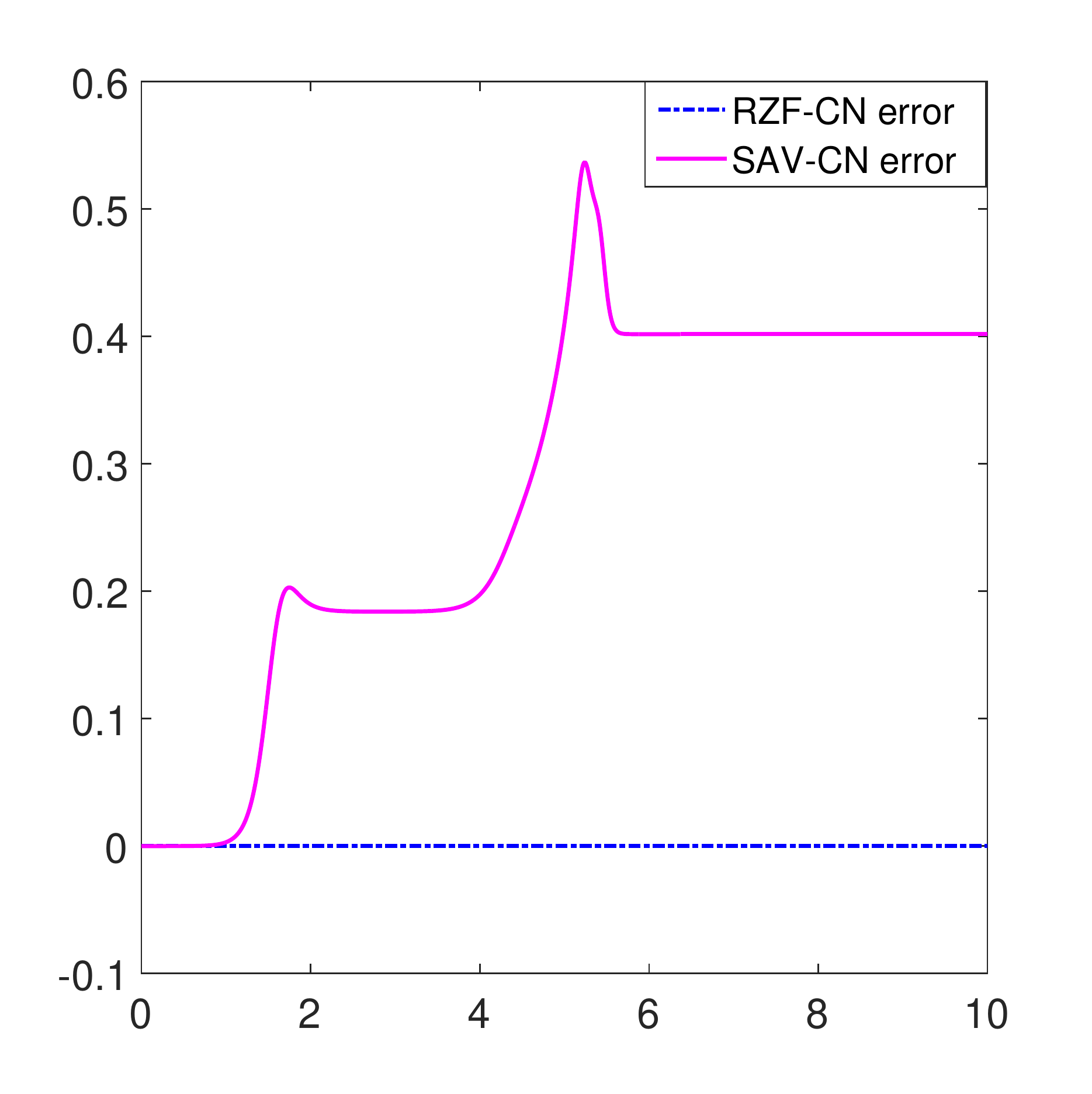}}
\caption{A comparison of the SAV-CN, and RZF-CN methods in solving the Allen-Cahn equation. (a) the numerical energies
using the SAV-CN and the RZF-CN schemes with $\Delta t=0.01$. (b) Numerical results of $\mathcal{\widetilde{E}}(\phi)-E(\phi)$ using the SAV-CN and the RZF-CN schemes with $\Delta t=0.01$.}\label{fig:fig1}
\end{figure}
\begin{figure}[htp]
\centering
\subfigure[MZF-CN and RZF-BDF2: numerical results of $\mathcal{\widetilde{E}}(\phi)-E(\phi)$]{
\includegraphics[width=8cm,height=8cm]{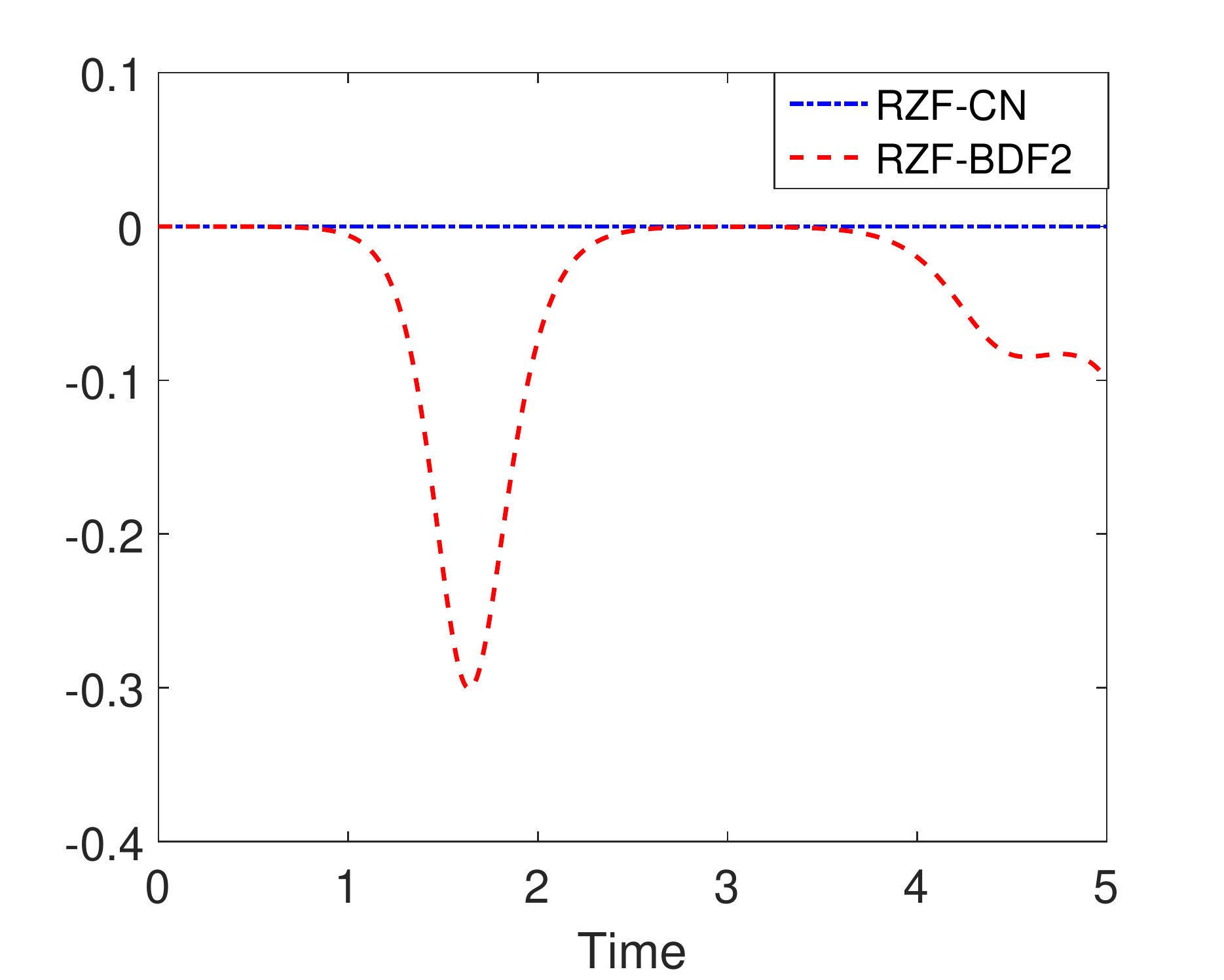}}
\subfigure[Numerical results of the zero factor $\mathcal{P}(\eta)$]{
\includegraphics[width=8cm,height=8cm]{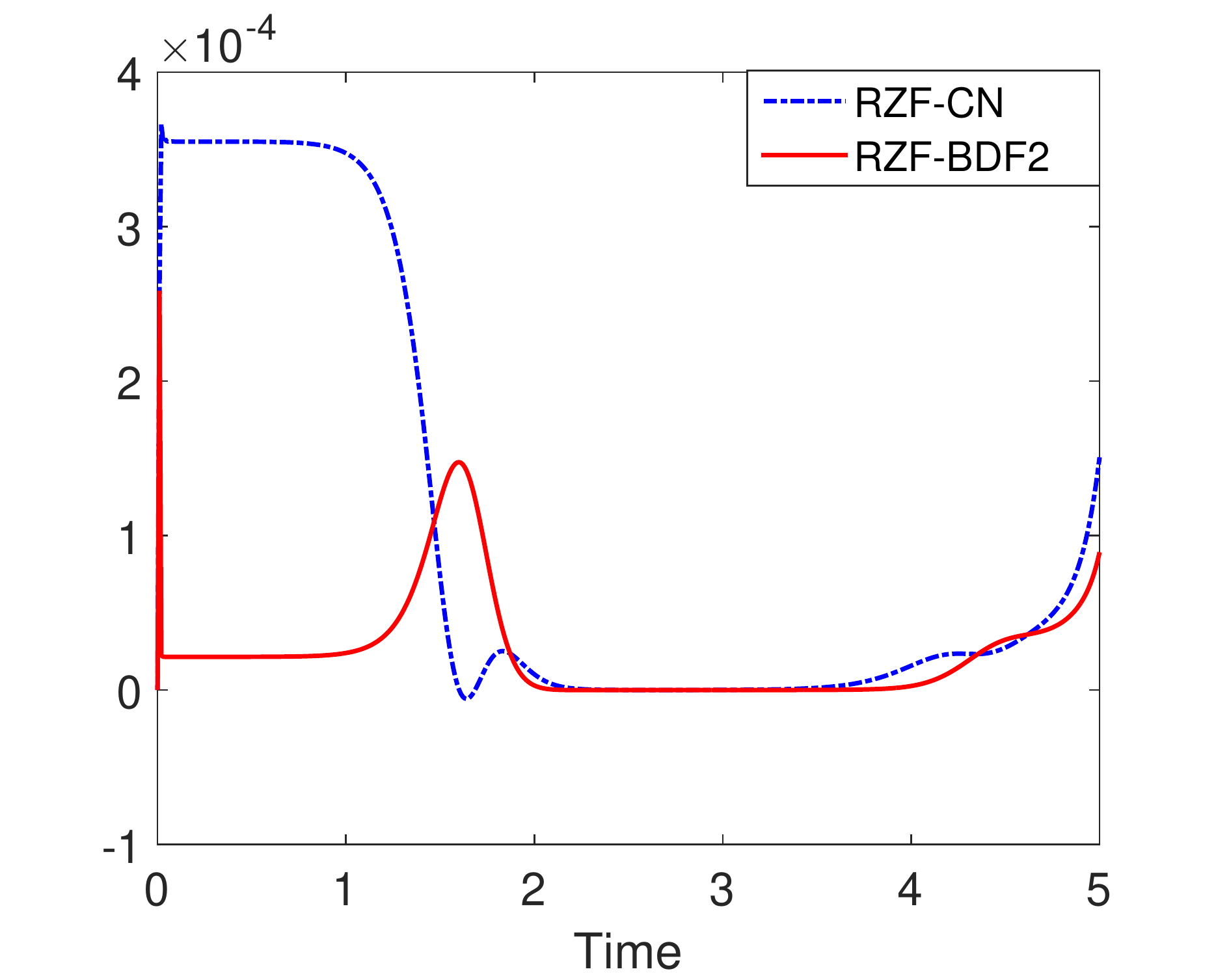}}
\caption{A comparison of SAV-CN and RZF-CN methods in solving the Allen-Cahn equation. (a) Numerical results of $\mathcal{\widetilde{E}}(\phi)-E(\phi)$ using the MZF-CN and the MZF-BDF2 schemes with $\Delta t=0.01$. (b) Numerical results of the zero factor $\mathcal{P}(\eta)$ using the RZF-CN and RZF-BDF2 schemes with $\Delta t=0.01$.}\label{fig:fig2}
\end{figure}

We perform the numerical test to the Allen-Cahn equation \eqref{section5_AC_model} with the following initial conditions in 2D and 3D,
\begin{eqnarray}
&&\phi(x,y,0)=\tanh\left(\frac{1.7+1.2\cos(6\theta)-\sqrt{x^2+y^2}}{\sqrt{2}\epsilon}\right),\label{AC-initial-e1}\\
&&\phi(x,y,z,0)=\tanh\left(\frac{\sqrt{(x-0.5)^2+(y-0.5)^2+(z-0.5)^2}-R_0}{\sqrt{2}\epsilon}\right),\label{AC-initial-e2}
\end{eqnarray}

For the initial condition \eqref{AC-initial-e1}, we set $\theta=tan^{-1}(y/x)$ for $(x,y)\in[-\pi,\pi]\times[-\pi,\pi]$ and $\epsilon=0.05$.
We choose $256\times256$ Fourier modes to discretize the space and use the time step $\Delta t=0.001$.  Figure \ref{fig:fig3-1} shows the numerical test results at $t=0$, $0.2$, $0.4$ and $1$ to the Allen-Cahn model using the RZF-CN method. These results are consistent with
those in \cite{yoon2020fourier}. In Figure \ref{fig:fig3-2}, we present a comparison of energy (a), the nonlinear free energy $(F(\phi),1)$ and $R^{n+1}$ (b) and the zero factor $\mathcal{P}(\eta)=\eta$ (c) of RZF-CN scheme. It is obvious to see that the proposed RZF method dissipates the almost original energy. From (b) and (c) in Figure \ref{fig:fig3-2}, we observe that the values of the zero factor are essentially zero except at a few time steps for this example. The trend of the value of zero factor $\mathcal{P}(\eta)$ is highly consistent with the variation of the nonlinear free energy which implies that the introduced zero factor $\mathcal{P}(\eta)$ is essential and very efficient to capture the sharp dissipation of the nonlinear free energy.

For the initial condition \eqref{AC-initial-e2}, we set $\Omega=[0,1]^3$ with $R_0=0.3$, $\epsilon=0.02$, $T=3.5$, $M=0.01$ and give $\Delta t=0.01$. We discretize the space by the Fourier spectral method with $128\times128\times128$ modes. The snapshots of zero level set to the numerical solutions using the RZF-CN method are shown in Figure \ref{fig:fig3-3}. The simulation results depict the motion by mean curvature property and non-conservation of mass very well.
\begin{figure}[htp]
\centering
\includegraphics[width=4cm,height=3.5cm]{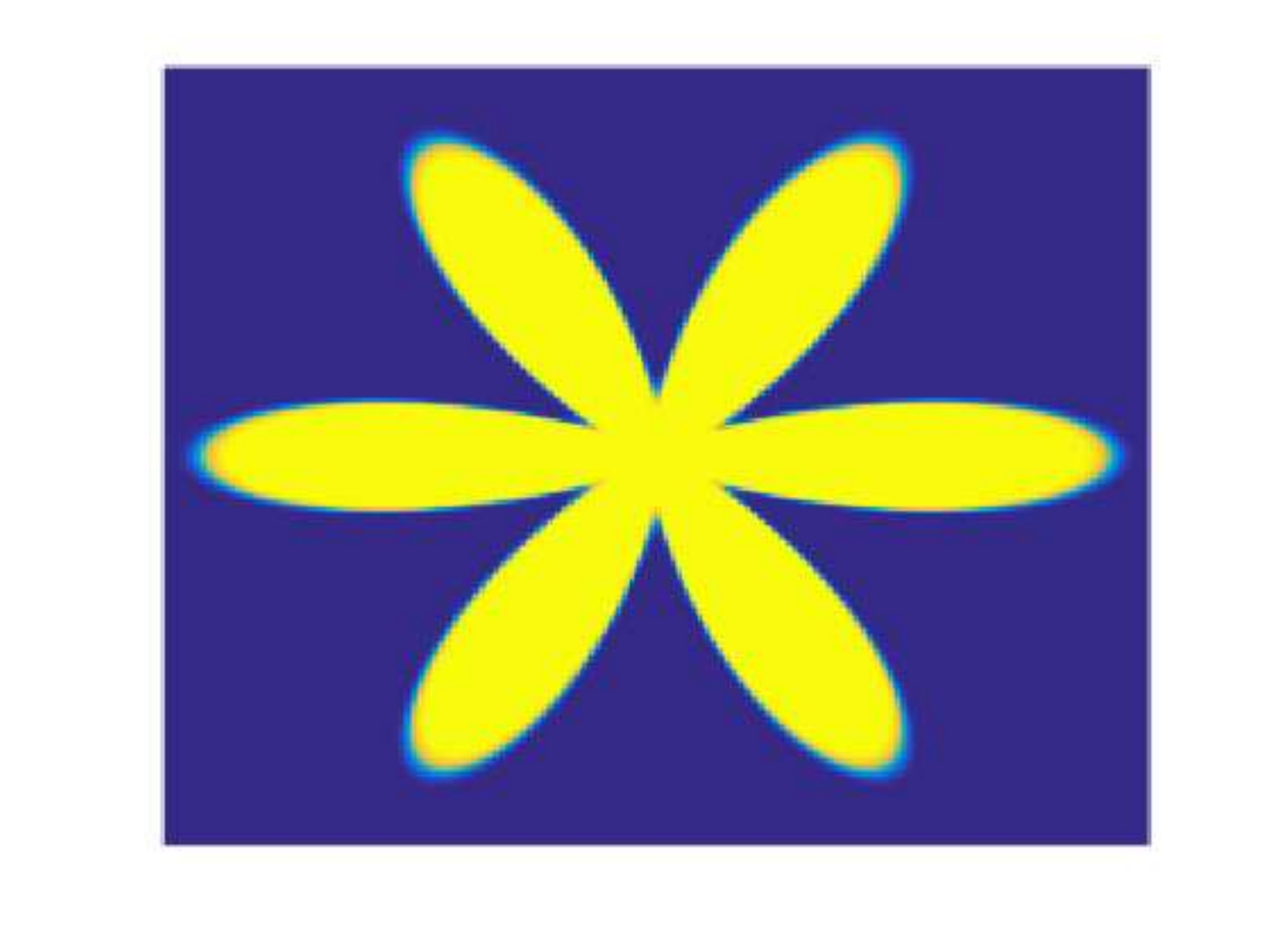}
\includegraphics[width=4cm,height=3.5cm]{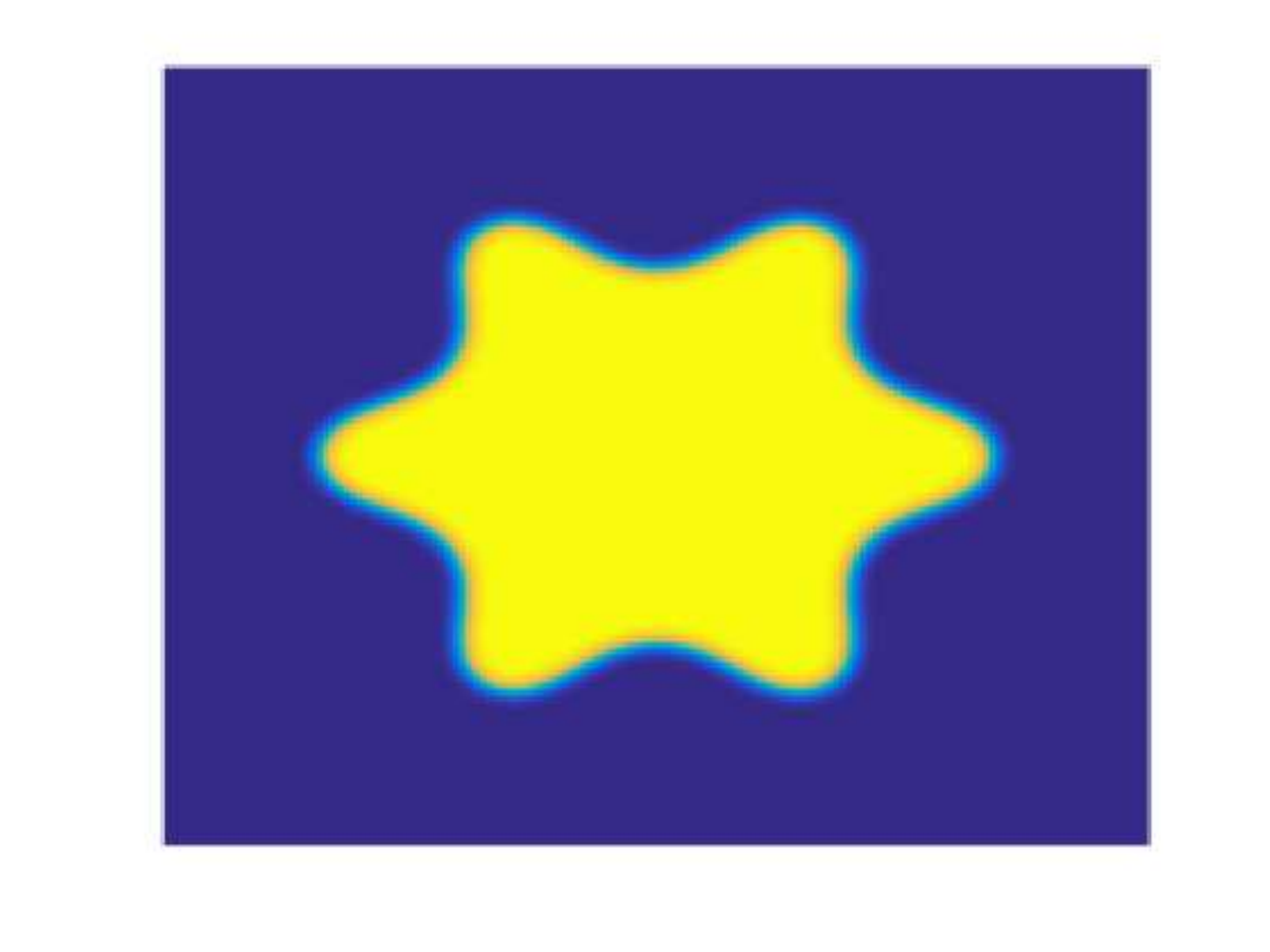}
\includegraphics[width=4cm,height=3.5cm]{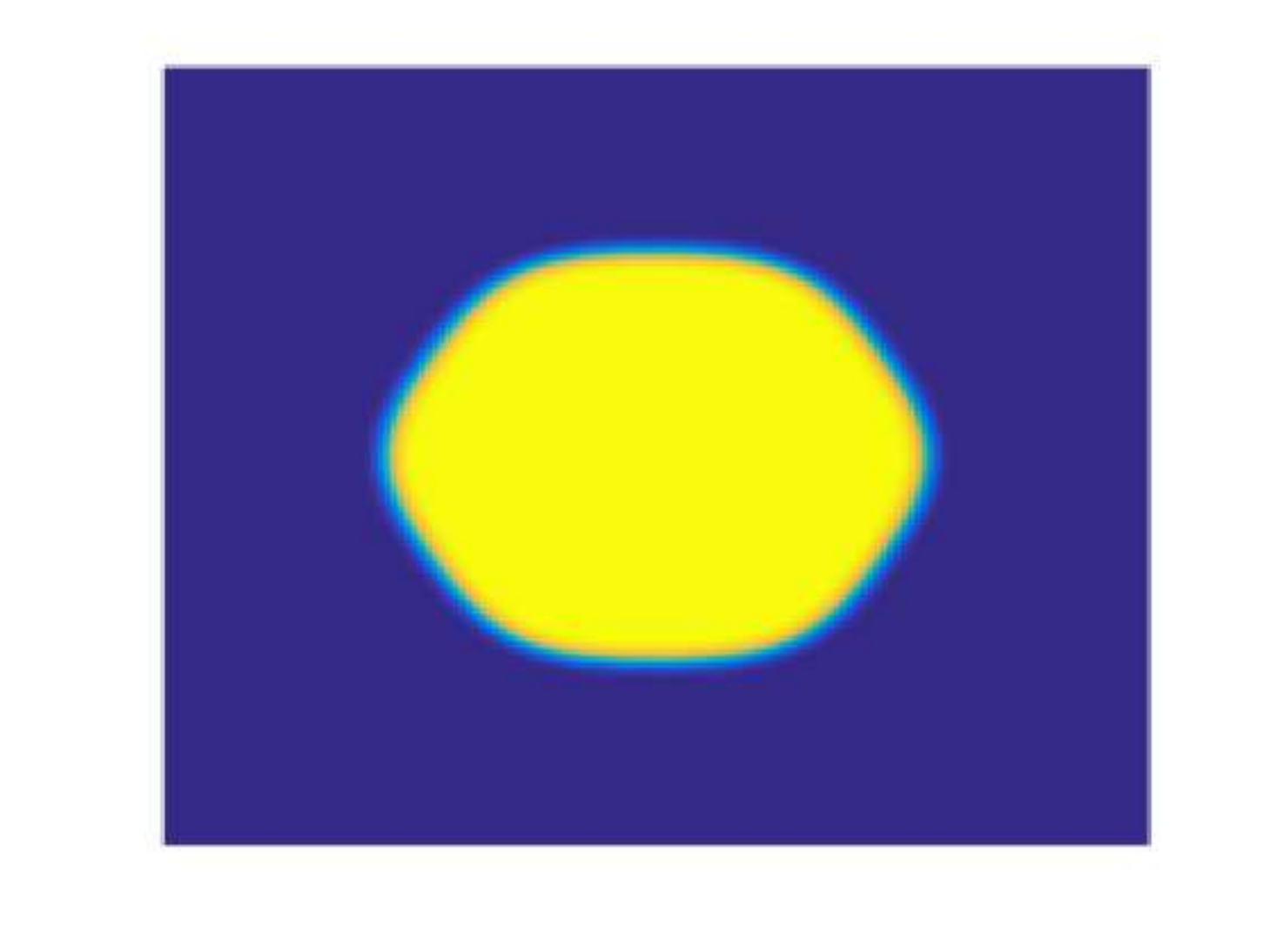}
\includegraphics[width=4cm,height=3.5cm]{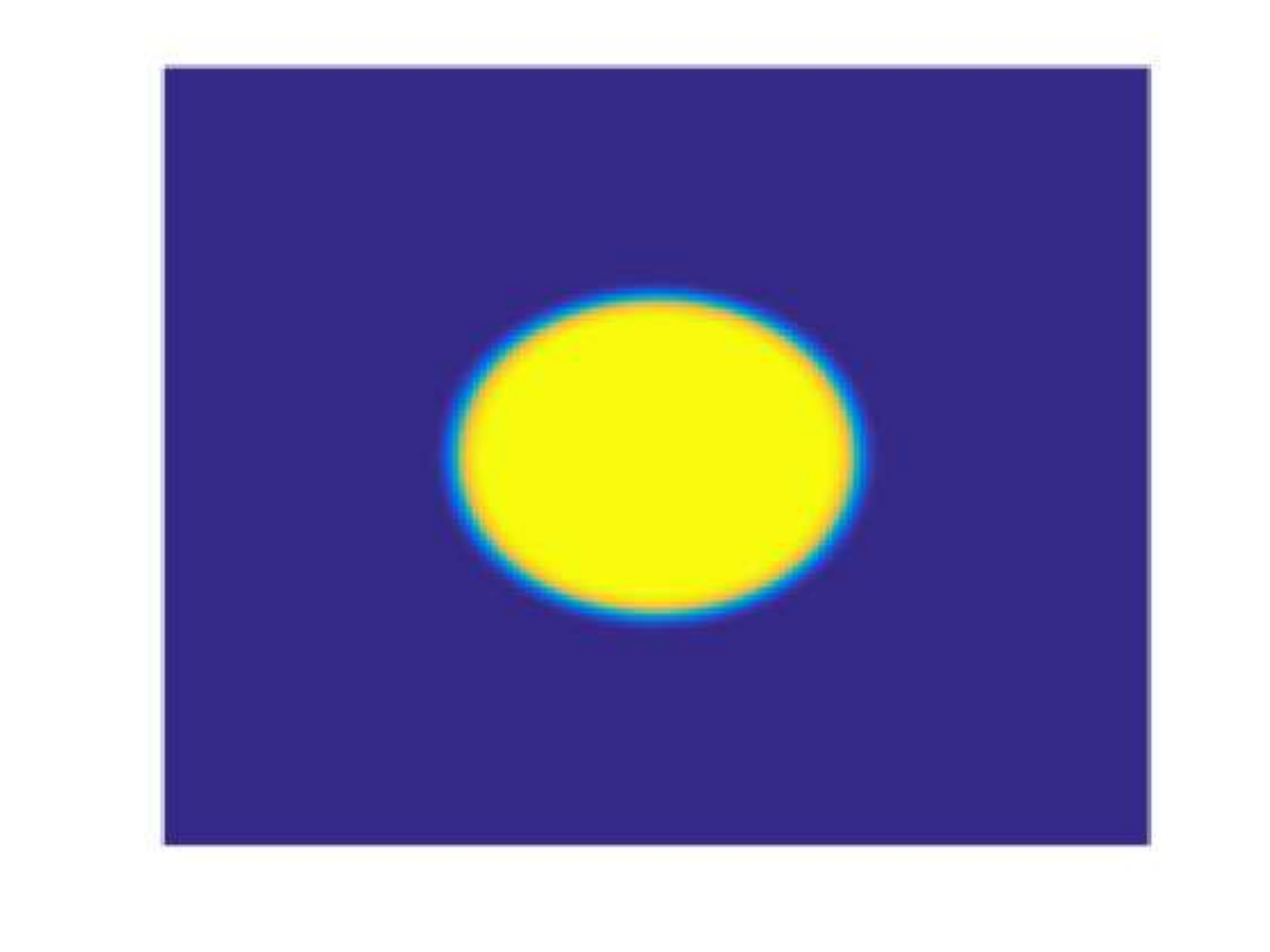}
\caption{Snapshots of zero level set to the numerical solutions of Allen-Cahn model with the initial condition \eqref{AC-initial-e1} at t=0, 0.2, 0.4, 1.}\label{fig:fig3-1}
\end{figure}
\begin{figure}[htp]
\centering
\subfigure[Energy of $\mathcal{\widetilde{E}}(\phi)$ and $E(\phi)$]{
\includegraphics[width=5cm,height=5cm]{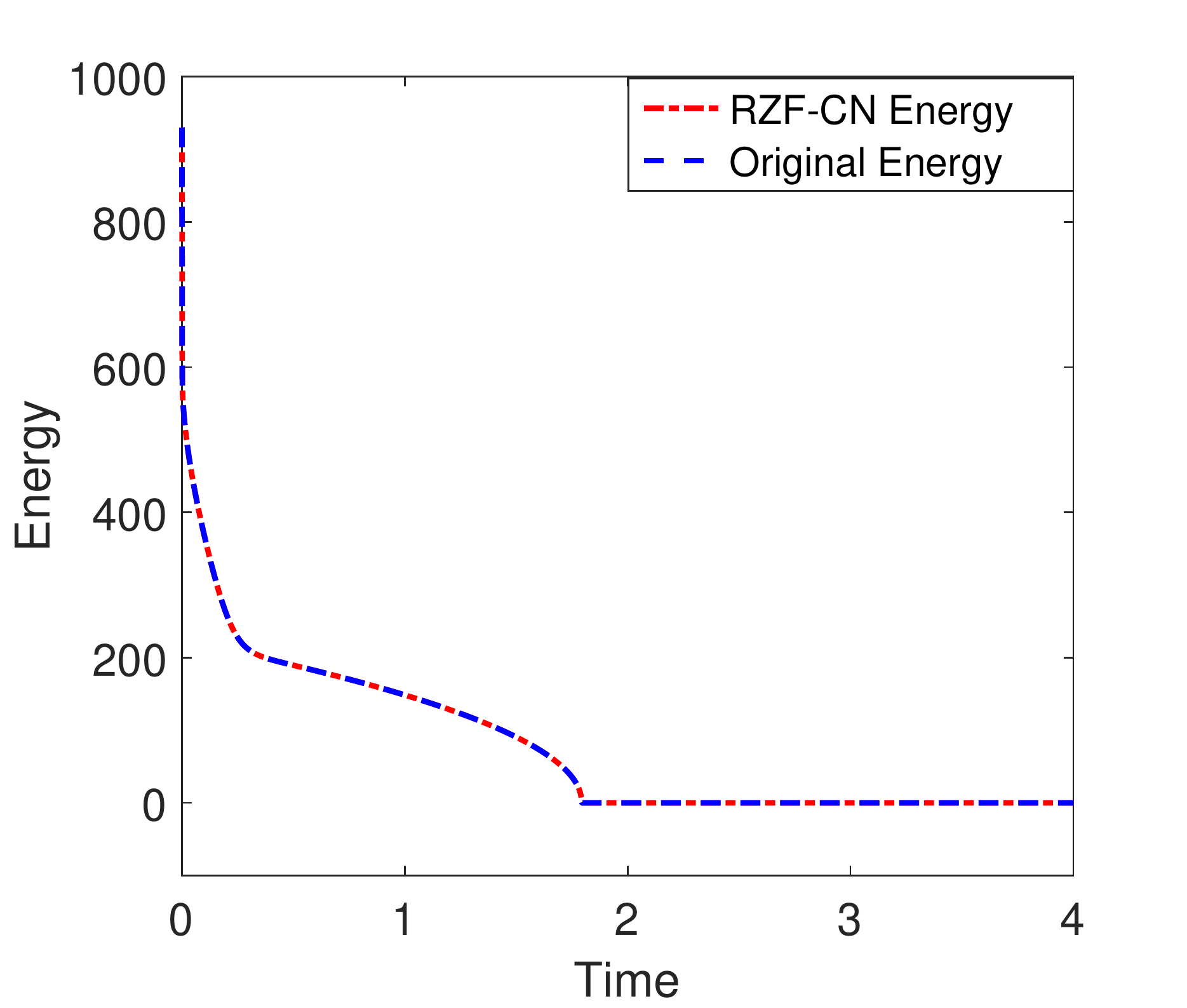}}
\subfigure[$R^{n+1}$ and $\left((F(\phi^{n+1}),1\right)$]{
\includegraphics[width=5cm,height=5cm]{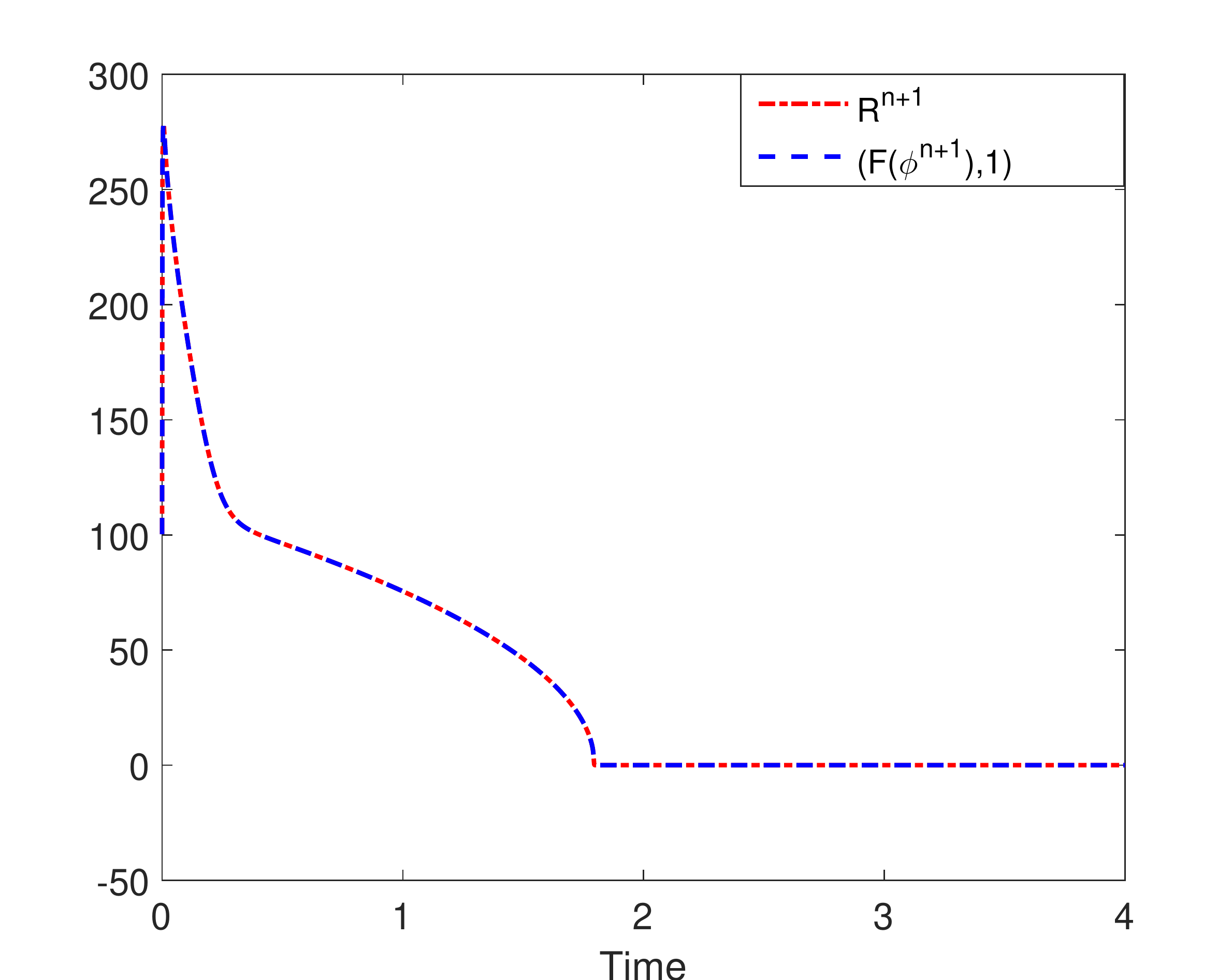}}
\subfigure[Numerical results of $\eta$]{
\includegraphics[width=5cm,height=5cm]{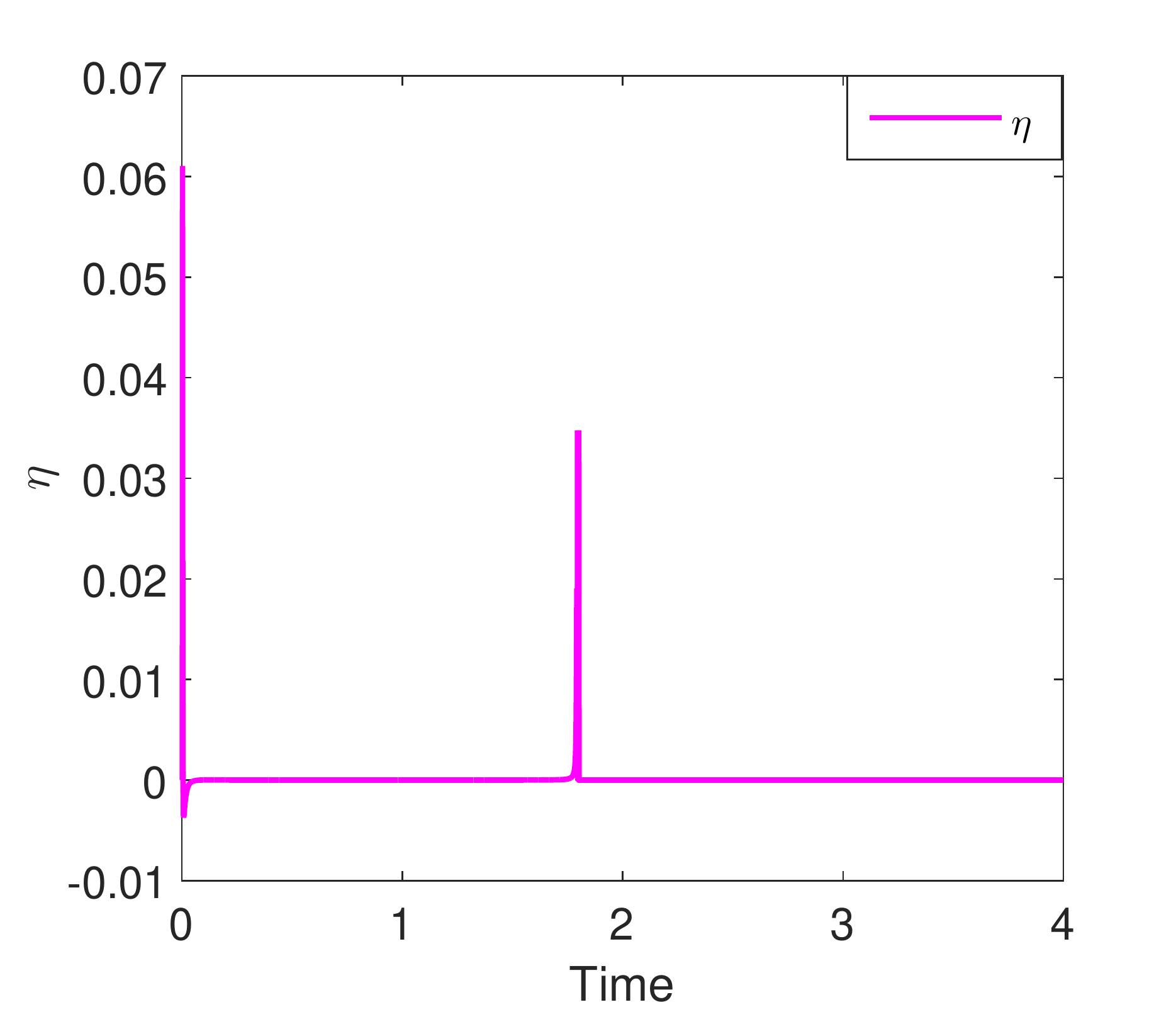}}
\caption{(a) the comparisons of numerical energies between $\mathcal{\widetilde{E}}(\phi)$ and the original energy $E(\phi)$, (b) the numerical results of $R^{n+1}$ and $\left((F(\phi^{n+1}),1\right)$ and (c) the evolution of $\eta(t)$ of Allen-Cahn model with the initial condition \eqref{AC-initial-e1}.}\label{fig:fig3-2}
\end{figure}
\begin{figure}[htp]
\centering
\includegraphics[width=4cm,height=3.5cm]{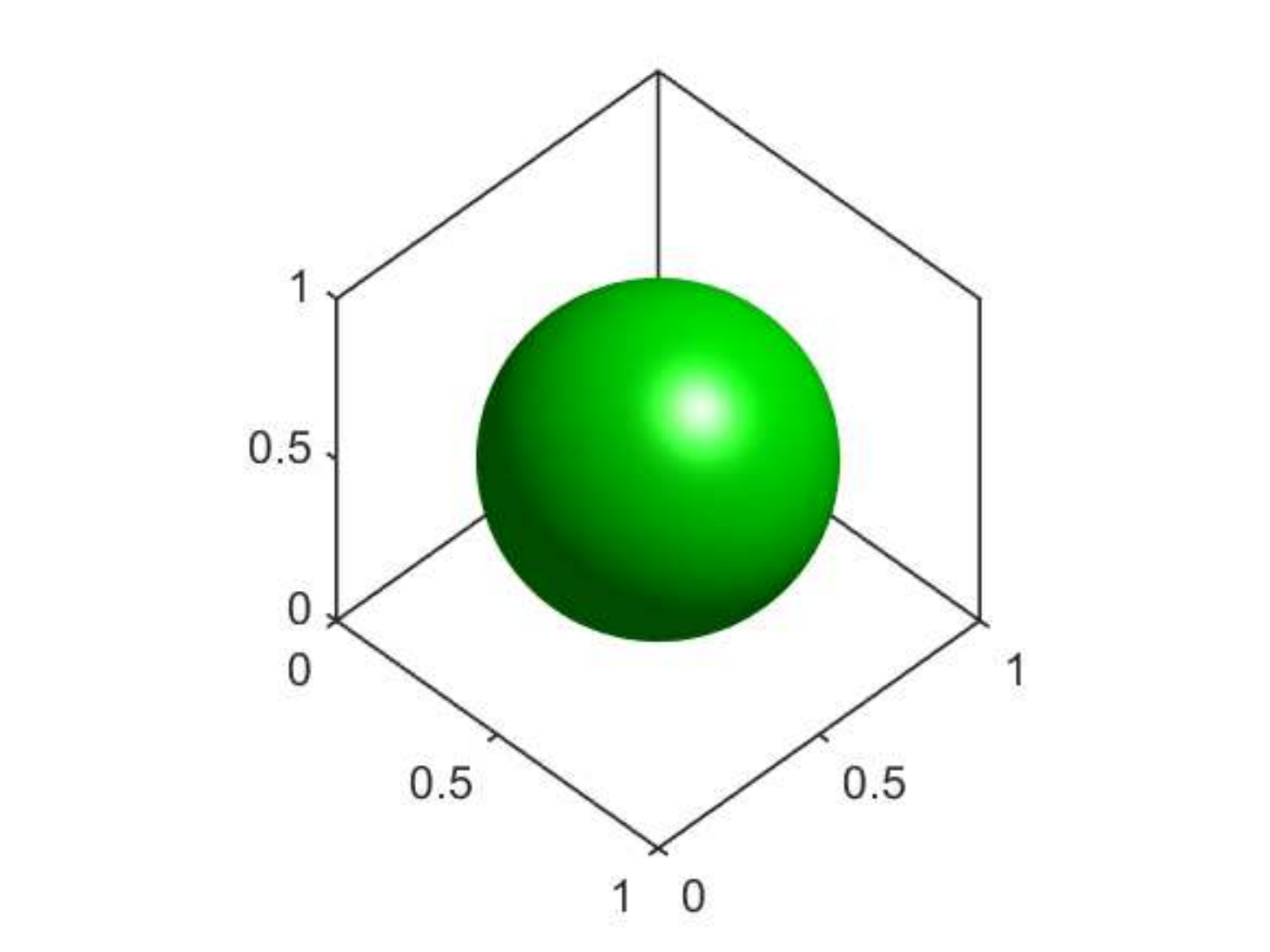}
\includegraphics[width=4cm,height=3.5cm]{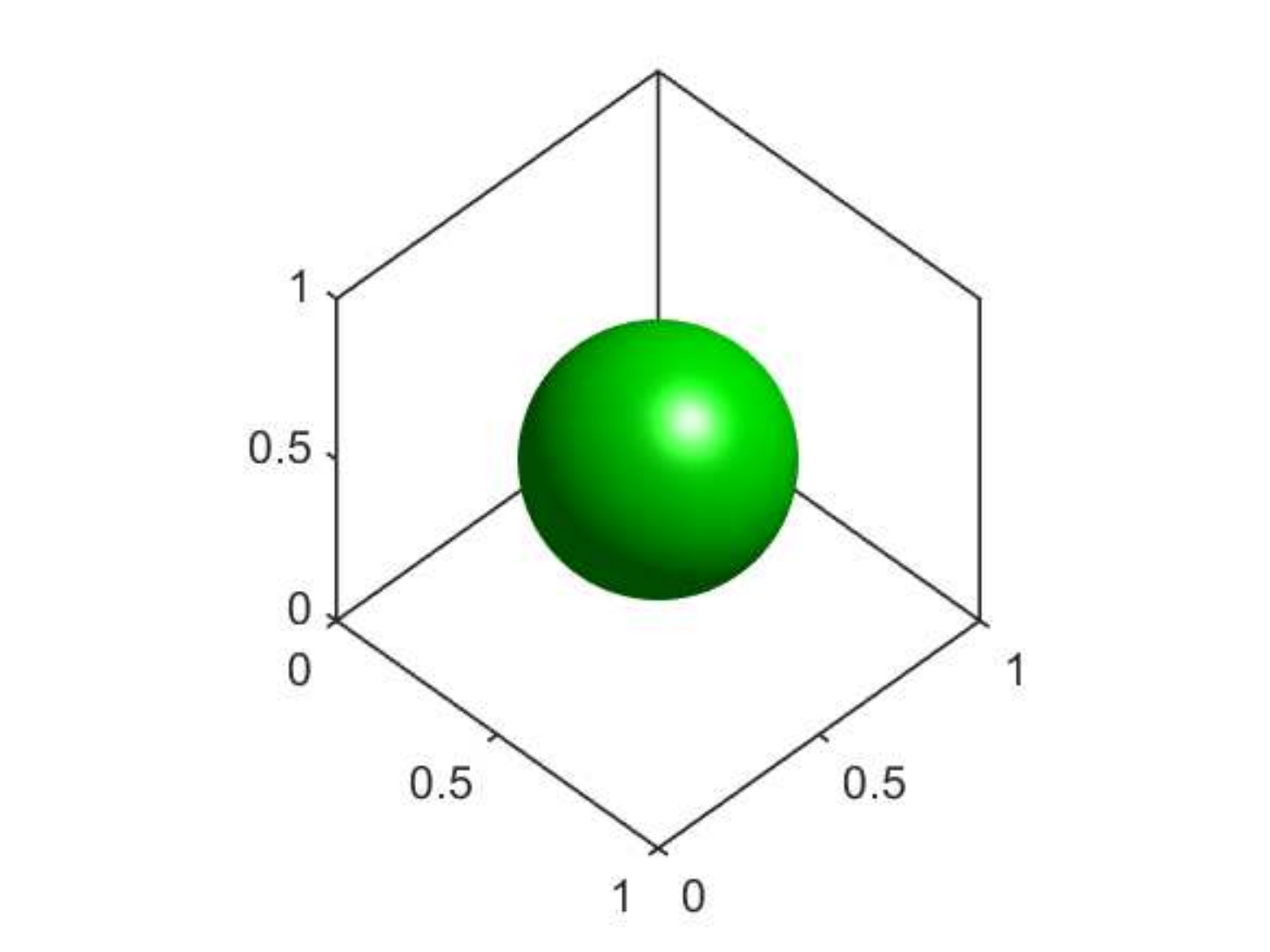}
\includegraphics[width=4cm,height=3.5cm]{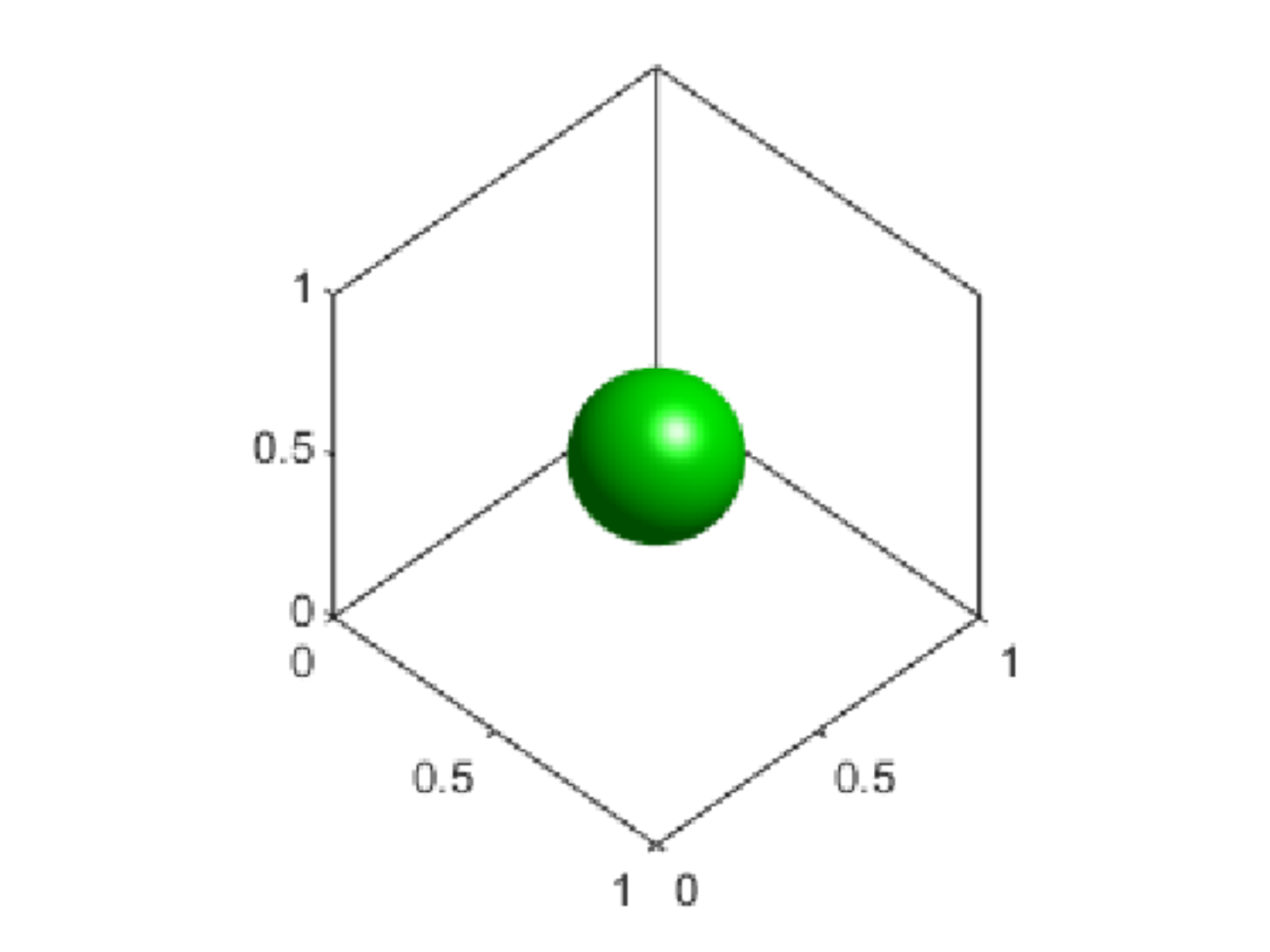}
\includegraphics[width=4cm,height=3.5cm]{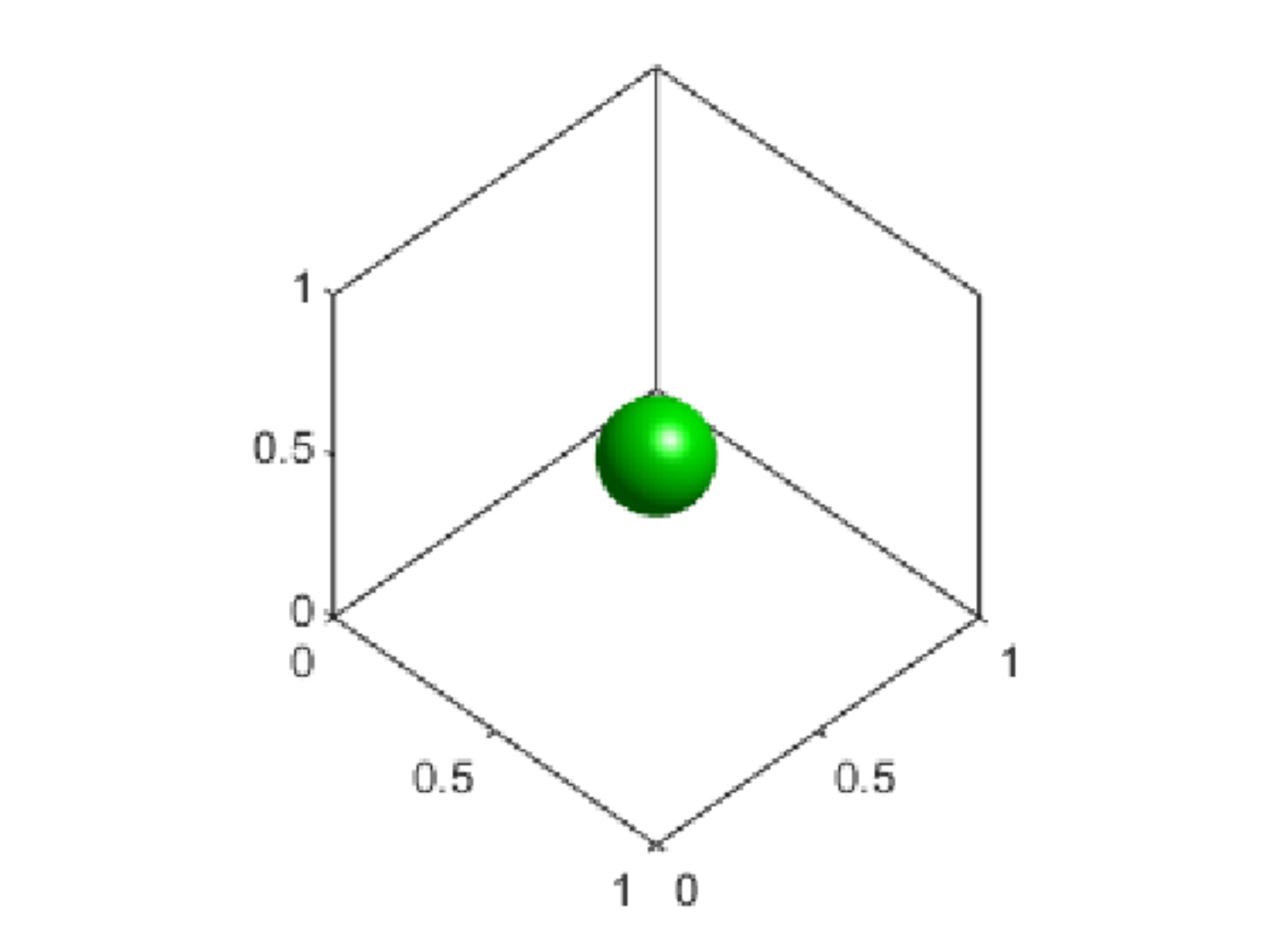}
\caption{Snapshots of zero level set to the numerical solutions of Allen-Cahn model with the initial condition \eqref{AC-initial-e2} at t=0, 2, 3, 3.5.}\label{fig:fig3-3}
\end{figure}
\subsection{Cahn-Hilliard model}
Consider the following well-known Cahn-Hilliard model:
\begin{equation}\label{section5_CH_model}
\displaystyle\frac{\partial \phi}{\partial t}=M\Delta\left(-\epsilon^2\Delta\phi+\phi^3-\phi\right),\quad(\textbf{x},t)\in\Omega\times [0,T],
\end{equation}
with the free energy
\begin{equation}\label{section5_CH-e1}
E(\phi)=\int_\Omega\frac{\epsilon^2}{2}|\nabla\phi|^2+\frac14(\phi^2-1)^2d\textbf{x},
\end{equation}
To give a more efficient simulation, we specify  $F(\phi)=\frac{1}{4}(\phi^2-1-\beta)^2$. Then the above Cahn-Hilliard equation is obtained as
\begin{equation}\label{section5_CH_e2}
\displaystyle\frac{\partial \phi}{\partial t}=M\Delta\left(-\epsilon^2\Delta\phi+\beta\phi+F'(\phi)\right),\quad F'(\phi)=\phi(\phi^2-1-\beta),\quad(\textbf{x},t)\in\Omega\times [0,T].
\end{equation}

Next, we investigated the numerical tests by the proposed RZF method to the Cahn-Hilliard model with the following initial conditions:
\begin{eqnarray}
&&\phi(x,y,0)=0.01\cos(2\pi x)\cos(2\pi y),\label{CH-initial-e1}\\
&&\phi(x,y,z,0)=1-\sum\limits_{i=1}^2\tanh\left(\frac{\sqrt{(x-x_i)^2+(y-y_i)^2+(z-z_i)^2}-r}{\sqrt{2}\epsilon}\right),\label{CH-initial-e2}\\
&&\phi(x,y,z,0)=0.05\text{rand}(x,y,z),\label{CH-initial-e3}
\end{eqnarray}
where "rand" implies a random number generating function ranged from -1 to 1.

We set the initial condition as in \eqref{CH-initial-e1} to check the convergence rates of our proposed schemes. we adopt uniform meshes $N_x=N_y=128$, and set $T=1$, $\epsilon=0.4$ $M=0.5$, $\beta=2$ and $\mathcal{P}(\eta)=\eta$. Figure \ref{fig:fig4} shows the results of the errors and convergence rates for the RZF-CN and RZF-BDF2 schemes. Numerical results demonstrate the accuracy and efficiency of our proposed scheme.
\begin{figure}[htp]
\centering
\subfigure[Convergence test of the variable $\phi$]{
\includegraphics[width=8cm,height=8cm]{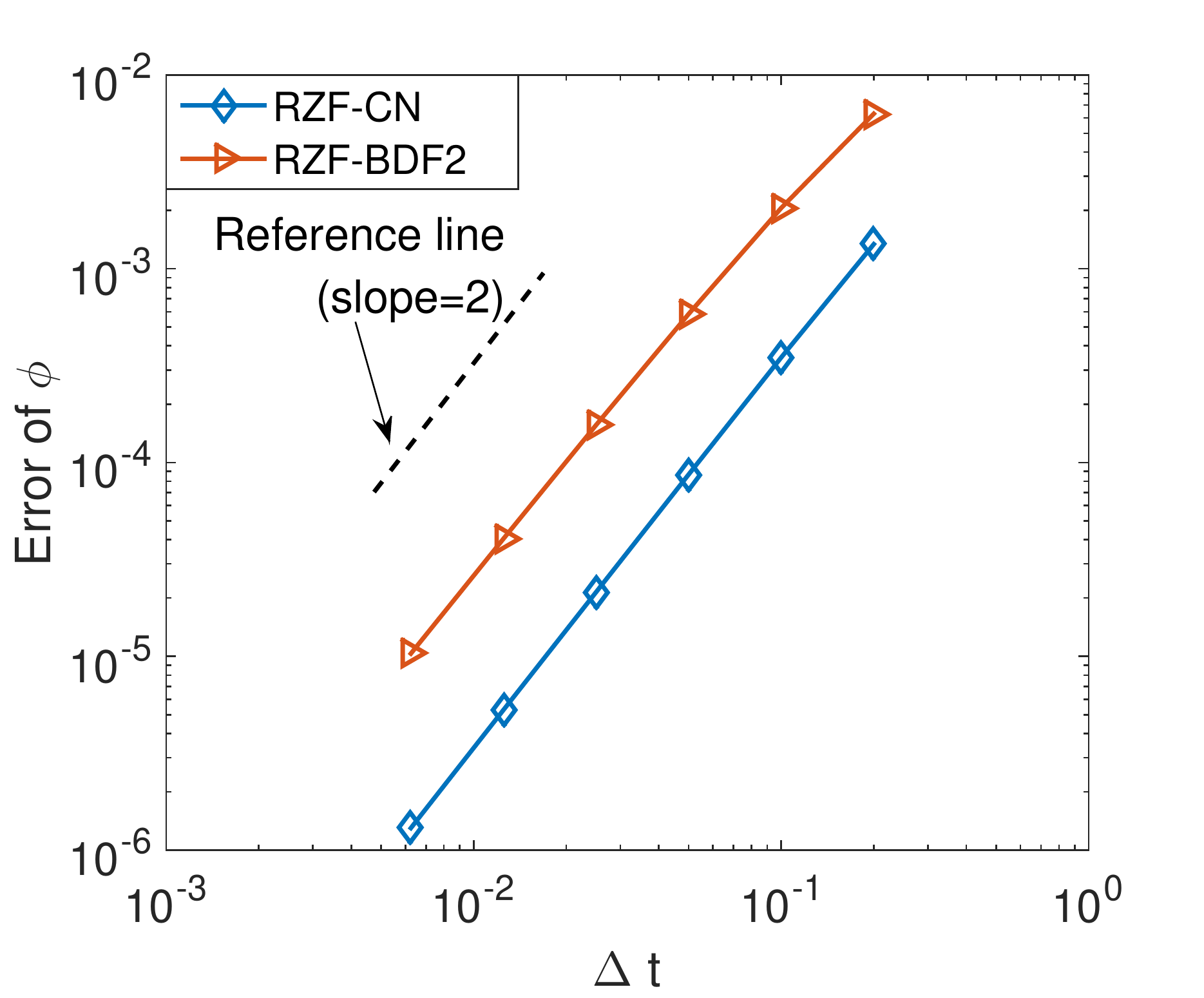}}
\subfigure[Convergence test of the variable $R$]{
\includegraphics[width=8cm,height=8cm]{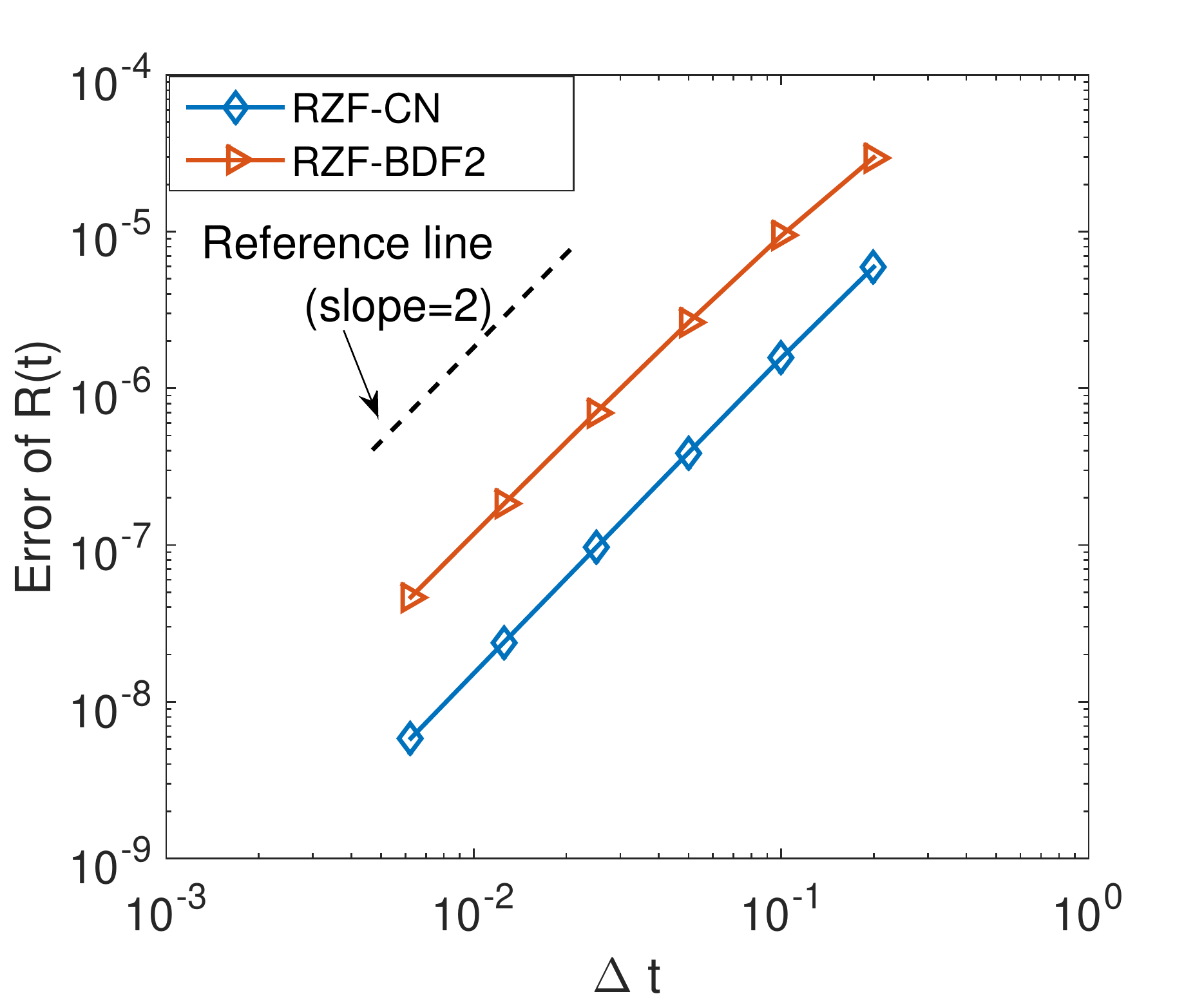}}
\caption{Convergence test for Cahn-Hilliard equation using RZF-CN and RZF-BDF2 schemes.}\label{fig:fig4}
\end{figure}

Next, we investigate the coarsening dynamics driven by the Cahn-Hilliard equation with the initial conditions \eqref{CH-initial-e2} and \eqref{CH-initial-e3}. We set $\Omega=[0,1]^3$, $\mathcal{P}(\eta)=\eta_t$, $\epsilon=0.01$, $M=1$, $r=0.14$, $(x_1,y_1,z_1)=(0.5,0.4,0.5)$, $(x_2,y_2,z_2)=(0.5,0.7,0.5)$. Here, we use $N_x=N_y=N_z=128$, $h=1/128$, and $\Delta t=1e-3$.  Figure \ref{fig:fig5} shows the numerical investigation results at $t=0$, $0.01$, $0.1$ and $0.2$ to the 3D Cahn-Hilliard model with the initial condition \eqref{CH-initial-e2}.  One can see that the initially separated spheres connect with each other gradually and finally merge into a big vesicle. The results are also consistent with those presented in \cite{cheng2020global}.  Figure \ref{fig:fig6} shows the numerical investigation results at $t=0$, $0.01$, $0.1$ and $0.2$ to the 3D Cahn-Hilliard model with the initial condition \eqref{CH-initial-e3}. The above results represent well the coarsening dynamics of the Cahn-Hilliard equation.
\begin{figure}[htp]
\centering
\includegraphics[width=4cm,height=3.5cm]{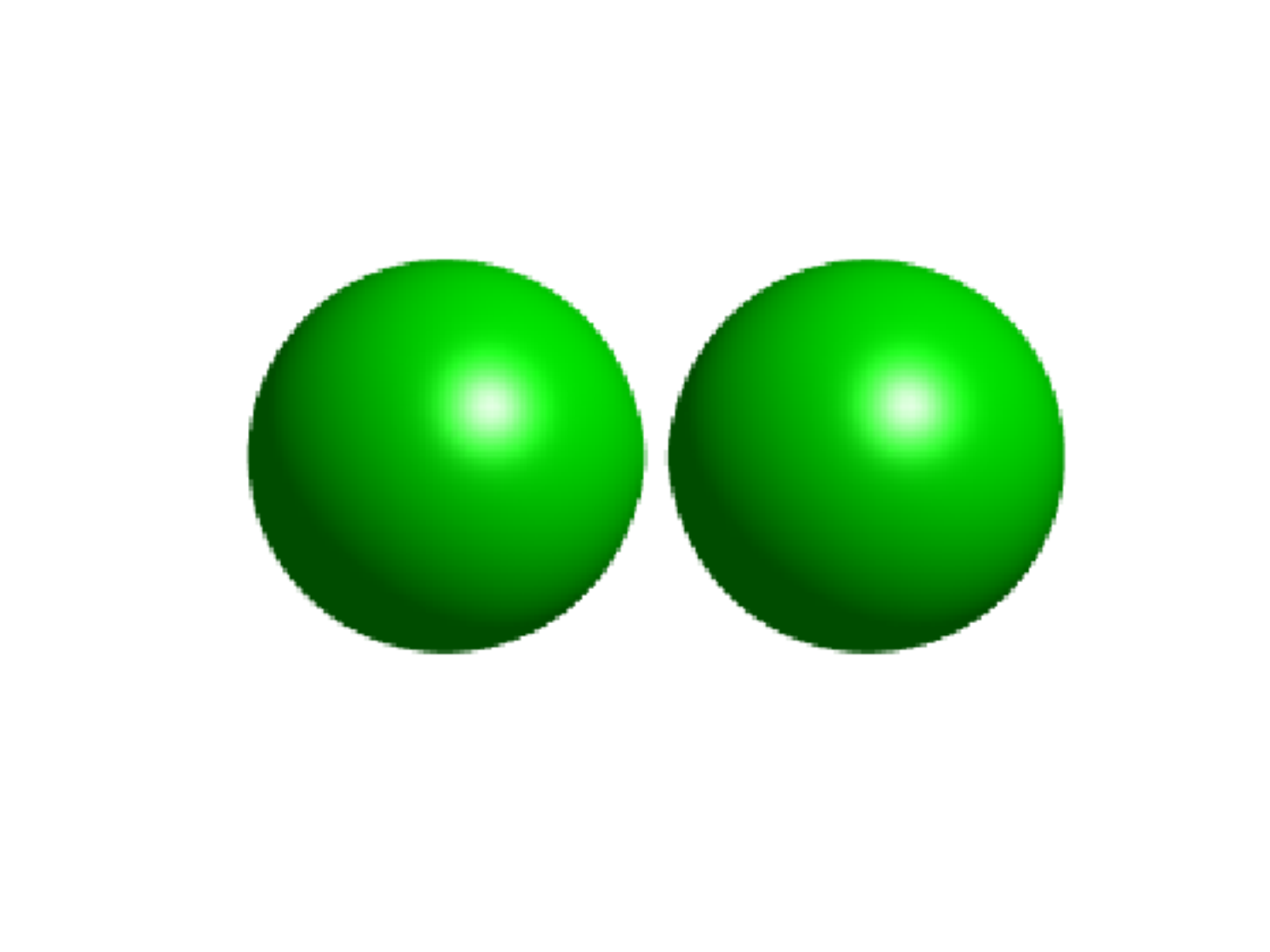}
\includegraphics[width=4cm,height=3.5cm]{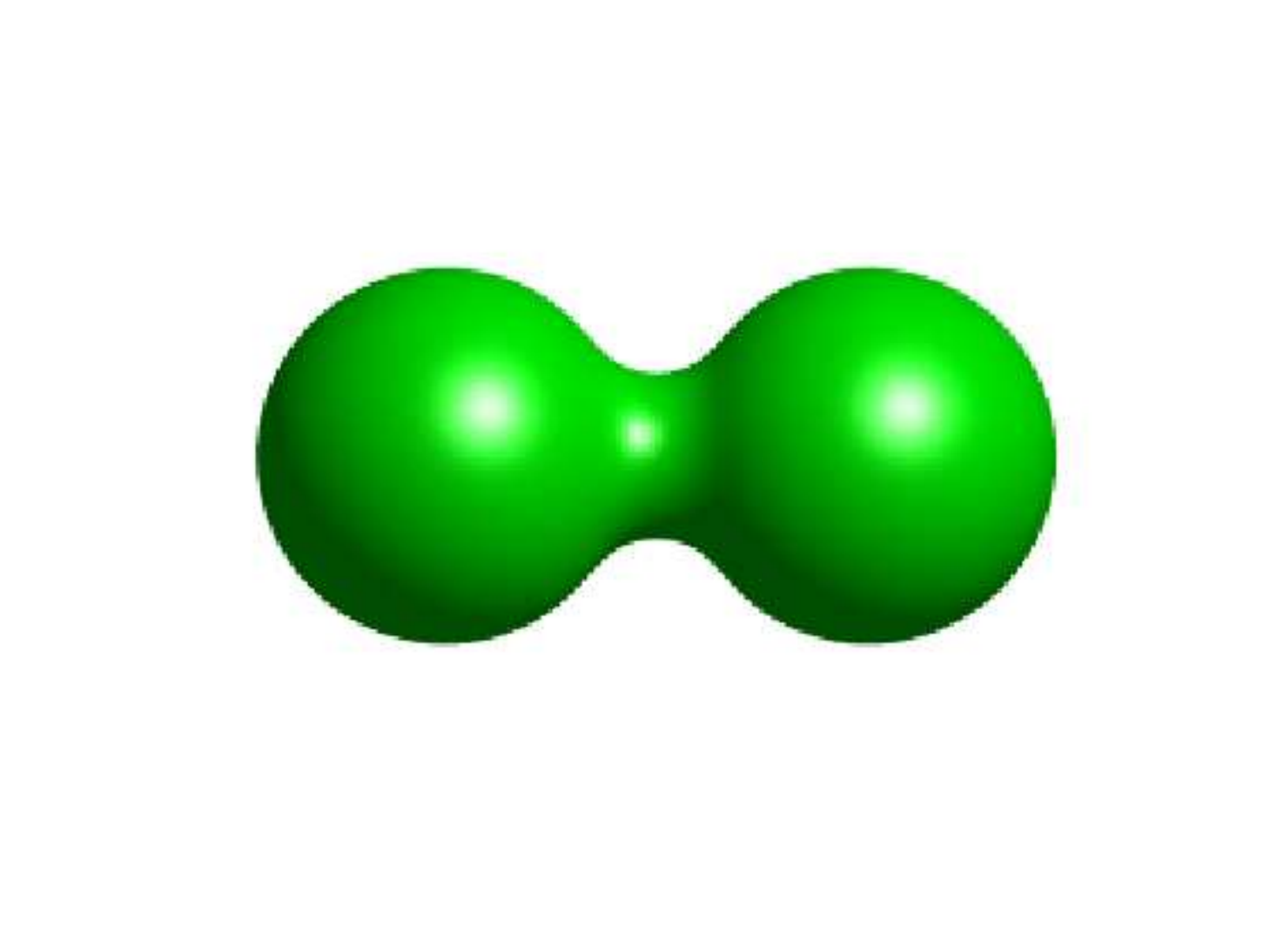}
\includegraphics[width=4cm,height=3.5cm]{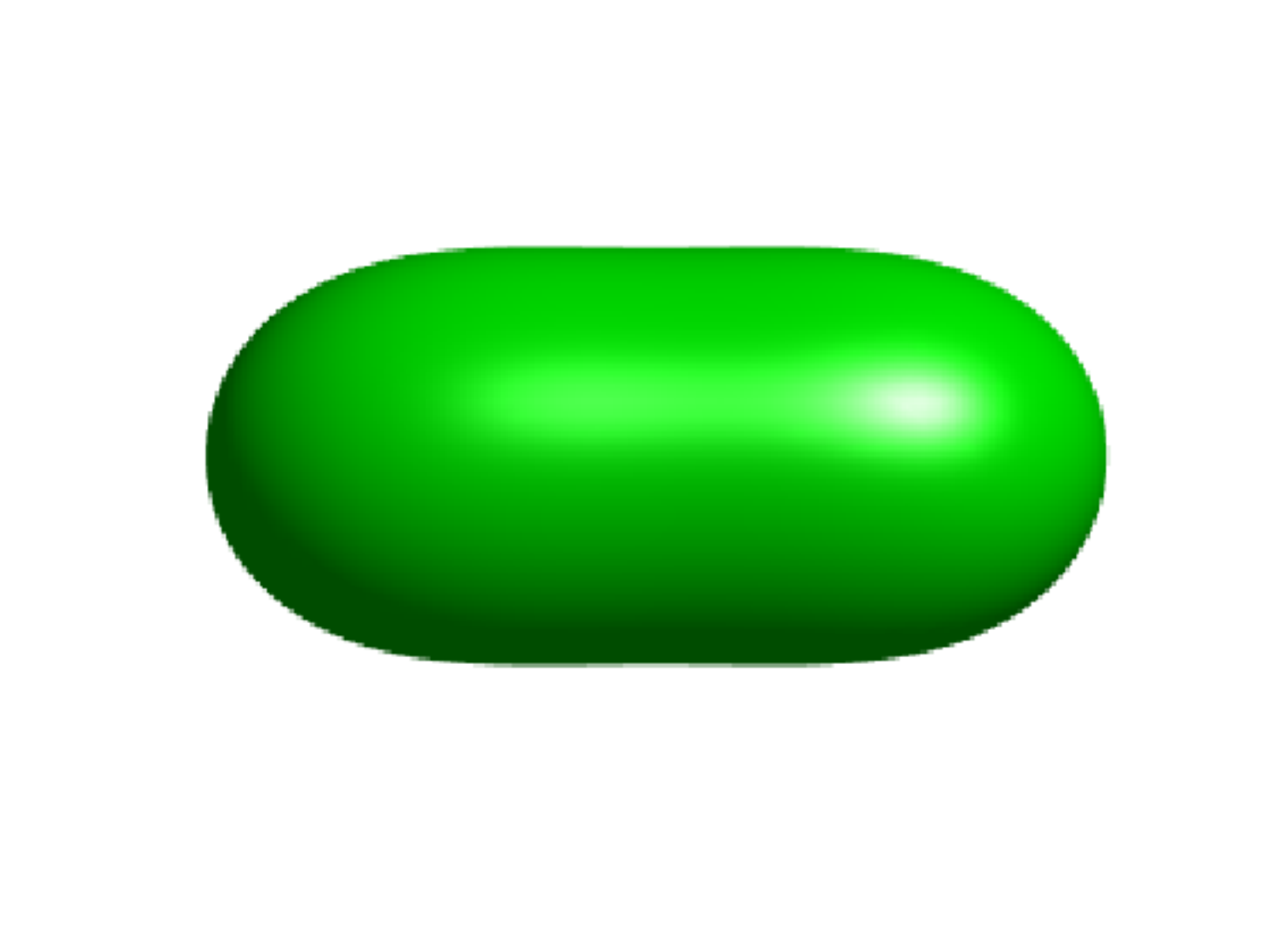}
\includegraphics[width=4cm,height=3.5cm]{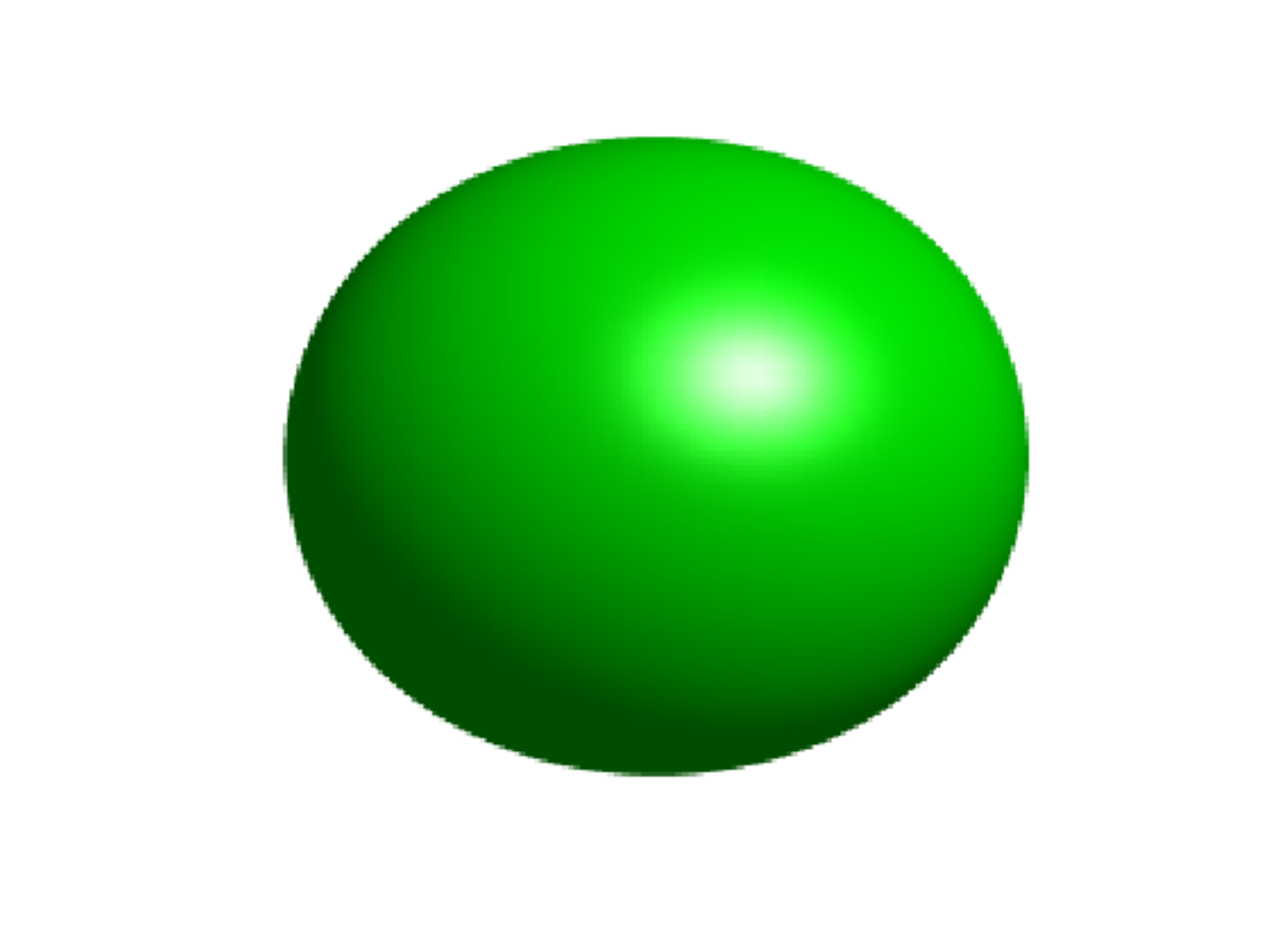}
\caption{Snapshots of zero level set to the numerical solutions of Cahn-Hilliard model with the initial condition \eqref{CH-initial-e2} using RZF-CN scheme at t=0, 0.01, 0.1, 0.2.}\label{fig:fig5}
\end{figure}
\begin{figure}[htp]
\centering
\includegraphics[width=8cm,height=5cm]{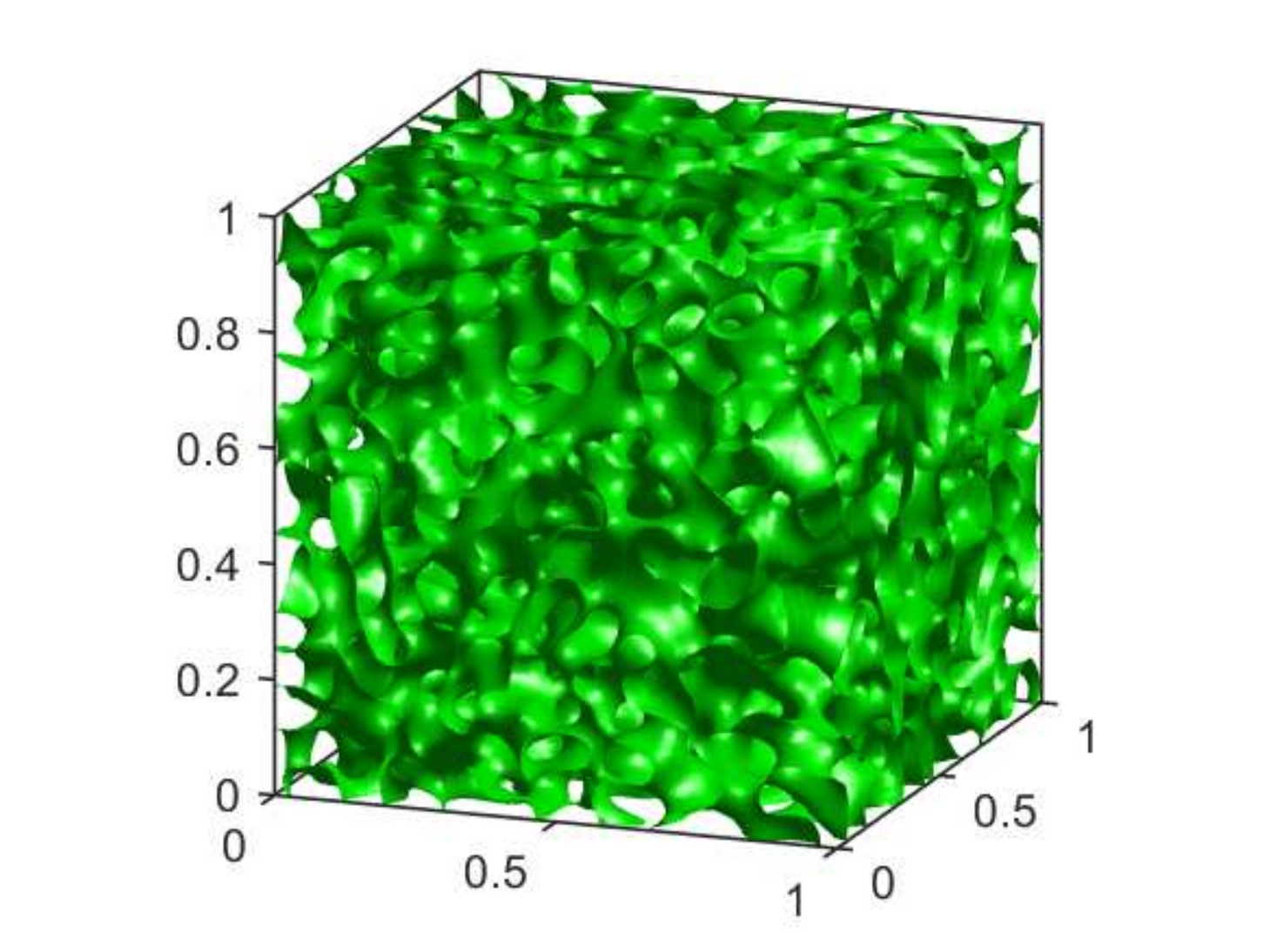}
\includegraphics[width=8cm,height=5cm]{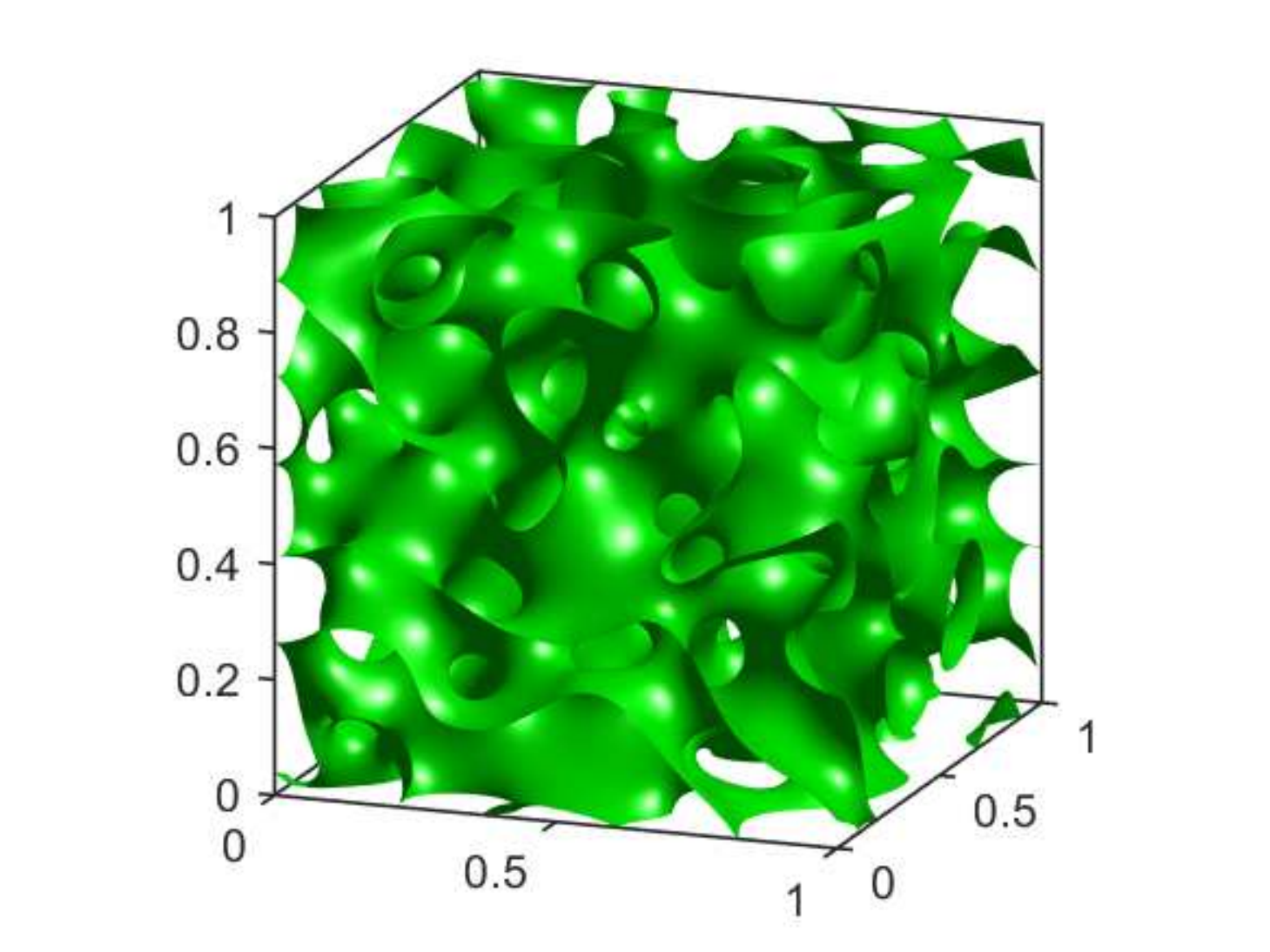}
\includegraphics[width=8cm,height=5cm]{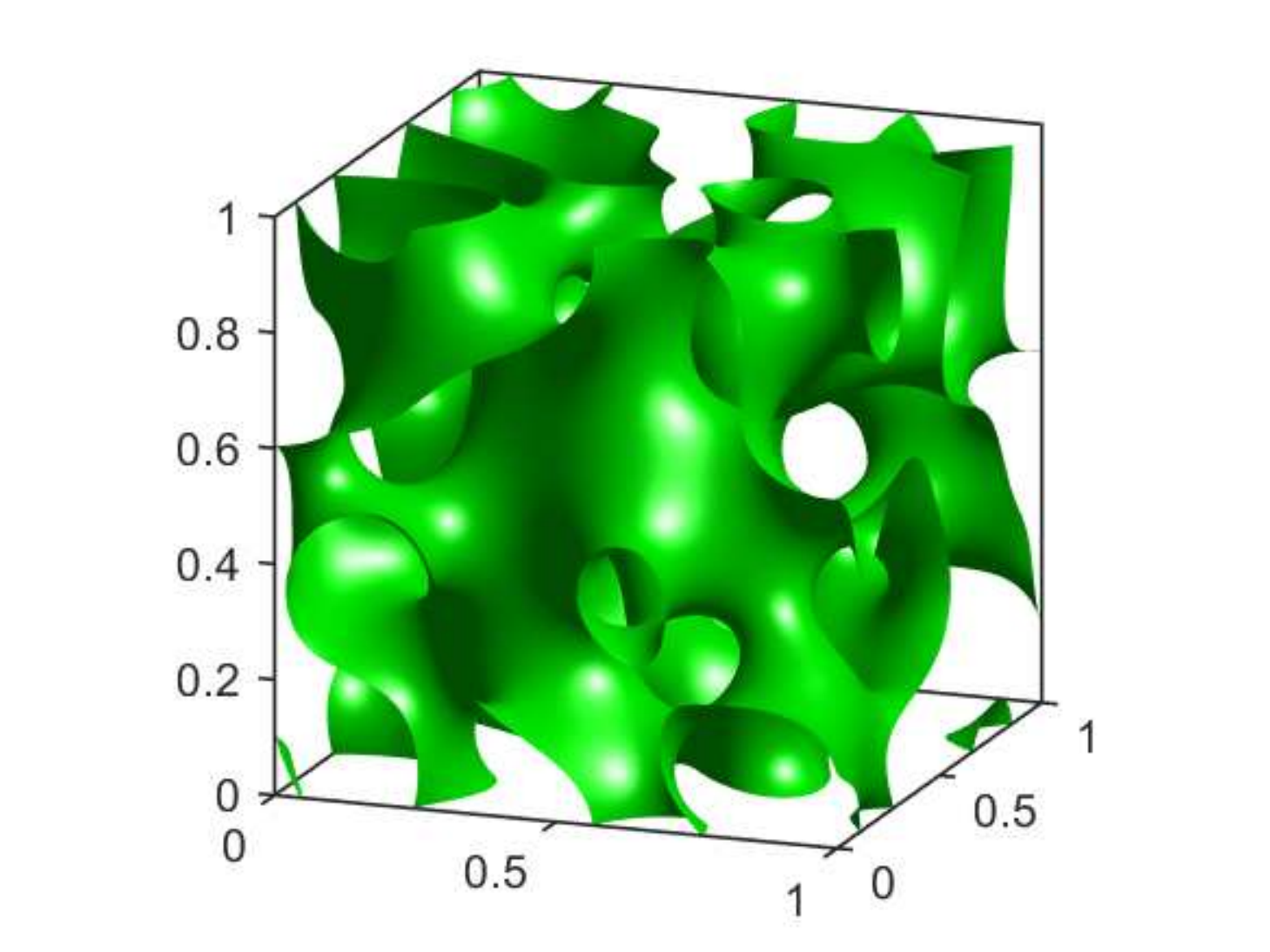}
\includegraphics[width=8cm,height=5cm]{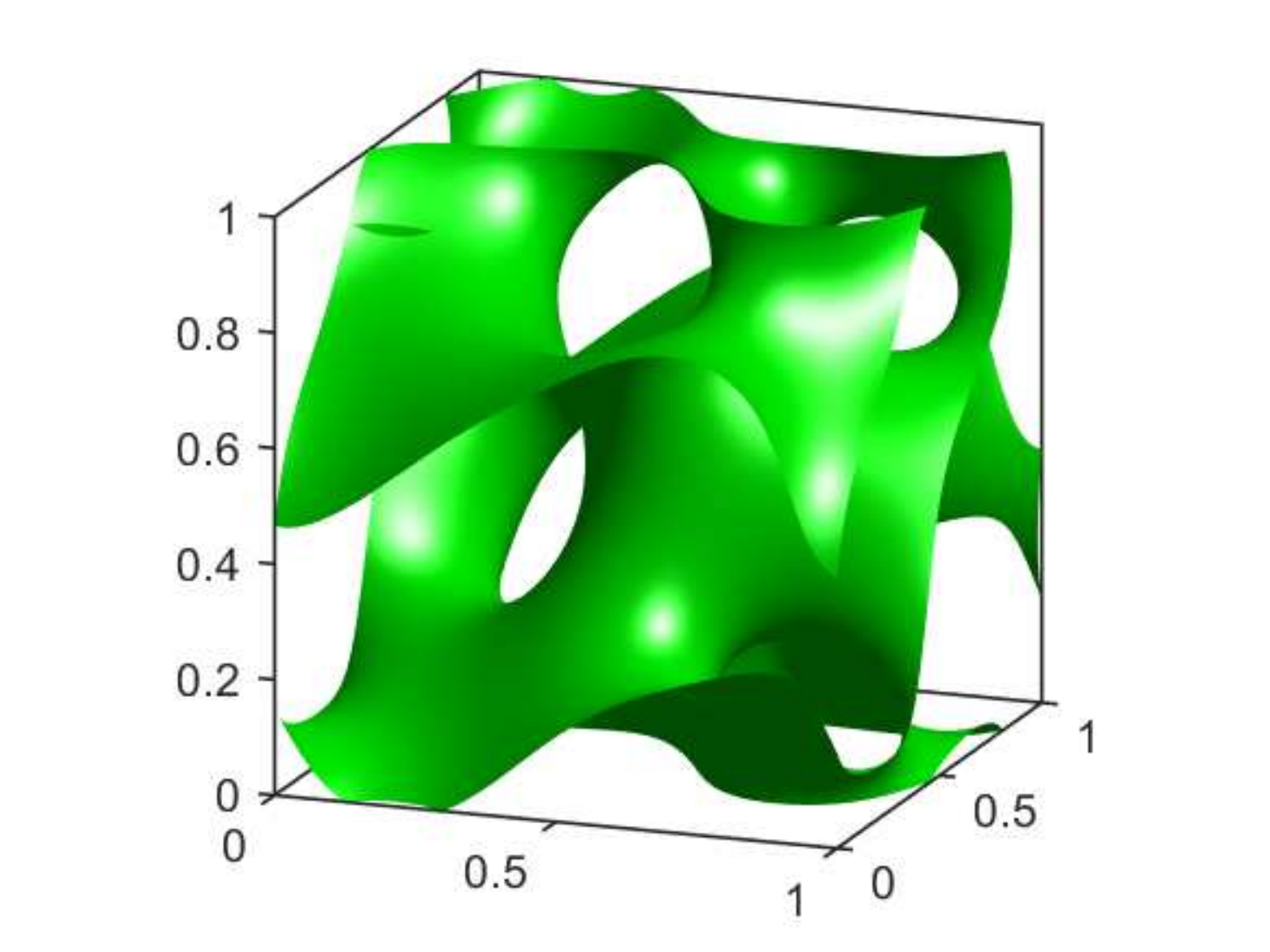}
\caption{Snapshots of zero level set to the numerical solutions of Cahn-Hilliard model with the initial condition \eqref{CH-initial-e3} using RZF-CN scheme at t=0.02, 0.1, 0.5, 2.}\label{fig:fig6}
\end{figure}
\subsection{Phase filed crystal model}
Consider the following Swift-Hohenberg free energy:
\begin{equation*}
E(\phi)=\int_{\Omega}\left(\frac{1}{4}\phi^4+\frac{1}{2}\phi\left(-\epsilon+(1+\Delta)^2\right)\phi\right)d\textbf{x},
\end{equation*}
where $\textbf{x} \in \Omega \subseteq \mathbb{R}^d$, $\phi$ is the density field, $g\geq0$ and $\epsilon>0$ are constants with physical significance, $\Delta$ is the Laplacian operator.

Considering a gradient flow in $H^{-1}$, one can obtain the following phase field crystal (PFC) model:
\begin{equation*}
\frac{\partial \phi}{\partial t}=\Delta\mu=\Delta\left(\phi^3-\epsilon\phi+(1+\Delta)^2\phi\right), \quad(\textbf{x},t)\in\Omega\times Q,
\end{equation*}
which is a sixth-order nonlinear parabolic equation and can be applied to simulate various phenomena such as crystal growth, material hardness and phase transition. Here $Q=(0,T]$, $\mu=\frac{\delta E}{\delta \phi}$ is called the chemical potential.

In the following, we simulate the benchmark simulation for the PFC model. we choose the initial condition $\phi_0(x,y)=\widehat{\phi}_0+0.01\times \text{rand}(x,y)$, where the $\text{rand}(x,y)$ is the random number in $[-1,1]$ with zero mean. In this test, set $\epsilon=0.325$ and adopt uniform meshes $N_x=N_y=128$ in the Fourier spectral method.

We show the phase transition behavior of the density field for different values at various times in Figures \ref{fig:fig7} with different $\widehat{\phi}_0$ and $\Omega$. We observe that for different $\widehat{\phi}_0$, the shape and rate of crystallization of crystals are different. In all cases, the process of the phase transition is qualitative agreement of the density fields. Similar computation results for phase field crystal model can be found in many articles such as in \cite{yang2017linearly}. The energy curves are plotted for PFC model with different initial conditions in Figure \ref{fig:fig8}. It is observed that the computed energy for all cases decays with time.
\begin{figure}[htp]
\centering
\subfigure[$\phi$ at t=10,30,50,100 with $\Omega=\text{[}0,100{]}^2$ and $\widehat{\phi}_0=0.25$]
{
\includegraphics[width=4cm,height=4cm]{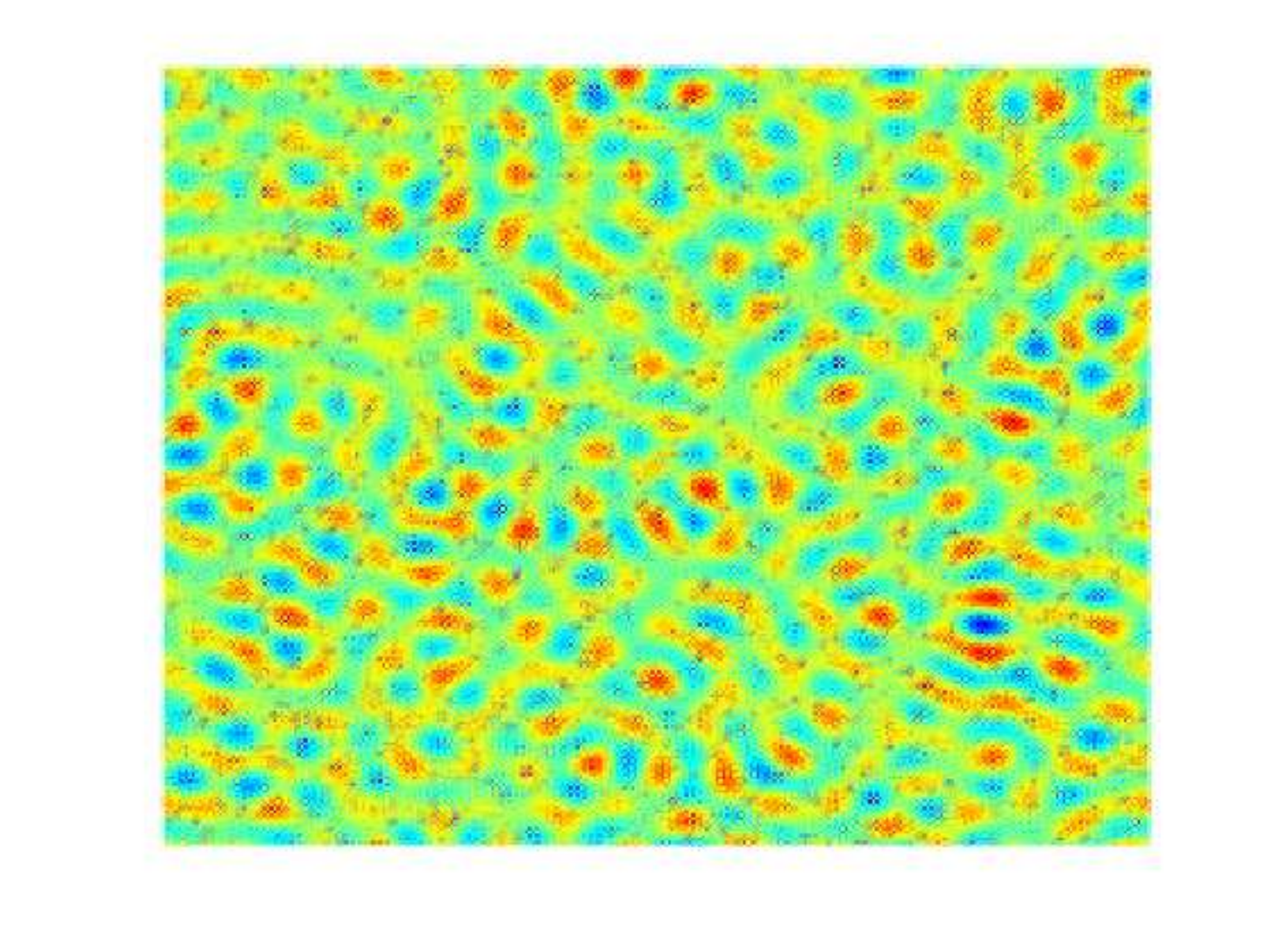}
\includegraphics[width=4cm,height=4cm]{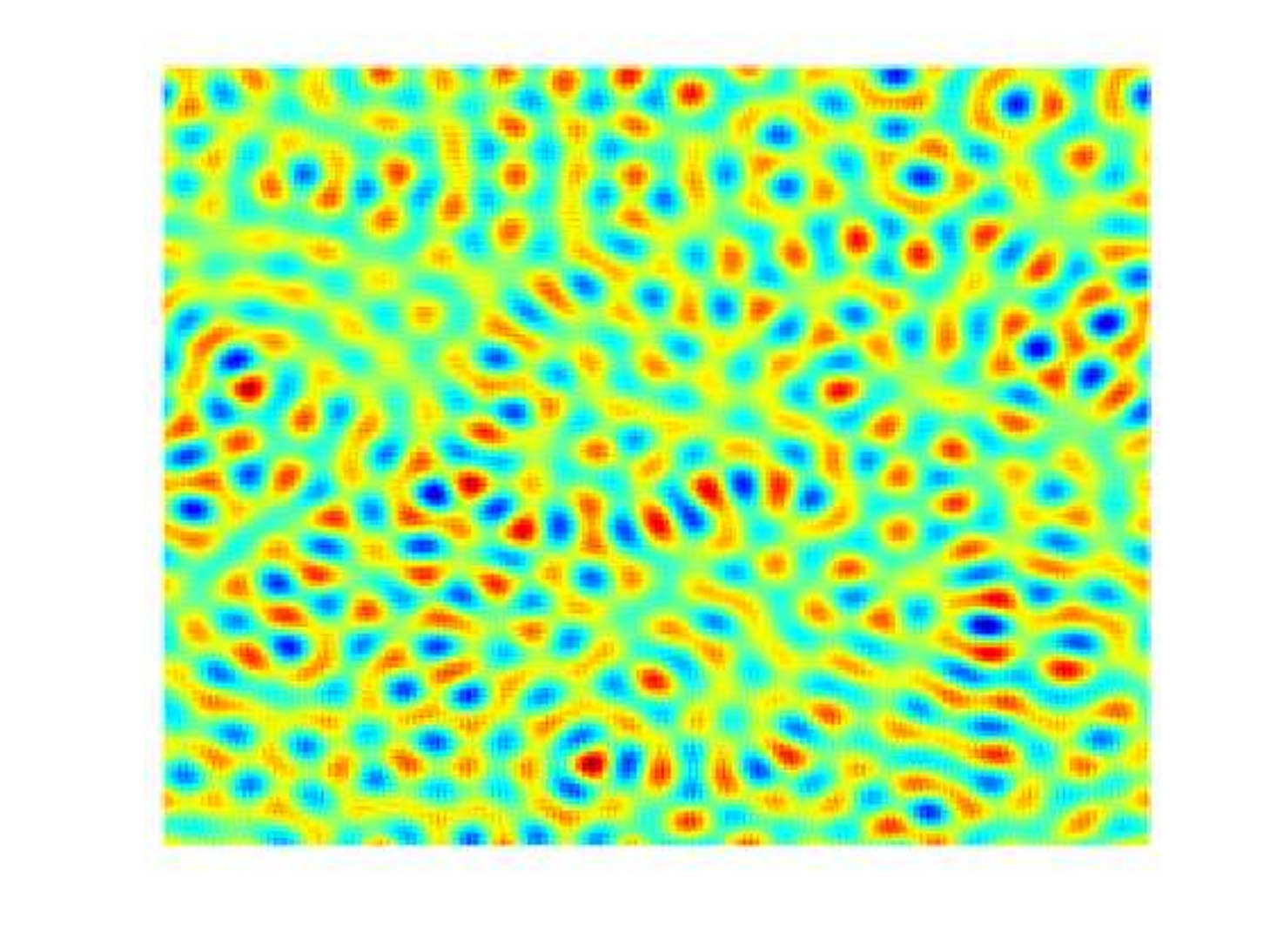}
\includegraphics[width=4cm,height=4cm]{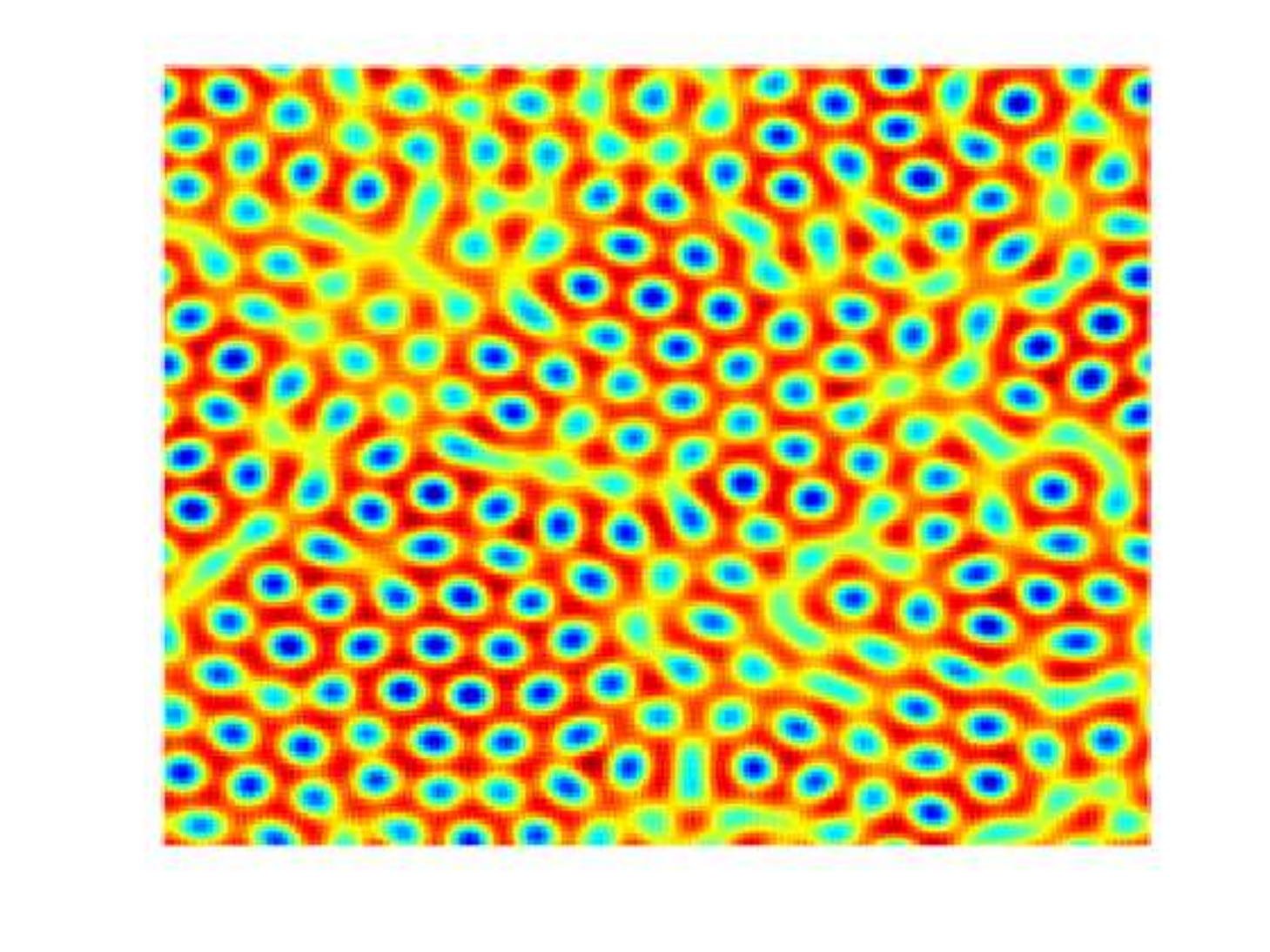}
\includegraphics[width=4cm,height=4cm]{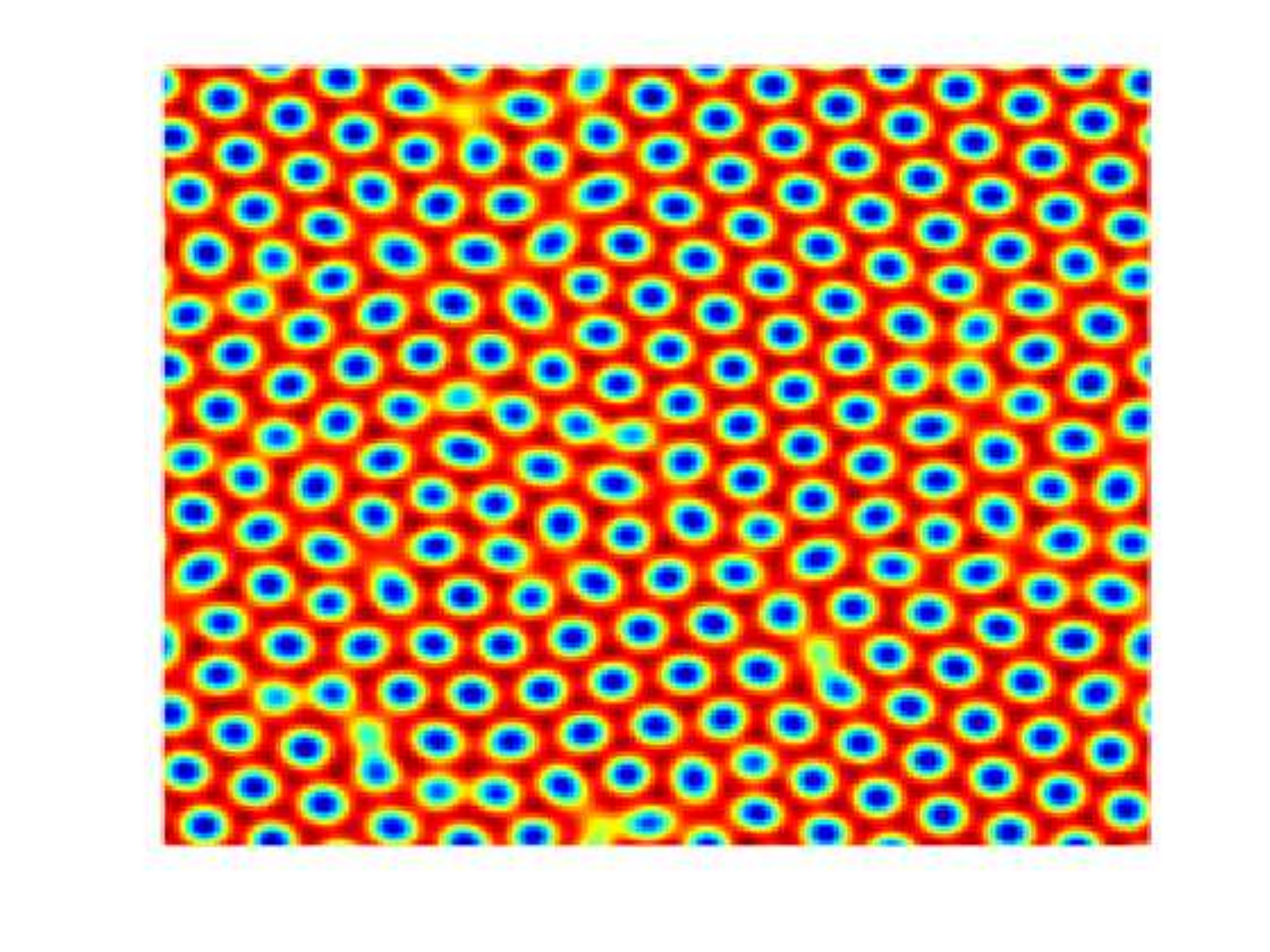}
}
\subfigure[$\phi$ at t=10,20,50,100 with $\Omega=\text{[}0,200{]}^2$ and $\widehat{\phi}_0=0$]
{
\includegraphics[width=4cm,height=4cm]{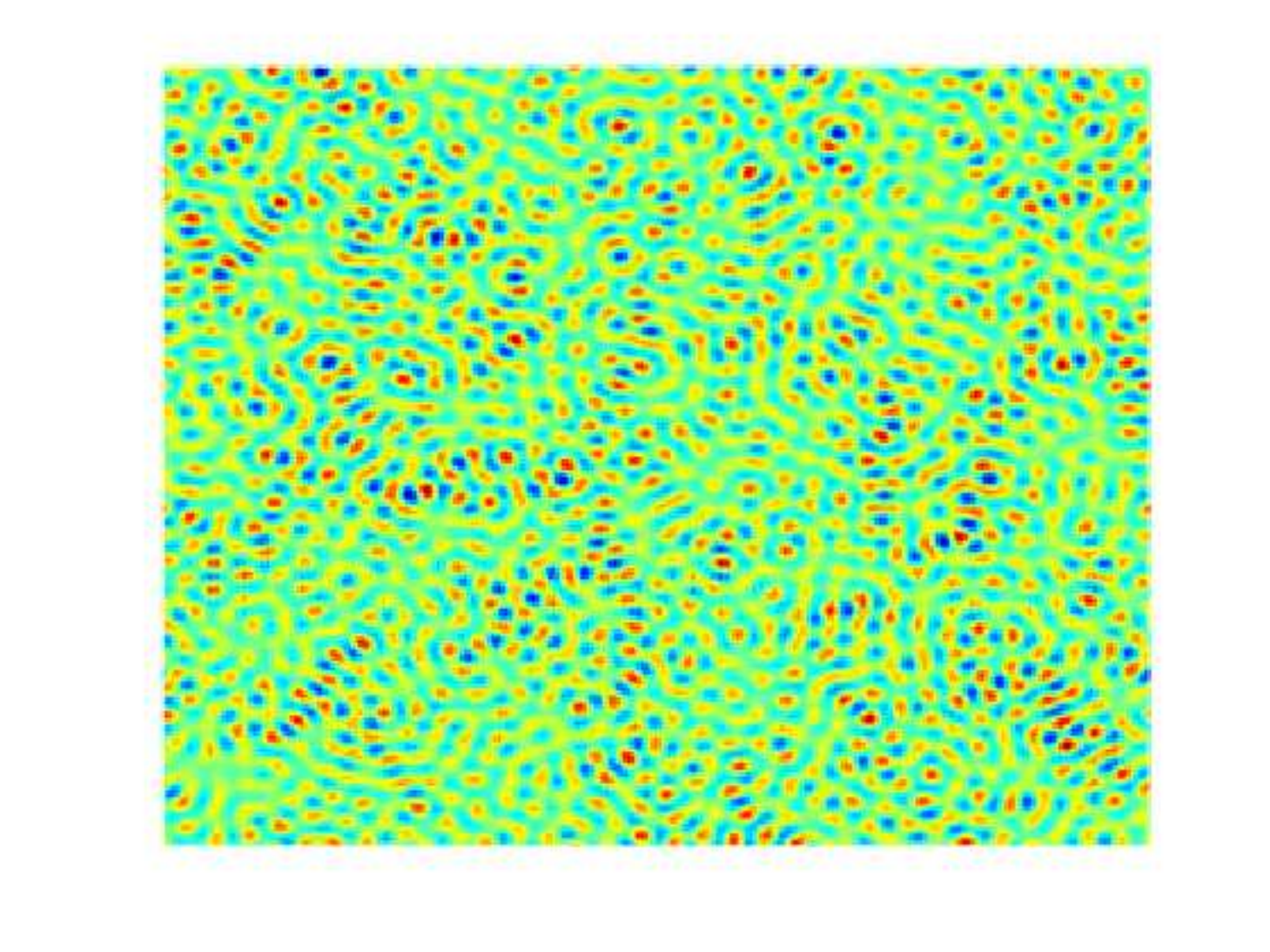}
\includegraphics[width=4cm,height=4cm]{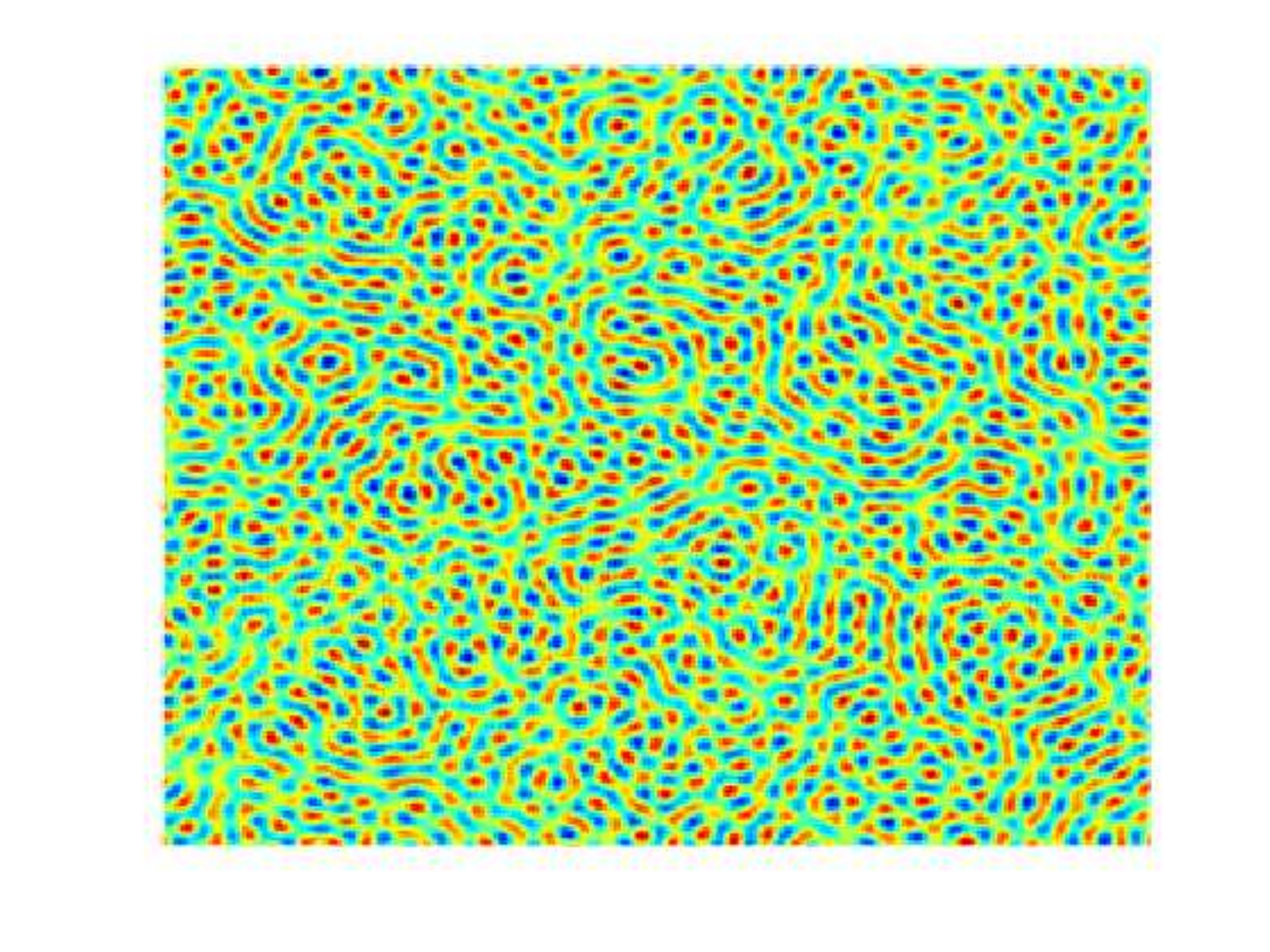}
\includegraphics[width=4cm,height=4cm]{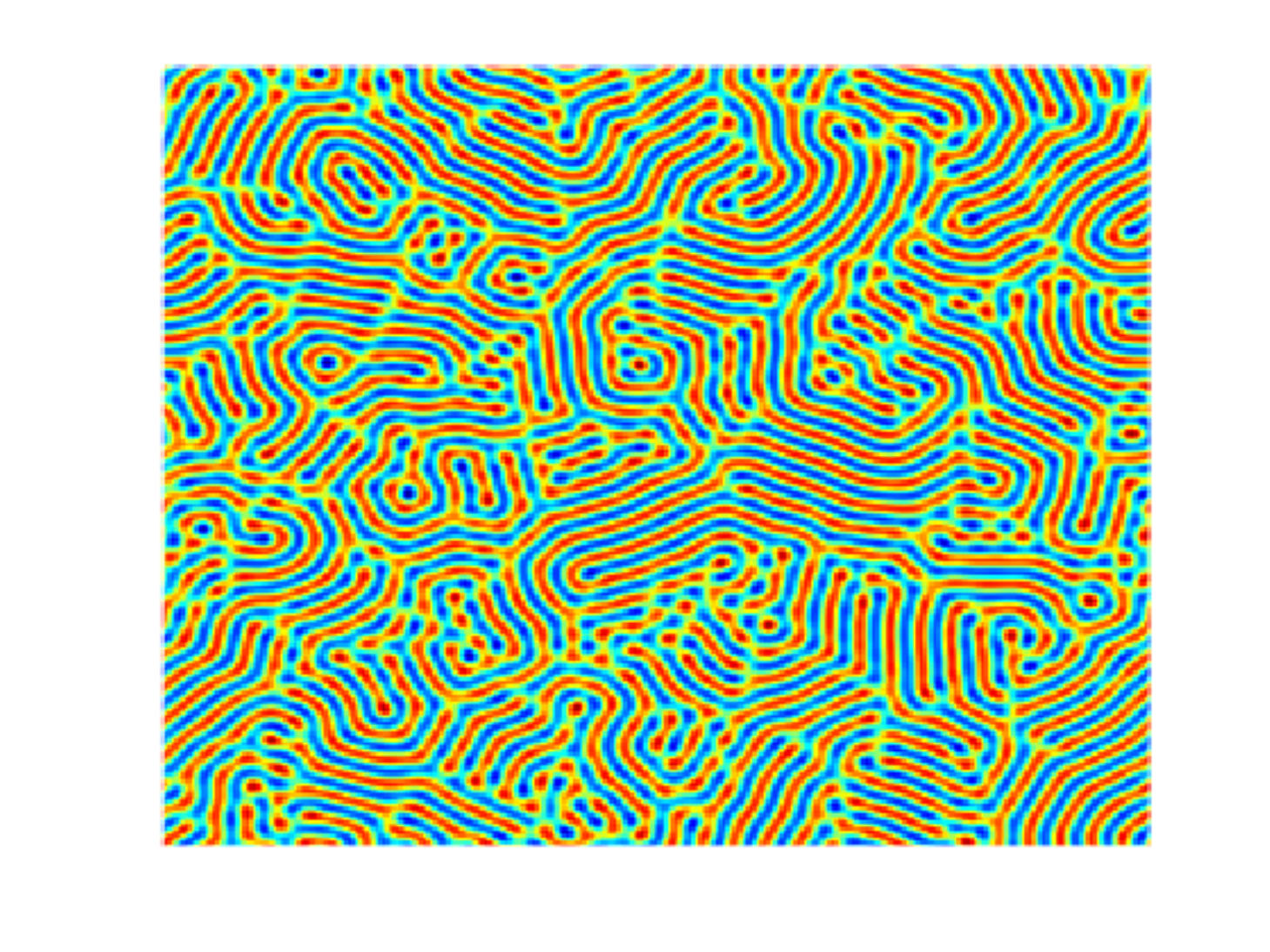}
\includegraphics[width=4cm,height=4cm]{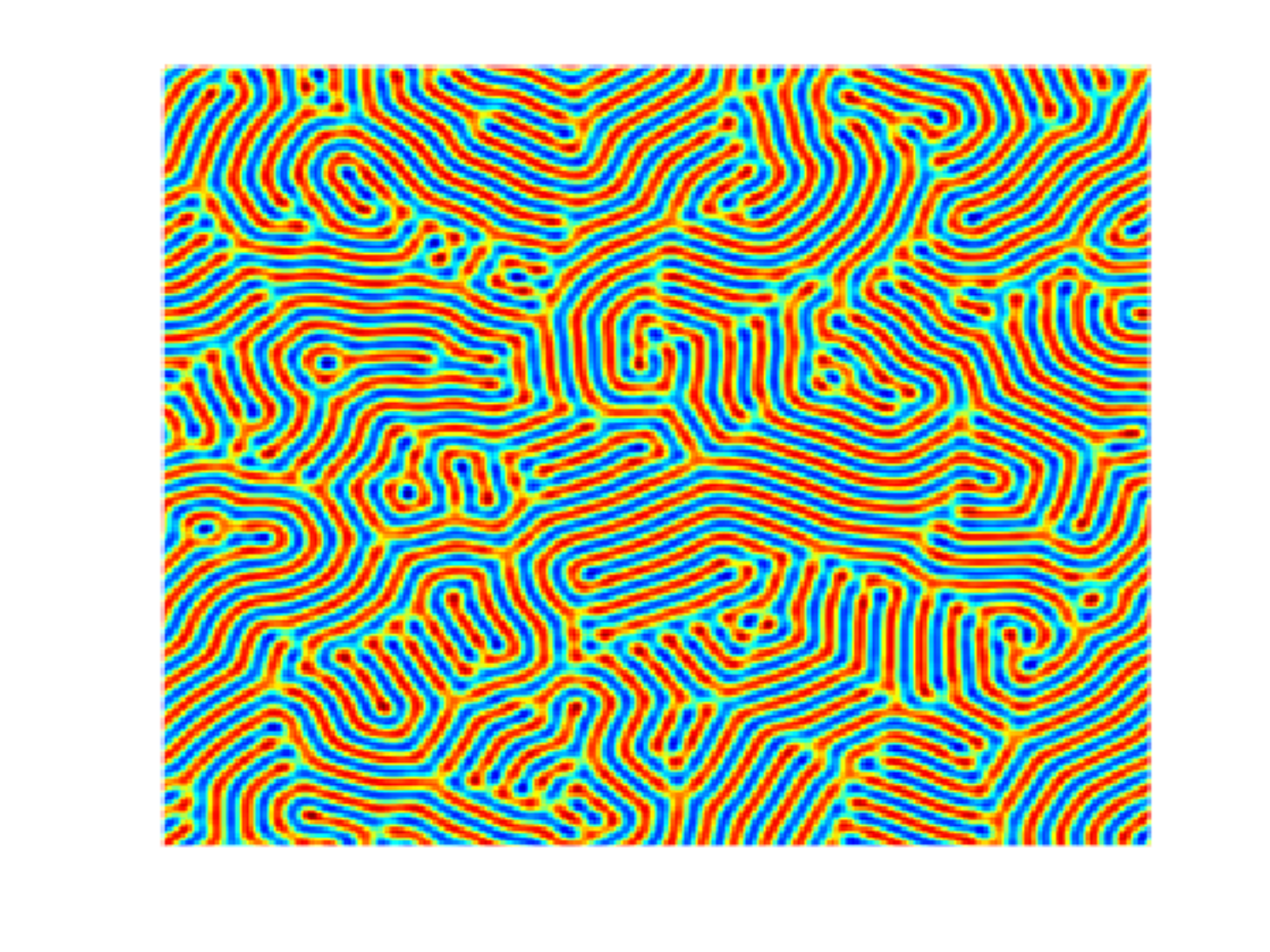}
}
\subfigure[$\phi$ at t=10,20,30,50 with $\Omega=\text{[}0,200{]}^2$ and $\widehat{\phi}_0=0.2$]
{
\includegraphics[width=4cm,height=4cm]{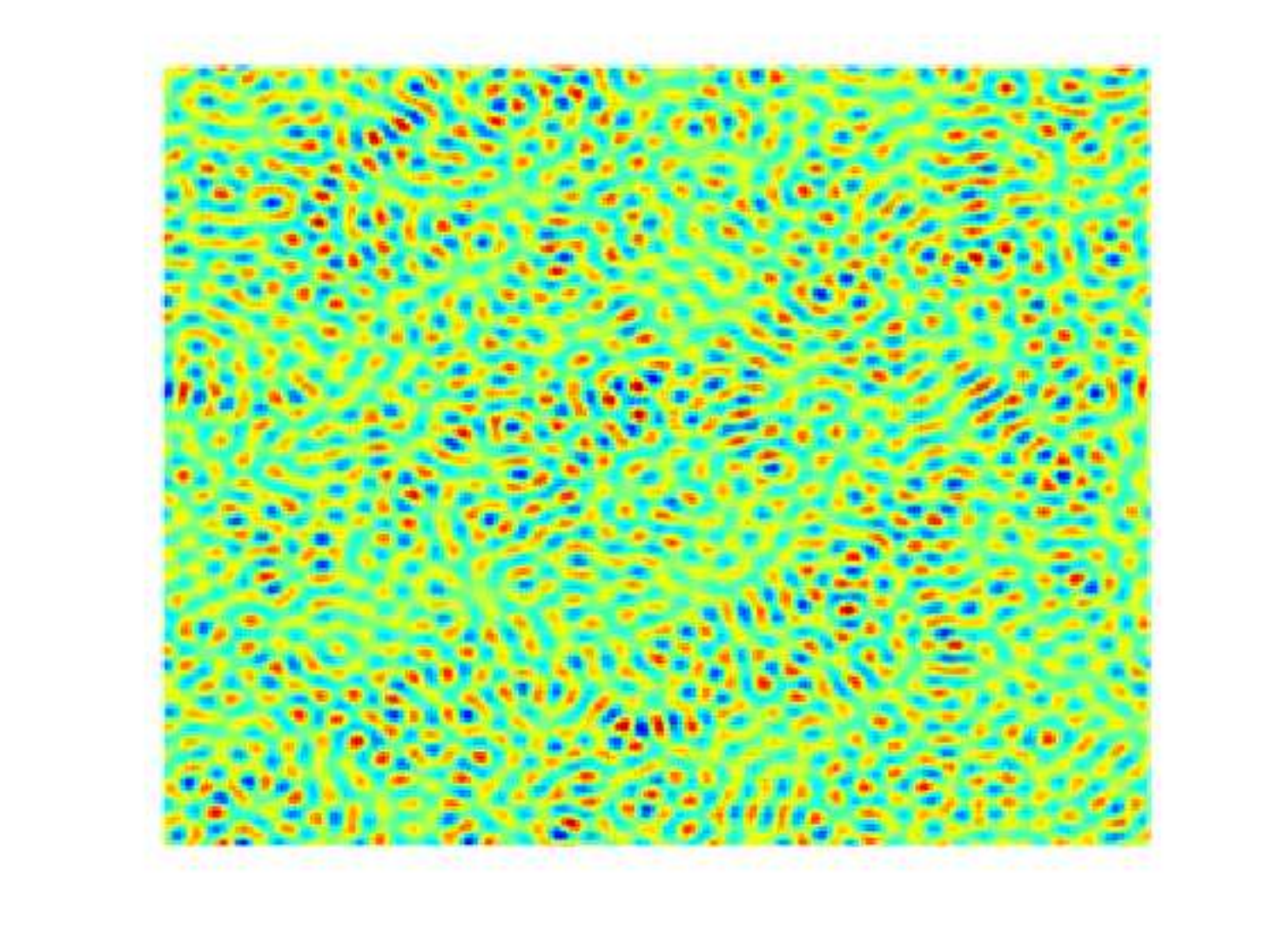}
\includegraphics[width=4cm,height=4cm]{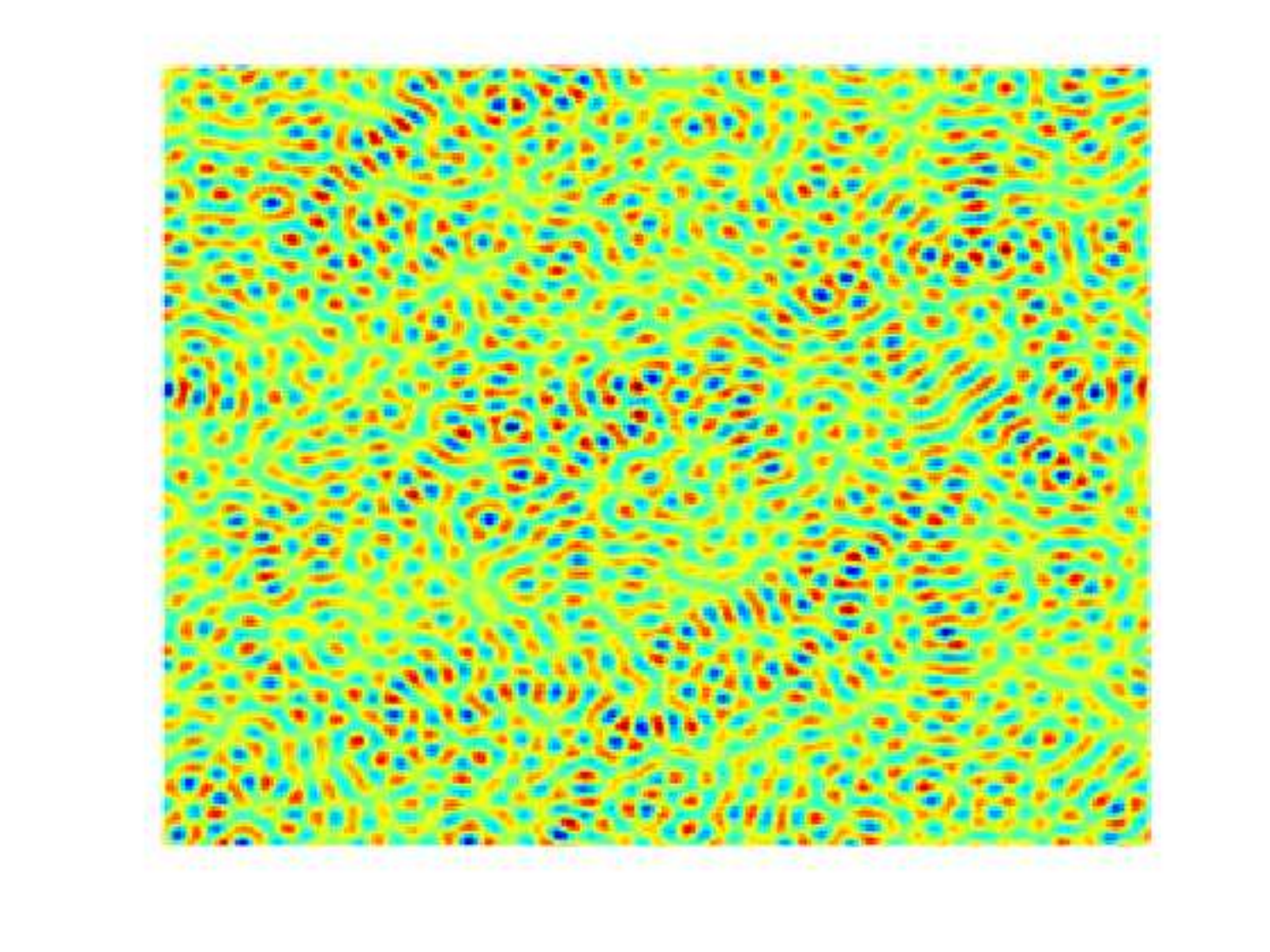}
\includegraphics[width=4cm,height=4cm]{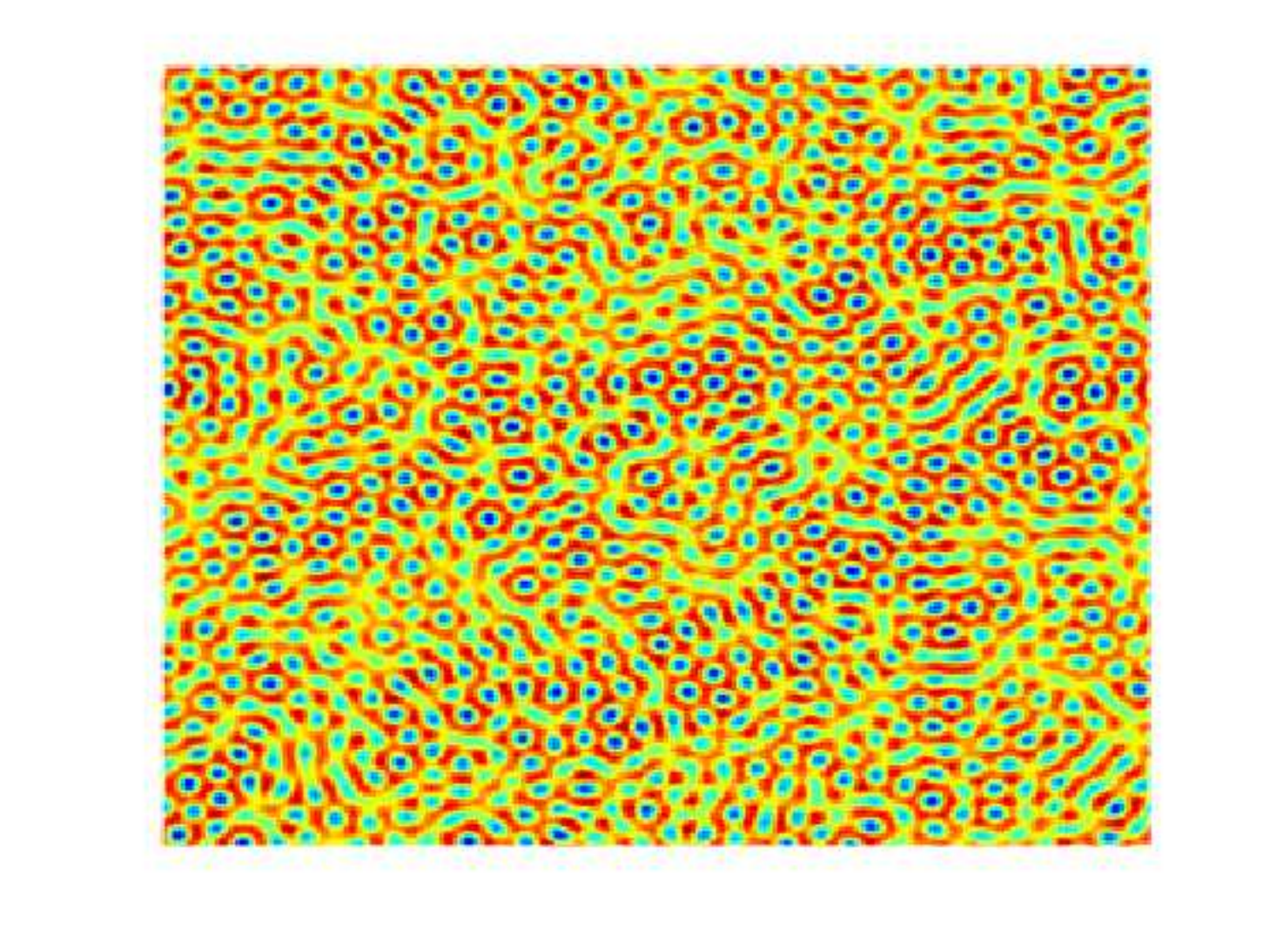}
\includegraphics[width=4cm,height=4cm]{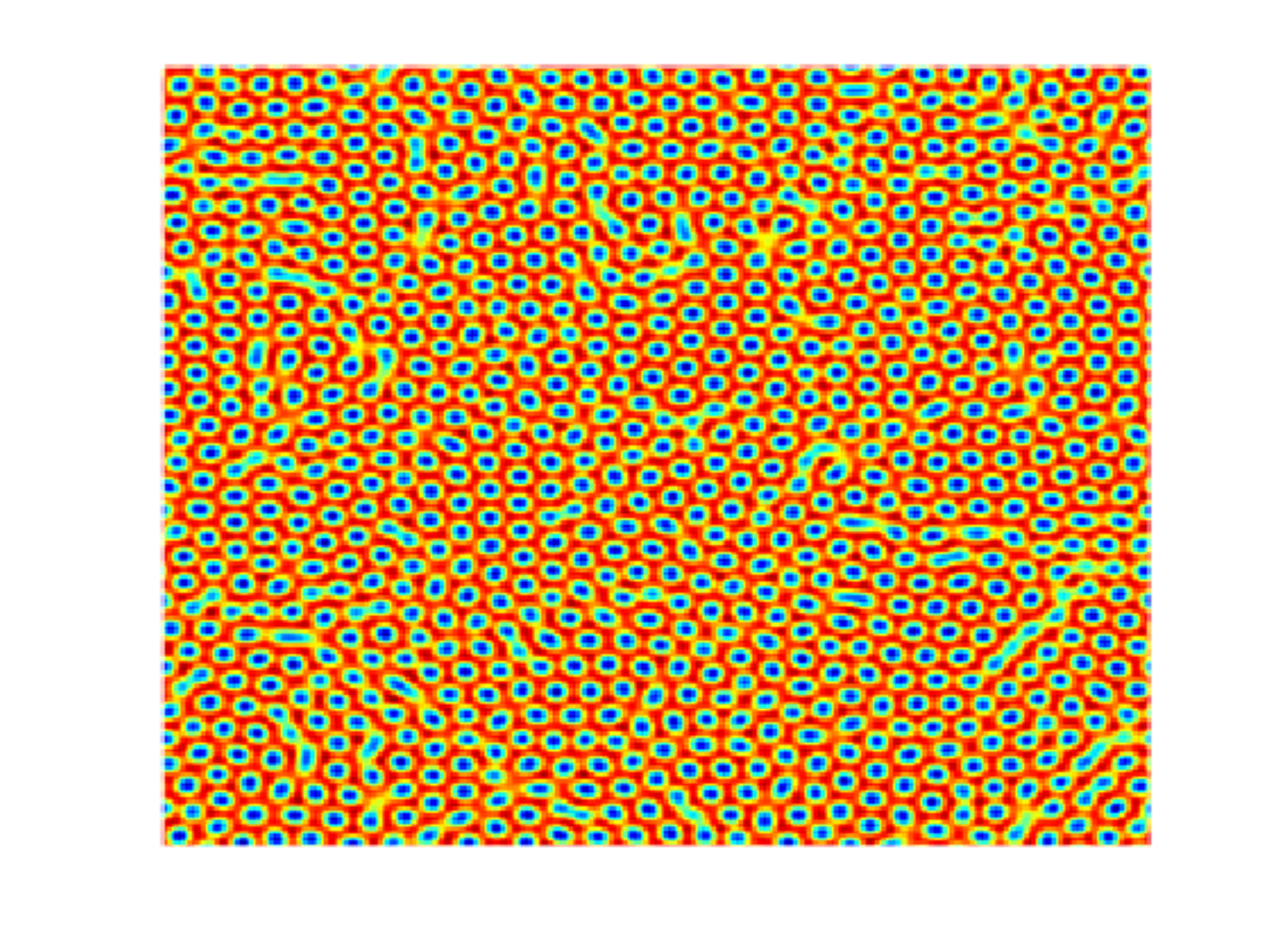}
}
\caption{Crystal growth pattern formation with different initial conditions governed by the PFC model.}\label{fig:fig7}
\end{figure}
\begin{figure}[htp]
\centering
\subfigure[Energy evolution with $\Omega=\text{[}0,100{]}^2$ and $\widehat{\phi}_0=0.1$]{
\includegraphics[width=5cm,height=5cm]{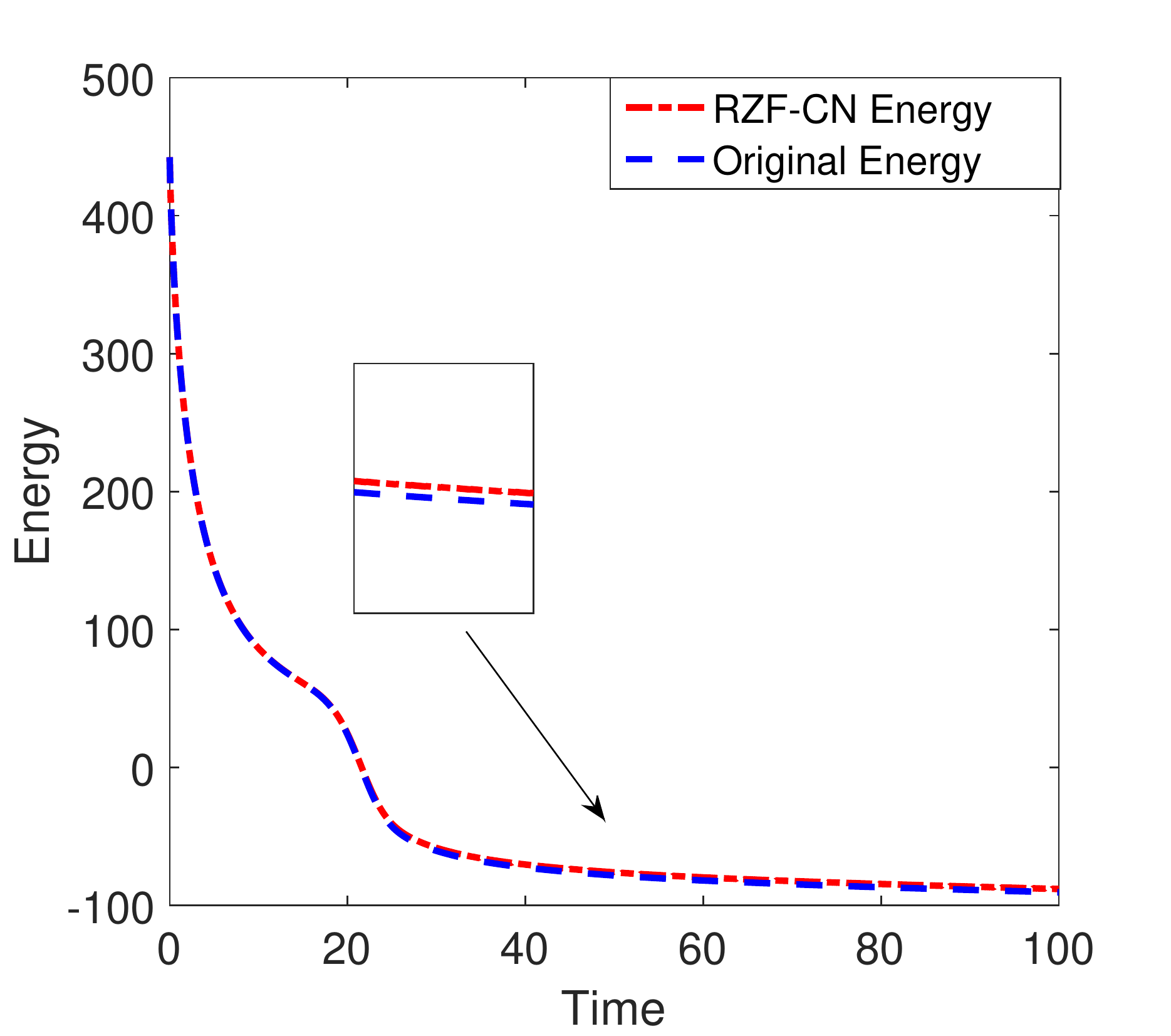}}
\subfigure[Energy evolution with $\Omega=\text{[}0,200{]}^2$ and $\widehat{\phi}_0=0$]{
\includegraphics[width=5cm,height=5cm]{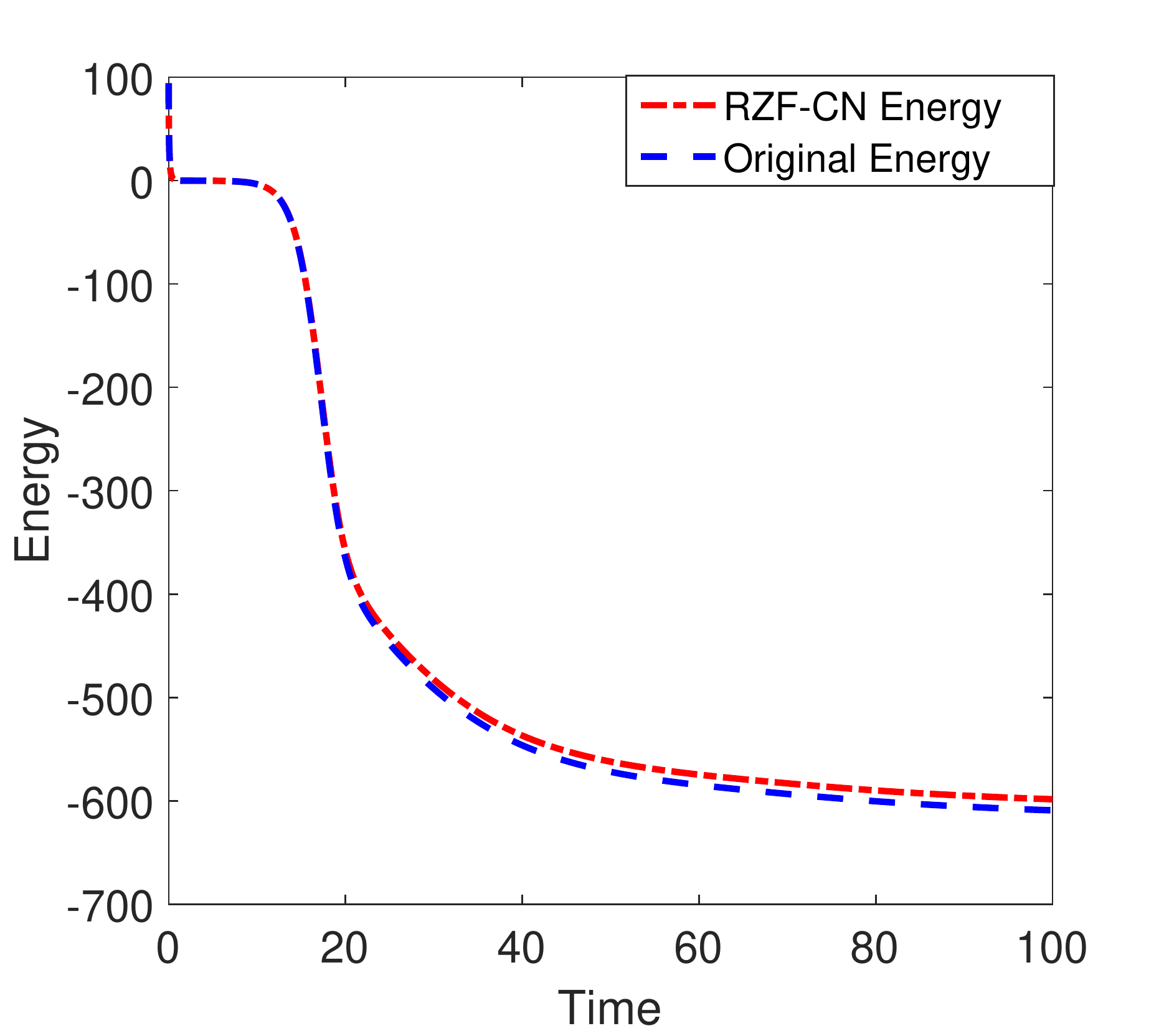}}
\subfigure[Energy evolution with $\Omega=\text{[}0,200{]}^2$ and $\widehat{\phi}_0=0.2$]{
\includegraphics[width=5cm,height=5cm]{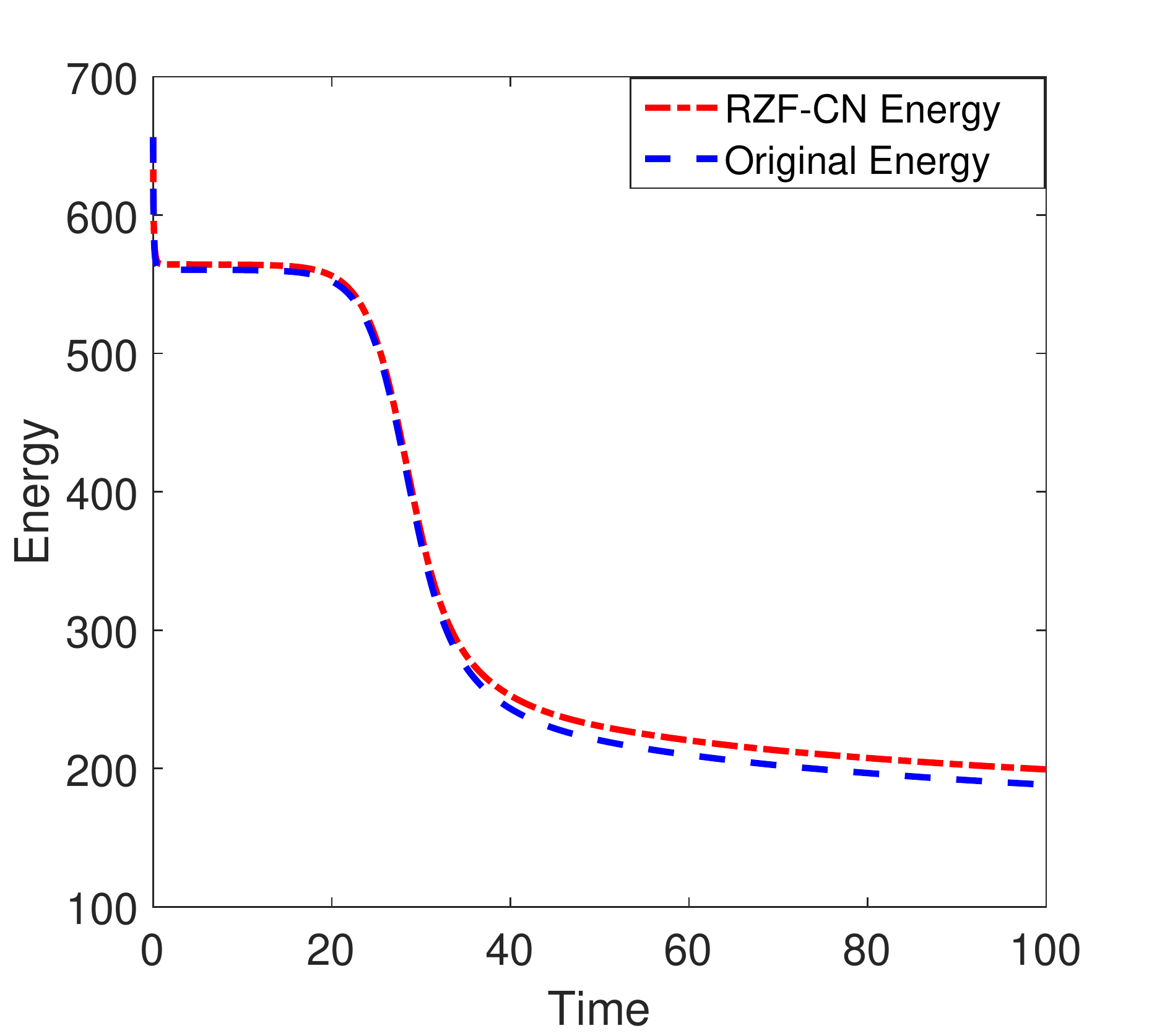}}
\caption{(a) Energy of $\mathcal{\widetilde{E}}(\phi)$ and $E(\phi)$ with $\Omega=[0,100]^2$ and $\widehat{\phi}_0=0.1$, (b) Energy of $\mathcal{\widetilde{E}}(\phi)$ and $E(\phi)$ with $\Omega=[0,200]^2$ and $\widehat{\phi}_0=0$, (c) Energy of $\mathcal{\widetilde{E}}(\phi)$ and $E(\phi)$ with $\Omega=[0,200]^2$ and $\widehat{\phi}_0=0.2$.}\label{fig:fig8}
\end{figure}
\section*{Acknowledgement}
No potential conflict of interest was reported by the author. We would like to acknowledge the assistance of volunteers in putting together this example manuscript and supplement.
\bibliographystyle{siamplain}
\bibliography{Reference}

\end{document}